\documentclass[a4paper,10pt]{article}

\usepackage{ucs}
\usepackage[utf8]{inputenc}

\usepackage{graphicx}
\usepackage{amsfonts,dsfont}
\usepackage{dsfont}
\usepackage{amssymb}
\usepackage{amsmath}
\usepackage{amsthm}
\usepackage{enumerate}
\usepackage{stmaryrd}
\usepackage{fullpage}
\usepackage{ifthen}
\usepackage{subfigure}
\usepackage{epic}
\usepackage{authblk}
\usepackage{textcomp}
\usepackage{mathrsfs}
\usepackage{bm}
\usepackage[small]{caption}
\usepackage{tikz,tkz-tab}
\usetikzlibrary{matrix}
\usepackage{pgfplots}
\pgfplotsset{compat=newest} 
\pgfplotsset{plot coordinates/math parser=false}

\usepackage[hypertexnames=false,colorlinks=true,linkcolor=black,citecolor=blue]{hyperref}
\usepackage[numbers,comma,square,sort&compress]{natbib}
\usepackage{geometry}

\usepackage{color}
\usepackage{titlesec}
\usepackage{mathrsfs}

\graphicspath{{eps/}{pdf/}{png/}}

\def\varep{\varepsilon}

\newcommand{\bqq}{\begin{equation}}
\newcommand{\eqq}{\end{equation}}
\newcommand{\bqs}{\begin{equation*}}
\newcommand{\eqs}{\end{equation*}}

\renewcommand{\Re}{\textnormal{Re}} 
\renewcommand{\Im}{\textnormal{Im}} 
\newcommand{\R}{\mathbb{R}} 
\newcommand{\N}{\mathbb{N}} 
\newcommand{\I}{\mathcal{I}}

\renewcommand{\S}{\mathcal{S}}

\newcommand{\K}{\mathcal{K}}

\newcommand{\md}{\mathrm{d}}

\newtheorem{lem}{Lemma}[section]
\newtheorem{thm}{Theorem}[section]
\newtheorem{prop}{Proposition}[section]

\newtheorem{rmk}[lem]{Remark}
\newtheorem{defi}[lem]{Definition}

\numberwithin{equation}{section}

\title{\bf Sharp asymptotics for a transport model with a nonlocal condition of the Fisher-KPP type at the boundary}

\author{Gr\'egory Faye$^{a}$, Jean-Michel Roquejoffre$^{a}$~and~Mingmin Zhang$^{a,b}$
\thanks{This work has been supported by the French Agence Nationale de la Recherche
		(ANR), in the framework of Indyana (ANR- 21-CE40-0008) and ReaCh (ANR-23-CE40-0023-01)
		projects.  M.Z. also acknowledges the support  by the Occitanie region, the European Regional Development Fund
		(ERDF), and the French government, through the France 2030 project managed by the National
		Research Agency (ANR) with the reference number ``ANR-22-EXES-0015'', and by the Excellent Young Talents Program (Overseas) of the National Natural Science Foundation of China. Email addresses: gregory.faye@math.univ-toulouse.fr; jean-michel.roquejoffre@math.univ-toulouse.fr; mingmin.zhang.math@gmail.com.
		}}

\date{}

\begin{document}
\maketitle

\vspace{-5mm} 

\begin{center}
	\small
	$^a$ Univ Toulouse, CNRS, Institut de Math\'ematiques de Toulouse, Toulouse, France\\[1mm]
	$^b$ School of Mathematical Sciences, University of Science and Technology of China, Hefei, Anhui 230026,  China
\end{center}

\begin{abstract}
This paper is concerned with the precise asymptotics, as time goes to infinity,  of a transport problem in a half plane coupled with a nonlinear nonlocal boundary condition. This system arises from a class of models for the spatial spread of epdemics, its space independent version being the classical Kermack-McKendrick model.  Using ideas pertaining to the study of nonlocal equations of the Fisher-KPP type, and exploiting the particular structure of the model, we prove that any initially localized solution will lag behind the minimal traveling wave, with a delay that grows logarithmically in time.

\end{abstract}

\hspace{0.6em} {\small {\bf Keywords:}  Sharp asymptotics; transport problem; nonlocal spatial interactions; logarithmic delay.}

\vskip 3mm

\section{Introduction and main results}

We address the sharp asymptotics, as $t$ tends to infinity, of the level sets for the following  problem with unknown function $\rho(t,i,x)$, $t>0$, $i\in[0,i_\dagger)$ and $x\in\R$:
\bqq
\label{fkpp}
\left\{
\begin{split}
	&\partial_t \rho(t,i,x) + \partial_i\rho(t,i,x) = - \gamma(i) \rho(t,i,x),~~~~\quad t>0, ~ i\in(0,i_\dagger),  ~ x\in\R,  \\
	&\rho(t,0,x) =\S_0 \bigg(1-\exp\Big(-\int_0^\infty \tau(i) \K*\rho(t,i,x)\md i \Big)\bigg),~ \quad  t> 0,  ~ x\in\R,
\end{split}\right.
\eqq
associated with compactly supported initial condition $\rho_0(i,x)$.  The model under study is therefore a linear transport equation in the upper half plane, together with a nonlinear and nonlocal Dirichlet condition that couples the values of the unknown function at the boundary to those inside.  

Its primary motivation is the study of the nonlocal Kermack-McKendrick model \cite{KMcK} for the spread of epidemics  in our previous work \cite{FRZ2023}, and it presents specific issues calling for a mathematical study of its own. In the context of modelling, $\rho(t,i,x)$  represents the cumulative density of infected individuals at time $t>0$ with elapsed time since infection $i\in(0,i_\dagger)$ and at position $x\in\R$,  $\gamma(i)$ denotes the recovery rate, $\tau(i)$ models the infection rate within a spatially homogeneous population, while the integral kernel $\K\ge 0$ describes nonlocal spatial interactions of individuals.

The large time behavior of the solutions of model \eqref{fkpp} has already been addressed at the end of the 1970's.  In this introduction, we just mention the pioneering works of Diekmann \cite{Diekmann1}, \cite{Diekmann2} and Thieme \cite{Thieme1}, \cite{Thieme2}, which deal with systems that are closely related to \eqref{fkpp}. To sum up their contributions, they prove that the solutions have the spreading property, that is, each level set of a solution emanating from a compactly supported initial datum will advance at the minimal velocity of a linear wave.  The main tool is the reduction of \eqref{fkpp} to an integral equation admitting a form of comparison principle, which allows a precise computation of the spreading speed. They also study  traveling waves, which reveals the same intrinsic essential  dynamical mechanism  as the KPP type reaction-diffusion equations, in reference to the fundamental work of Kolmogorov-Petrovsky-Piskunov \cite{KPP37}. 
These works have been generalized in many directions,  and have in particular given birth to the theory of monotone systems. As the present work goes in another direction, we will not attempt here to present an extensive bibliography of this rich subject.

Much less is understood  about how the precise asymptotics of \eqref{fkpp} as $t\to+\infty$  is influenced by the piling-on effect of the transport equation and the nonlocal nonlinear boundary condition. While the sharp asymptotics issue has fostered an important activity for the Fisher-KPP equation and related models, something that will be accounted for later in this introduction, such explorations for hyperbolic type models, additionally coupled with nonlocal boundary conditions, have so far not been undertaken. To what extent the behavior of  this hyperbolic model resembles that of the KPP type  reaction-diffusion equations  is the question that we are going to tackle in this paper.

The observation at the basis  of our work is that the structure of model \eqref{fkpp} is reminiscent of that of the ``Road-field model''  introduced by  the second author together with Berestycki and Rossi \cite{BRR13} to model biological invasions directed by lines of fast diffusions. In \cite{BRR13}, a line, having a diffusion of its own, exchanges individuals with a neighboring half-spece, where a reaction-diffusion process occurs. It is shown that, in the end, the line dictates the global behavior of the system. This observation has already triggered a new look at the long time behavior of \eqref{fkpp},  initiated in \cite{FRZ2023}. There, we not only reprove the former results of Diekmann and Thieme, but we add new elements of understanding to the traveling wave solutions. The goal of the present paper is to exploit this observation further, in order to understand the sharp asymptotics of \eqref{fkpp}. One important advantage over the integral formulation -- that we nevertheless exploit repeatedly in our work -- of this viewpoint is that \eqref{fkpp} appears as an initial value problem, which allows comparisons of solutions starting from arbitrary large times. 

Throughout this paper, we assume that $0<i_\dagger<+\infty$ and that 
\begin{itemize}
\item[(i)] $\gamma:[0,i_\dagger)\mapsto\R_+$ is a nonnegative function such that $\gamma(i)\in L^1_{loc}([0,i_\dagger))$, $\int_0^{i_\dagger}\gamma(s)\md s=+\infty$. We set
\begin{equation*}
	\pi(i):=e^{-\int_0^i\gamma(s)\md s},~~i\in[0,i_\dagger].
\end{equation*}
\item[(ii)] $\tau:[0,i_\dagger]\mapsto\R_+$ is a nontrivial, bounded and absolutely continuous function such that  $0\leq\tau(i)\leq \tau_\infty$ for all $i\in[0,i_\dagger]$ with some $\tau_\infty>0$.

\item[(iii)] Set $\omega(i):=\tau(i)\pi(i)$ for $i\in[0,i_\dagger]$.  From (i)-(ii), $\omega$ is absolutely continuous and bounded on $[0,i_\dagger]$.

\item[(iv)] The kernel $\K\in W^{1,1}(\R)$ is even, nonnegative and compactly supported such that  $\int_{\R}\K(x)\md x=1$ and $\text{spt}(\K)=[-R,R]$. 

\item[(v)]  Define
$
	\displaystyle\mathscr{R}_0:=\S_0 \int_0^{i_\dagger} \omega(i)\md i = \S_0 \int_0^{i_\dagger} \tau(i) e^{-\int_0^i\gamma(s)\md s}\md i,
$ 
 and assume $\mathscr{R}_0>1$.
This ensures, based on  \cite{FRZ2023} that spreading actually occurs.
\item[(vi)] The initial function $\rho_0$ is   nonnegative, absolutely
continuous and compactly supported  in $[0,i_\dagger)\times\R$ such that $\rho_0/\pi\in L^\infty([0,i_\dagger)\times\R)$ and such that $\rho_0(i,\cdot)/\pi(i)\in H^1_x(\R)$ uniformly in $i\in[0,i_\dagger)$. 

\item[(vii)] 
There exists $\tau_0>0$ such that
\begin{equation}
	\label{cdn-strict positiv}
	[0,\tau_0]\subset\text{spt}(\tau)\cap\{i\in[0,i_\dagger)|\rho_0(i,x)>0~\text{for some}~x\in\R\}. 
\end{equation} 
\end{itemize}

\begin{rmk}\label{rk1.1}
\textnormal{
	Condition (vii) guarantees the strict positivity of the solution for $t>i$.
 The function $\omega$ is often understood as a well-defined function in $\R_+$ with zero extension outside $[0,i_\dagger)$. 	Compared with the hypotheses in \cite{FRZ2023},  additional restriction  here is $i_\dagger<+\infty$ which  helps to clarify the maximum principles.
}
\end{rmk}

 Let us  recall some results for \eqref{fkpp}; see  \cite{FRZ2023} for the proofs. The unique positive steady solution of \eqref{fkpp} is $$\rho^s(i)=\S_0\rho^*\pi(i),$$ 
where $\rho^*\in(0,1)$ is the unique positive constant solution to $v=1-e^{-\mathscr{R}_0 v}$.  Moreover, the  {\it  dispersion relation} 
\begin{equation}
	\	\label{dr}
	D_c(\alpha):=	\S_0\int_{\R}\K(x)e^{\alpha x} \md x \int_0^{i_\dagger} \omega(i) e ^{-\alpha ci}\md i-1=0,~~~~~c>0,~\alpha>0,
\end{equation}
admits a unique solution pair $(c(\alpha),\alpha)$ for each $\alpha\in(0,+\infty)$.
In particular, by defining $c_*$ as 
\begin{equation*}\label{c*-def}
	c_*:=\underset{\alpha\in(0,+\infty)}{\inf} c(\alpha)>0,
\end{equation*}
there is the  unique value $\alpha_*\in(0,+\infty)$ such that above infimum is attained. 
On the one hand, such a $c_*$ is the {\it asymptotic spreading speed}, in the sense that the solution $\rho$ to the transport problem \eqref{fkpp}, with a  nonnegative, absolutely
continuous and compactly supported  initial datum $\rho_0$ on $[0,i_\dagger)\times\R$ such that $\rho_0/\pi\in L^\infty([0,i_\dagger)\times\R)$, satisfies
\bqs
\forall c\in(0,c_*),~\forall j\in(0,i_\dagger),~~~~~~~\underset{t\rightarrow+\infty}{\lim } \underset{|x|\leq ct,~ 0 \leq i \leq j}{\sup} \left| \rho(t,i,x)-\rho^s(i)\right| =0.
\eqs
and
\bqs
\forall c>c_*,~\forall j\in(0,i_\dagger),~~~~~~~\underset{t\rightarrow+\infty}{\lim } \underset{|x|\geq ct,~ 0 \leq i \leq j}{\sup} \rho(t,i,x) =0.
\eqs
On the other hand,  $c_*$ is also the minimal wave speed, in the sense that \eqref{fkpp} admits traveling wave solutions
\begin{equation}
	\label{TW_original}
	U(t,i,x)=\S_0\chi_{c}\big(x-c(t-i)\big)\pi(i),~~~~~~~(t,i,x)\in\R\times[0,i_\dagger)\times\R,
\end{equation}  
 if and only if $c\ge c_*$, where   $\chi_{c}(z)$ satisfies
\begin{equation}\label{chi_TW}
	\left\{
	\begin{split}
		&	\chi_{c}(z)=1-\exp\Big(-\S_0\int_0^{i_\dagger} \omega(i) \K*\chi_{c}(z+ci)\md i\Big),~~~z\in\R,\\ 
		&\chi_c(-\infty)=\rho^*,~~\chi_c(+\infty)=0,~~0<\chi_c <\rho^*~~~\text{in}~\R.
	\end{split}\right.
\end{equation}
Moreover, $\chi_{c}$ has the asymptotics\footnote{Throughout this paper, we use the convention $f(t)\approx g(t)$ as $t\to+\infty$ if $\lim_{t\to+\infty}f(t)/g(t)=1$; while  we use $f(t)\sim(\lesssim,\gtrsim) g(t)$ as $t\to+\infty$ if $\lim_{t\to+\infty}f(t)/g(t)=(\le,\ge)C$  for some universal constant $C$.}  (up to normalization)
\begin{align}\label{normalization of TW}
	\chi_c(z)\approx e^{-\alpha_c z}~~~(\text{for}~c>c_*),~~~~\chi_{c_*}(z)\approx ze^{-\alpha_* z}\quad \text{ as } z\rightarrow+\infty,
\end{align}
in which $\alpha_c\in(0,\alpha_*)$ is the unique value such that $D_c(\alpha_c)=1$ for each $c> c_*$.

\vskip 2mm

We are concerned here with  the precise asymptotic position of the level sets, up to $O(1)$ error, something that goes one step beyond the spreading property.  Before stating our main result, let us give a brief account of what is already known for the Fisher-KPP equation
\begin{equation}
	\label{KPP-nonlinearity}
	u_t=u_{xx}+f(u),~~~~~t>0,~x\in\R,
\end{equation}
and its variants. Here,  $f$ is a $C^1$ function such that  $f(0)=f(1)=0$,  $f'(1)<0$ and   $0<f(s)\le f'(0)s$ for all $s\in(0,1)$.
 Equation \eqref{KPP-nonlinearity} admits traveling wave solutions  $u(t,x)=\varphi_\nu(x-\nu t)$ if and only if $\nu\ge \nu_*=2\sqrt{f'(0)}$, and the wave profile satisfies $\varphi_\nu''+\nu \varphi_\nu'+f(\varphi_\nu)=0$ in $\R$, $\varphi_\nu(-\infty)=1$, $\varphi_\nu(+\infty)=0$ and $0< \varphi_\nu<1$ in $\R$. Moreover, the leading edge behaviors when $x\to +\infty$ of the profiles $\varphi_\nu$  obey: $\varphi_\nu(x)\sim e^{-\lambda_\nu x}$ if $\nu>\nu_*$, whereas $\varphi_{\nu_*}(x)\sim xe^{-\lambda_{\nu_*} x}$ if $\nu=\nu_*$, where $\lambda_\nu>0$ is the smallest value such that $\lambda^2-\nu\lambda+f'(0)=0$ for each $\nu\ge \nu_*$.

Define for  each $t>0$ the leading edge of the level set of the solution $u(t,\cdot)$ to \eqref{KPP-nonlinearity} associated with   Heaviside type initial data  as $s(t)=\sup\{x\in\R~|~u(t,x)=m\}$ for any $m\in(0,1)$. The description  of the position for $s(t)$ was first discussed by Kolmogorov-Petrovsky-Piskunov \cite{KPP37} for $f(u)=u(1-u)^2$, showing that $s(t)=c_*t+o_{t\to+\infty}(t)$ with $c_*=2$. Later, a remarkable result was proved  by  {Bramson} \cite{Bram2},  who showed by analyzing the representation of the solutions of the Fisher-KPP equation via the Feynman-Kac formula that
\begin{equation}
	\label{s}
	s(t)=c_*t-\frac{3}{2\lambda_*}\ln t+\sigma_\infty+o_{t\to+\infty}(1),
\end{equation}
with $c_*=2$ and $\lambda_*=1$. The presence of this logarithmic lag has since been identified in  wide classes of reaction-diffusion models arising in various domains of physics and chemistry. The works \cite{EvS00}, \cite{EvSP} have uncovered the universality of this feature and have proposed an explanation of the phenomenon, mostly in the formal style of matched asymptotics.
A weaker version of  \eqref{s} via a new way of thinking with purely PDE arguments  was provided in  \cite{HNRR13}, where it is shown that the solutions of the KPP equation are closely related to
those of the associated linearized problem with a Dirichlet moving boundary. This turns out to be the key information in catching precisely the asymptotic behavior of the solution in the diffusive regime $|x-c_*t|= O(\sqrt{t})$. This observation has led to a full PDE proof of \eqref{s}  \cite{NRR17}, and has allowed \cite{HNRR16}  the retrieval of the logarithmic lag, with $O(1)$ error, to spatially periodic case with $f(x,u)$ instead of $f(u)$, something that had not been done before. The presence of the logarithmic lag, still with $O(1)$ precision, has since then been identified in equations with no comparison principle, such as in KPP equations with nonlocal nonlinearites \cite{BHR20}, or models of higher order \cite{AS22}, furthering the initial insight of \cite{HNRR13}, and confirming the formal predictions of \cite{EvS00}, \cite{EvSP}.
The intrinsic connection between some classes of KPP equations and  branching Brownian motion/branching random walk have allowed extensions to classes of nonlocal or discrete equations. The most precise result so far is  \cite{A2013}, which  proves \eqref{s} to $o(1)$ precision in this context.  The contribution  \cite{Graham} proposes a study of the nonlocal Fisher-KPP equation with $O(1)$ precision  for a large class of diffusion kernels, with proofs mixing PDE and probabilistic arguments.  This result is made more precise in \cite{Roquejoffre2022} with a new  PDE  proof. The idea developed there is at the basis of the study  \cite{BFRZ2023} for  KPP equations on one-dimensional lattices. We finally refer to  \cite{Gartner1982}, \cite{Ducrot}, \cite{RRR2019} for related results on the multi-dimensional homogenous Fisher-KPP equation.

To the best of our knowledge, our work is the first mathematical exploration of sharp asymptotics for this class of hyperbolic  type Fisher-KPP problems, in the sense that we are not aware of any formal/numerical prediction or  probabilistic exploration of the model. 

We transform \eqref{fkpp} by setting, as in our previous work  \cite{FRZ2023}, $\varrho(t,i,x):=\frac{\rho(t,i,x)}{\pi(i)}$, so as to obtain:\begin{equation}
	\label{fkpp-auxi}
	\left\{\begin{split}
		&	\partial_t \varrho(t,i,x) + \partial_i\varrho(t,i,x) =0,~~~~~~~~~~~~~~~~~~~~\quad t>0, \quad i\in(0,i_\dagger), \quad x\in\R, \\
		&	\varrho(t,0,x)=\S_0 \left(1-\exp\Big(-\int_0^{i_\dagger} \omega(i) \K*\varrho(t,i,x)\md i \Big)\right), ~~\quad t> 0,  \quad x\in\R,
	\end{split}\right.
\end{equation}
associated with $\varrho_0(i,x):= \displaystyle\frac{\rho_0(i,x)}{\pi(i)}$.

Define  the level set of $\varrho(t,i_0,x)$ with any $i_0\in[0,i_\dagger)$ as
\begin{equation*}
	E_m(t):=\sup\{x\in\R|\varrho(t,i_0,x)=m\},~~~~\text{for}~~t~~\text{large},
\end{equation*}
with any $m\in(0,\S_0\rho^*)$. Our main result is then that any initially localized solution $\varrho$  will lag behind the minimal traveling waves and the delay grows logarithmically in time.

\begin{thm}
	\label{thm-1'}
Let $\varrho$ be the solution  to problem \eqref{fkpp-auxi} with  nonnegative, absolutely
	continuous, bounded and compactly supported initial datum $\varrho_0$  in $[0,i_\dagger)\times\R$  such that $\varrho_0(i,\cdot)\in H^1_x(\R)$ uniformly in $i\in[0,i_\dagger)$. Then  
	\begin{equation*}
		E_m(t)=x-c_*t+\frac{3}{2\alpha_*}\ln t+O_{t\to+\infty}(1).
	\end{equation*} 
\end{thm}

As a consequence we have the
 \begin{thm}
 	\label{thm-1}
 	Let $\rho$ be the solution to problem \eqref{fkpp}. Then
 	for any $\varep>0$, there exist $T(\varep)>i_\dagger$ and $L(\varep)>0$ such that
 	\begin{equation*}
 		\label{thm1-log delay}
 		\rho(t,i,x)> \big(\S_0\rho^*-\varep\big)\pi(i)~~~~\text{for all}~t>T(\varep),~i\in[0,i_\dagger),~x\in\Big[0,c_*t-\frac{3}{2\alpha_*}\ln t-L(\varep)\Big]
 	\end{equation*}
 	and 
 	\begin{equation*}
 		\rho(t,i,x)< \varep\pi(i)~~~~\text{for all}~t>T(\varep),~i\in[0,i_\dagger),~x\in\Big[c_*t-\frac{3}{2\alpha_*}\ln t+L(\varep),+\infty\Big).
 	\end{equation*}
 \end{thm}

Our starting point is the knowledge from our work \cite{FRZ2023} that model \eqref{fkpp} performs KPP type dynamic properties, and the general two-step PDE argument that we developed in \cite{Roquejoffre2022}, \cite{BFRZ2023} to investigate KPP equations that are discrete or of the nonlocal type. The latter will turn out to be robust enough for hyperbolic evolution problem \eqref{fkpp} considered here, but will not
come without  issues that are specific to the problem. One of them is comparison: while the equations look simple, we have to establish new versions of comparison principles in moving domains with one or two moving boundaries, where, unlike in diffusion equations, the two boundaries do not play the same role.  Once this is under control, a rather serious investigation of the linearized problem with nonlocal spatial effect coming from the boundary $i=0$ is in order. We first have to estimate the heat kernel in the whole domain; in contrast with the Fourier analysis in  \cite{Roquejoffre2022} of nonlocal diffusive KPP equations, we need to combine it with Laplace transform arguments, as the convolutions are not only in space but also in time. 
	However, the quantity of interest is not the heat kernel in the whole domain, but the solution of a nonlinear problem whose solution turns out to look like that of  a linear  Dirichlet boundary value problem set around $x\approx c_*t$ within the diffusive regime. The estimates   are once again specific to the hyperbolic/nonlocal nature of our problem.
	 This Dirichlet solution will need to be sufficiently precise so as to allow the construction of barrier functions tailored for the nonlinear problem, eventually leading us to conclude that  the position of any initially localized solution for the nonlinear transport problem lags behind that of the minimal traveling waves, and this delay grows as $3/(2\alpha_*)\ln t$. 

\begin{rmk}
\textnormal{
 We believe that we can retrieve the large time behavior of the level sets of  $\rho(t,i,x)$ up to $o(1)$ precision. This would entail the analysis of the heat kernel for well spread initial data --  that is, initial data that vary like $x/\sqrt\varepsilon$, where $\varepsilon$ is dictated by  what happened to the solution at former times, as well as a detailed analysis of the barrier functions in that scaling. This is not beyond reach, but we do not perform it here: we want to keep the length of the paper under control. We also believe that, similarly to what is done in \cite{Roquejoffre2022}, we can go back to the original unknowns $S(t,x)$ and $I(t,i,x)$ of the nonlocal Kermack-McKendrick model, and prove that they converge, for large times, to a travelling wave still shifted by the logarithmic term.}
\end{rmk}

\vspace{2mm}
 
\noindent
{\bf Organization of the paper.}  The first issue concerns the set-up of super- and subsolutions, and comparison principles with moving boundaries, all of which will be stated in  Section \ref{sec2}. Section \ref{sec: heat kernel estimate}  is dedicated to a heat kernel asymptotics for the renewal equation \eqref{renewal eqn=3} within the diffusive scale. Section \ref{sec4} is devoted to the approximation of the Dirichlet heat kernel, as well as new type of estimates outside the diffusive zone. In  Section \ref{sec5}, we construct a pair of super- and subsolutions for the nonlinear transport problem \eqref{moving frame-nonlinear} ahead of $x\approx c_*t$. Finally, we prove Theorem \ref{thm-1'} in Section \ref{sec6}, which directly implies  Theorem \ref{thm-1}.
Throughout this paper, we will often write $C$ for a uniform positive constant
that may change from line to line.

\vspace{2mm}

\section{Framework and set-up of comparison principles}
\label{sec2}

By the  method of characteristics, the solution to \eqref{fkpp-auxi} has the form
\begin{align*}
	\varrho(t,i,x)=
	\begin{cases}
		\varrho_0(i-t,x),  &t\le i,~x\in\R,\\
		\Psi(t-i,x), &t>i,~ x\in\R,
	\end{cases}
\end{align*}
with 
\begin{equation*}
	~~~~~~	\Psi(t,x):=\S_0 \left(1-\exp\Big(-\int_0^{i_\dagger} \omega(i) \K*\varrho(t,i,x)\md i \Big)\right),~~~~~~~~~~t>0,~x\in\R.
\end{equation*}

One reason why comparison is much less easy in \eqref{fkpp}, or equivalently \eqref{fkpp-auxi}, than in diffusion equations, is that the region of $t\le i$, where the initial condition is transported along the characteristic line, has to be treated differently from the region of $t>i$. Let us first  recall the definition of super- and subsolutions for \eqref{fkpp-auxi} as well as the comparison principle  Proposition \ref{prop_cp_R} for which we provide a new proof in the Appendix.

\begin{defi}[\cite{FRZ2023}]\label{def_super+sub_R}
	We say that a  bounded and continuous function $\overline \varrho(t,i,x)$ defined for  $(t,i,x)\in\R_+\times[0,i_\dagger)\times\R$  is a
	supersolution of  \eqref{fkpp-auxi} for $(t,i,x)\in\R_+\times[0,i_\dagger)\times\R$, 
	if $\partial_t \overline \varrho(t,i,x)$ and $\partial_i\overline \varrho(t,i,x)$ exist a.e. and satisfy $\partial_t \overline \varrho(t,i,x)  + \partial_i \overline \varrho(t,i,x) \ge 0$ a.e. for $(t,i,x)\in\R_+\times[0,i_\dagger)\times\R$,  as well as the second equation of \eqref{fkpp-auxi} with ``$=$'' replaced by ``$\ge$''   for $t>0$ and $x\in\R$. A subsolution $\underline\varrho$ can be defined in a similar way with both inequality signs above being reversed.
\end{defi}

\begin{prop}[Comparison principle \cite{FRZ2023}]\label{prop_cp_R}
	Let $\overline \varrho$ and $\underline \varrho$ be  bounded and continuous functions  defined for $(t,i,x)\in\R_+\times[0,i_\dagger)\times\R$ such that  $\overline \varrho$ is a nonnegative supersolution and $\underline \varrho$ is a subsolution of \eqref{fkpp-auxi} for $(t,i,x)\in\R_+\times[0,i_\dagger)\times\R$ in the sense of Definition \ref{def_super+sub_R}. Assume that $\overline \varrho (0,i,x)\ge \underline \varrho(0,i,x)$ for $i\in [0,i_\dagger)$ and $x\in\R$, then $\overline \varrho(t,i,x)\ge \underline \varrho(t,i,x)$ for $(t,i,x)\in\R_+\times[0,i_\dagger)\times\R$.  If further 
	 \eqref{cdn-strict positiv} holds true with $\rho_0$ replaced by $(\overline\varrho-\underline\varrho)(0,\cdot,\cdot)$, then $\overline \varrho(t,i,x)> \underline \varrho(t,i,x)$ for $t>i$, $i\in[0,i_\dagger)$ and $x\in\R$.
\end{prop}

\subsubsection*{The reduced equation in the zone $\{x\approx c_*t\}$ and its linerization}

One specific feature of the problem is that the appropriate space variable to consider is $x-c_*(t-i)$, something that does not always interact well with the boundary condition.   So, we set
\begin{equation}\label{def_v}
	\varrho(t,i,x)=v(t,i,x)e^{-\alpha_*(x-c_*(t-i))},~~~~~~~~t>0, \quad  i\in[0,i_\dagger), \quad x\in\R,
\end{equation}
then $v$ satisfies
\begin{equation}
	\label{moving frame-nonlinear}
	\begin{aligned}
\!\!\!\!\!\left\{\begin{split}
			&\partial_t v(t,i,x)  + \partial_i v(t,i,x) =0,~~~~~~~~~~~~~~~~~~~~~~~~~~~~~~~~~~~~~~~~~~~~~~~~~~~~~~~~ t>0,~ i\in(0,i_\dagger), ~ x\in\R, \\
		&	v(t,0,x)=e^{\alpha_*(x-c_*t)}\S_0\!\left(\!1\!-\exp\!\Big(\!\!-\! e^{-\alpha_*(x-c_*t)}\! \int_0^{i_\dagger}\! \omega(i) e^{-\alpha_*c_* i} \K_* *v(t,i,x)\md i\Big)\!\!\right)\!,~~ ~t> 0,~ x\in\R,\\
		&	v(0,i,x)= \varrho_0(i,x)e^{\alpha_*x+\alpha_*c_* i}=:v_0(i,x),~~~~~~~~~~~~~~~~~~~~~~~~~~~~~~~~~~~~~~~~~~~~~~~~~  i\in[0,i_\dagger),~ x\in\R,
	\end{split}\right.
	\end{aligned}
\end{equation}
with $\K_*(x)=\K(x)e^{\alpha_* x}$ for $x\in\R$.

Learning from our previous work \cite{FRZ2023} that there is the KPP type mechanism behind this transport problem which is particularly amenable to linearized analysis, we therefore
begin with the linearization of \eqref{moving frame-nonlinear} around the unstable state zero:
\begin{equation}
	\label{moving frame-u}
	\left\{\begin{split}
		&	\partial_t u(t,i,x)  + \partial_i u(t,i,x) =0,~~~~~~~~~~ t>0, ~ i\in(0,i_\dagger), ~ x\in\R, \\
		&	u(t,0,x)=\S_0 \int_0^{i_\dagger} \omega(i) e^{-\alpha_*c_* i} \K_* *u(t,i,x)\md i, ~ t> 0,  ~ x\in\R,
	\end{split}\right.
\end{equation} 
with $u(0,i,x)= u_0(i,x)$ which is absolutely continuous and compactly supported for $(i,x)\in [0,i_\dagger)\times\R$.

In what follows, we establish new versions of comparison principles in moving domains with one or two moving boundaries, which will fit better with our needs. This preliminary step as already addressed in the introduction is not immediate but rather  delicate, for which  the dynamical mechanism of the solution -  the roles of $t$ and $i$ - should be understood.

\subsubsection*{Super- and subsolutions}
 To start with, let us give the definition of super- and subsolutions for the linear problem  \eqref{moving frame-u} and for the nonlinear problem \eqref{moving frame-nonlinear} in moving domains. 
\begin{defi}
	\label{Def_super+sub} 	For any fixed $t_*\ge 0$ and for any  continuous  function $X(t,i)$ defined for $t\ge i+t_*$ and $i\in[0,i_\dagger)$,
	we say that a  bounded and continuous function $\underline u(t,i,x)$ defined for $t\ge i+t_*$, $i\in[0,i_\dagger)$ and  $x\ge X(t,i)-R$ is a
	subsolution of   \eqref{moving frame-u} for $t\ge i+i_\dagger+t_*$, $i\in[0,i_\dagger)$ and $x\ge X(t,i)$,
	if $\partial_t \underline u(t,i,x)$ and $\partial_i\underline u(t,i,x)$ exist a.e. such that 
		\begin{equation*}
		\begin{aligned}
			\left\{\begin{split}
				&	\partial_t \underline u(t,i,x)  + \partial_i \underline u(t,i,x) \le 0,~~~~~~~ ~~~~~~~~~~~~t\ge i+t_*, ~i\in(0,i_\dagger), ~x\ge X(t,i)-R, \\
				&	\underline u(t,0,x)\le \S_0 \int_0^{i_\dagger} \omega(i) e^{-\alpha_*c_* i} \K_* *\underline u(t,i,x)\md i,  ~~~~~~~~~~~~~~~t\ge  i_\dagger+t_*,   ~x\ge X(t,0).
			\end{split}\right.		
		\end{aligned}
	\end{equation*}
A supersolution $\overline u$ to \eqref{moving frame-u} for $t\ge i+i_\dagger+t_*$, $i\in[0,i_\dagger)$ and $x\ge X(t,i)$ can be defined in the same way with the inequality signs being reversed.

	Analogously, 	for any fixed $t_*>0$ large and for any  continuous and nondecreasing  functions $X^\pm(s)$ defined for $s\ge t_*$ such that $X^-(s)\le X^+(s-i_\dagger)$ for all $s\ge t_*$,  a  bounded and continuous function $\underline u(t,i,x)$ defined for $t\ge i+t_*$, $i\in[0,i_\dagger)$ and  $X^-(t)-R\le x\le X^+(t)+R$ is called a
	subsolution of \eqref{moving frame-u}   for $t\ge i+i_\dagger+t_*$, $i\in[0,i_\dagger)$ and $x\in[X^-(t), X^+(t-i)]$,
	if $\partial_t \underline u(t,i,x)$ and $\partial_i\underline u(t,i,x)$ exist a.e. such that
\begin{equation*}
\begin{aligned}
	\left\{\begin{split}
		&	\partial_t \underline u(t,i,x)  + \partial_i \underline u(t,i,x) \le 0,~~~~~~~~~~~~~~~~~~~~~t\ge i+t_*, ~i\in(0,i_\dagger), ~X^-(t)\le x\le X^+(t-i), \\
		&\underline	u(t,0,x)\le \S_0 \int_0^{i_\dagger} \omega(i) e^{-\alpha_*c_* i} \K_* *\underline u(t,i,x)\md i,  ~~~~~~~~~~~~~~~~t\ge t_*+ i_\dagger,   ~X^-(t)\le x\le X^+(t).
	\end{split}\right.		
\end{aligned}
\end{equation*}
A supersolution $\overline u$ to \eqref{moving frame-u} for $t\ge i+i_\dagger+t_*$, $i\in[0,i_\dagger)$ and $x\in[X^-(t), X^+(t-i)]$ can be defined in the same way with the inequality signs being reversed.
	
	For each case above,  super- and subsolutions for \eqref{moving frame-nonlinear} can be defined in a similar way.
\end{defi}

\subsubsection*{Comparison principles}

  Our comparison principles for the nonlinear problem \eqref{moving frame-nonlinear} with one or two moving boundaries state as follows.  
  \begin{prop}[Comparison principle with one boundary]\label{prop_cp_nonlinear transport}
  	For any fixed $t_*\ge 0$ and for any  continuous  function $X(t,i)$ defined for $t\ge i+t_*$ and $i\in[0,i_\dagger)$  satisfying $X(t,i)$ for each $i\in[0,i_\dagger)$ is nondecreasing in $t\in(i+t_*,+\infty)$ and  $X(t,0)\ge \max_{s\in[0,i]}X(t-i+s,s)$,
  	let $\overline v(t,i,x)$ and $\underline v(t,i,x)$ defined for
  	 $t\ge i+t_*$, $i\in[0,i_\dagger)$ and $x\ge X(t,i)-R$
  	 be  respectively a  super- and  a subsolution of \eqref{moving frame-nonlinear} for $t\ge i+i_\dagger+t_*$, $i\in[0,i_\dagger)$ and $x\ge X(t,i)$, in the sense of Definition \ref{Def_super+sub}. Assume   that $\overline v(t,0,x)\ge \underline v(t,0,x)$  for $t\in [ t_*,i_\dagger+t_*]$ and $x\ge X(t,0)$, that $\overline v(t,0,x)\ge \underline v(t,0,x)$  for $t\ge t_*$ and $x\in[X(t,0)-R,X(t,0)]$, and $\overline v(t,i,X(t,i))\ge\underline v(t,i,X(t,i))$ for $t\ge i+i_\dagger+t_*$ and $i\in[0,i_\dagger)$,
  	 then $\overline v(t,i,x)\ge \underline v(t,i,x)$  for $t\ge i+i_\dagger+t_*$, $i\in[0,i_\dagger)$ and $x\ge X(t,i)$. 
  	  \end{prop}

 \begin{prop}[Comparison principle with two boundaries]
 	\label{prop_cp_nonlinear transport2} 
 	For any fixed $t_*>0$ large and for any  continuous and nondecreasing  functions $X^\pm(s)$ defined for $s\ge t_*$ such that $X^-(s)<X^+(s-i_\dagger)$ for all $s\ge t_*$, 	let $\overline v(t,i,x)$ and $\underline v(t,i,x)$ defined for
 	$t\ge i+t_*$, $i\in[0,i_\dagger)$ and $X^-(t)-R\le x\le X^+(t)+R$
 	be  respectively a  super- and  a subsolution of \eqref{moving frame-nonlinear} for $t\ge i+i_\dagger+t_*$, $i\in[0,i_\dagger)$ and $x\in[X^-(t), X^+(t-i)]$, in the sense of Definition \ref{Def_super+sub}.
 	Assume that $\overline v(t,0,x)\ge \underline v(t,0,x)$ for $t\in [ t_*,i_\dagger+t_*]$ and $x\in[X^-(t), X^+(t)]$, and that  $\overline v(t,i,x)\ge\underline v(t,i,x)$ for $t\ge i+t_*$, $i\in[0,i_\dagger)$ and $x\in[X^-(t)-R,X^-(t)]\cup[X^+(t-i),X^+(t)+R]$,  then $\overline v(t,i,x)\ge \underline v(t,i,x)$  for $t\ge i+i_\dagger+t_*$, $i\in[0,i_\dagger)$ and $x\in[X^-(t), X^+(t-i)]$. 
 \end{prop}

\begin{rmk}
\textnormal{	In the application of Proposition \ref{prop_cp_nonlinear transport}, when verifying the boundary conditions at $i=0$ and $x=X(t,i)$, it will be sufficient to have $\overline v(t,i,x)\ge \underline v(t,i,x)$ for $t\ge i+t_*$, $i\in[0,i_\dagger)$ and $x\in[X(t,i)-R,X(t,i)]$.
	The above comparison results will also apply for  generalized super- and subsolutions which are, respectively, given by the minimum of supersolutions and the maximum of subsolutions.}
\end{rmk}

The main step here will be to establish the maximum principles for the linear transport problem \eqref{moving frame-u} with single or two moving boundaries.  The comparison principles for the nonlinear problem \eqref{moving frame-nonlinear} follow immediately by linearization.

\begin{prop}[Maximum principle with one boundary for \eqref{moving frame-u}]
	\label{prop_mp}
	For any fixed $t_*\ge 0$ and for any  continuous  function $X(t,i)$ defined for $t\ge i+t_*$ and $i\in[0,i_\dagger)$  satisfying $X(t,i)$  for each $i\in[0,i_\dagger)$ is nondecreasing in $t\in(i+t_*,+\infty)$, and  $X(t,0)\ge \max_{s\in[0,i]}X(t-i+s,s)$, let $u(t,i,x)$ be a  bounded and continuous  function defined for $t\ge i+t_*$, $i\in[0,i_\dagger)$ and  $x\ge X(t,i)-R$,  such that
	\begin{equation}
		\label{mp-u}
		\begin{aligned}
			\left\{\begin{split}
				&	\partial_t u(t,i,x)  + \partial_i u(t,i,x) \le 0,~~~~~~~ ~~~~~~~~~~~~~~~~~t\ge i+t_*, ~i\in(0,i_\dagger), ~x\ge X(t,i) -R, \\
				&	u(t,0,x)\le \S_0 \int_0^{i_\dagger} \omega(i) e^{-\alpha_*c_* i} \K_* *u(t,i,x)\md i,  ~~~~~~~~~~~~~~~~~~~t\ge  i_\dagger+t_*,   ~x\ge X(t,0).
			\end{split}\right.		
		\end{aligned}
	\end{equation}
	Assume further that $u(t,0,x)\le0$ for $t\in [ t_*,i_\dagger+t_*]$ and $x\ge X(t,0)$, that\footnote{For this boundary assumption, it will be sufficient to assume: $u(t,i,x)\le 0$ for $t\ge i+t_*$, $i\in[0,i_\dagger)$ and $x\in[X(t,i)-R,X(t,i)]$.}  $u(t,0,x)\le 0$ for $t\ge t_*$ and $x\in[X(t,0)-R,X(t,0)]$, and $u(t,i,X(t,i))\le 0$ for $t\ge i+i_\dagger+t_*$ and $i\in[0,i_\dagger)$, then $u(t,i,x)\le 0$  for $t\ge i+i_\dagger+t_*$, $i\in[0,i_\dagger)$ and $x\ge X(t,i)$. 
\end{prop}

\begin{prop}[Maximum principle with two boundaries for \eqref{moving frame-u}]
	\label{prop_cp_linear transport_2}  	For any fixed $t_*>0$ large and for any  continuous and nondecreasing  functions $X^\pm(s)$ defined for $s\ge t_*$ such that $X^-(s)<X^+(s-i_\dagger)$ for all $s\ge t_*$, let $u(t,i,x)$ be a  bounded and continuous  function defined for $t\ge i+t_*$, $i\in[0,i_\dagger)$ and  $X^-(t)-R\le x\le X^+(t)+R$,  such that
	\begin{equation}
		\label{mp-u2}
		\begin{aligned}
			\left\{\begin{split}
				&	\partial_t u(t,i,x)  + \partial_i u(t,i,x) \le 0,~~~~~~~~~~~~~~~~~~~~~t\ge i+t_*, ~i\in(0,i_\dagger), ~X^-(t)\le x\le X^+(t-i), \\
				&	u(t,0,x)\le \S_0 \int_0^{i_\dagger} \omega(i) e^{-\alpha_*c_* i} \K_* *u(t,i,x)\md i,  ~~~~~~~~~~~~~~~~t\ge i_\dagger+t_*,   ~X^-(t)\le x\le X^+(t).
			\end{split}\right.		
		\end{aligned}
	\end{equation}
	Assume further that $u(t,0,x)\le0$ for $t\in [ t_*,i_\dagger+t_*]$ and $x\in[X^-(t), X^+(t)]$ and that  $u(t,i,x)\le 0$ for $t\ge i+t_*$, $i\in[0,i_\dagger)$ and $x\in[X^-(t)-R,X^-(t)]\cup[X^+(t-i),X^+(t)+R]$,  then $u(t,i,x)\le 0$  for $t\ge i+i_\dagger+t_*$, $i\in[0,i_\dagger)$ and $x\in[X^-(t), X^+(t-i)]$. 
\end{prop} 

\begin{proof}[Proof of Proposition \ref{prop_mp}]
	We first claim that 
	\begin{equation}
		\label{2.5-claim}
		u(t,i,x)\le u(t-i,0,x),~~~~~~t\ge i+t_*,~ i\in[0,i_\dagger),~ x\ge \max_{s\in[0,i]}X(t-i+s,s) -R.
	\end{equation}
	To do so, let us define for $t\ge i+t_*$ and $i\in(0,i_\dagger)$,
	\begin{equation}\label{2.5-v:eqn}
		v(s,x):=u(t-i+s,s,x),~~~~~~~~s\in[0,i],~x\ge X(t-i+s,s)-R.
	\end{equation}
	Then we derive from the transport inequality of \eqref{mp-u} that $\partial_s v(s,x)\le 0$ for $s\in(0,i]$, which, together with the continuity of $v$ in $s\in[0,i]$, implies that  for $t\ge i+t_*$ and $i\in(0,i_\dagger)$,
	\begin{equation*}
	~~~~~~	v(i,x)\le v(0,x),~~~~\text{for}~~x\ge \max_{s\in[0,i]} X(t-i+s,s){ -R},
	\end{equation*}
	which leads to our claim \eqref{2.5-claim}.
	
	Next, we shall show that
	\begin{equation}
		\label{2.5-claim2}
		u(t,0,x)\le 0,~~~~~~~~~~t\ge i_\dagger+t_*,~~x\ge X(t,0).
	\end{equation}
 Since $X(t,0)\ge\max_{s\in[0,i]}X(t-i+s,s)$ for $t\ge i+t_*$ and $i\in[0,i_\dagger)$, one then infers from \eqref{mp-u} and  \eqref{2.5-claim} that  for $t\ge i_\dagger+t_*$ and $x\ge X(t,0)$,
	\begin{align*}
		u(t,0,x)&\le \S_0 \int_0^{i_\dagger} \omega(i) e^{-\alpha_*c_* i} \K_* *u(t,i,x)\md i\\
		&=\S_0 \int_0^{i_\dagger} \omega(i) e^{-\alpha_*c_* i} \int_{X(t,0)-R}^{+\infty} \K_*(x-y) u(t,i,y)\md y\md i\le   C \int_0^{i_\dagger} \int_{X(t,0)-R}^{+\infty}  \K_*(x-y) u(t-i,0,y)\md y\md i.
	\end{align*}
By taking the positive part on both sides, together with the assumptions that   $X(t,0)$ is nondecreasing in $t\in(t_*,+\infty)$  and that  $u(t,0,x)\le 0$ for $t\ge t_*$ and $x\in[X(t,0)-R,X(t,0)]$, one has  
\begin{align*}
		u^+(t,0,x)&\le   C \int_0^{i_\dagger}  \bigg(\int_{X(t,0)-R}^{+\infty}  \K_*(x-y) u(t-i,0,y)\md y\bigg)^+\md i\\
		&\le  C \int_0^{i_\dagger}  \sup_{x\ge X(t,0)-R}  u^+(t-i,0,x)\md i\le C \int_0^{i_\dagger}  \sup_{x\ge X(t-i,0)-R}  u^+(t-i,0,x)\md i\\
		&=  C \int_0^{i_\dagger}  \sup_{x\ge X(t-i,0)} u^+(t-i,0,x)\md i	= C \int_{t-i_\dagger}^t \sup_{x\ge X(\tau,0)} u^+(\tau,0,x)\md \tau\\
	&\le C \int_{t_*}^{t} \sup_{x\ge X(\tau,0)} u^+(\tau,0,x) \md \tau, ~~~~~~~~~~~~~~~ t\ge i_\dagger+t_*,~~x\ge X(t,0).
\end{align*}
It further gives that 
	\begin{equation*}
		\sup_{x\ge X(t,0)} u^+(t,0,x)\le C \int_{t_*}^{t} \sup_{x\ge X(\tau,0)} u^+(\tau,0,x) \md \tau= C \int_{i_\dagger+t_*}^t \sup_{x\ge X(\tau,0)} u^+(\tau,0,x)\md \tau,~~~~~~~t\ge i_\dagger+t_*,
	\end{equation*}
where the last equality is due to the assumption that   $u(t,0,x)\le0$ for $t\in [ t_*,i_\dagger+t_*]$ and $x\ge X(t,0)$.
	The Gronwall inequality then implies that $\sup_{x\ge X(t,0)}u^+(t,0,x)= 0$ for $t\ge i_\dagger+t_*$, which proves \eqref{2.5-claim2}.
	
	Finally, let us consider $u(t,i,x)$ in the region $t\ge i+i_\dagger+t_*$, $i\in[0,i_\dagger)$ and $x\ge X(t,i)$. It immediately  follows from \eqref{2.5-claim} and \eqref{2.5-claim2} that
	\begin{equation*}
		~~~~~~~~~~~~~	u(t,i,x)\le u(t-i,0,x)\le 0,~~~~~~~t\ge i+i_\dagger+t_*,~i\in [0,i_\dagger),~x\ge \max_{s\in[0,i]}X(t-i+s,s).
	\end{equation*}
	We divide into two situations: either $X(t,i)= \max_{s\in[0,i]} X(t-i+s,s)$ or $X(t,i)< \max_{s\in[0,i]} X(t-i+s,s)$. For the former case, it is trivial. For the latter,  let us take any point $\hat x\in[X(t,i), \max_{s\in[0,i]} X(t-i+s,s)]$. By the continuity (in $t$ and in $i$) of $X(t,i)$,  there necessarily exists $\hat s\in[0,i]$ such that  $\hat x=\max_{s\in[\hat s,i]}X(t-i+s,s)=X(t-i+\hat s,\hat s)$. The decreasing property  of the function $v$ in \eqref{2.5-v:eqn} implies in particular that
	\begin{equation*}
		~~~	u(t,i,\hat x)=v(i,\hat x)\le v(\hat s,\hat x)= u(t-i+\hat s,\hat s, \hat x)\le 0,~~~~~~t\ge i+i_\dagger+t_*,~i\in(0,i_\dagger),
	\end{equation*}
	in which the last inequality is due to the assumption that $u(t,i,X(t,i))\le 0$ for $t\ge i+i_\dagger+t_*$ and $i\in[0,i_\dagger)$.
	Consequently, together with \eqref{2.5-claim2}, we conclude that $u(t,i,x)\le 0$ for $t\ge i+i_\dagger+t_*$, $i\in[0,i_\dagger)$ and $x\ge X(t,i)$.
\end{proof}

\begin{proof}[Proof of Proposition \ref{prop_cp_linear transport_2}]
	Analogous to the proof of Proposition \ref{prop_mp}, we first claim that 
	\begin{equation}\label{2.4-claim}
		u(t,i,x)\le u(t-i,0,x),~~~~~~~t\ge i+t_*,~i\in[0,i_\dagger),~X^-(t)\le x\le X^+(t-i).
	\end{equation}
	To do so, let us define for $t\ge i+t_*$ and $i\in(0,i_\dagger)$,
\begin{equation}\label{2.4-v:eqn}
	v(s,x):=u(t-i+s,s,x),~~~~~~~~s\in[0,i],~X^-(t-i+s)\le x\le X^+(t-i).
\end{equation}
Then  the transport inequality of \eqref{mp-u2} implies that $\partial_s v(s,x)\le 0$ for $s\in(0,i]$, together with the continuity of $v$ in $s\in[0,i]$, implies that for $t\ge i+t_*$ and $i\in(0,i_\dagger)$,
\begin{equation*}
~~~~~~	v(i,x)\le v(0,x),~~~~\text{for}~~ \max_{s\in[0,i]} X^-(t-i+s)\le x\le X^+(t-i).
\end{equation*}
The nondecreasing property of $X^-$ indicates that $\max_{s\in[0,i]} X^-(t-i+s)=X^-(t)$. Our claim \eqref{2.4-claim} is therefore achieved. 

	In what follows, we claim that 
	\begin{equation}
		\label{2.4-claim2}
		u(t,0,x)\le 0,~~~~~~~~~~t\ge i_\dagger+t_*,~X^-(t)\le x\le X^+(t).
	\end{equation}
Indeed, we have for $t\ge i_\dagger+t_*$ and $x\in[X^-(t), X^+(t)]$,
\begin{equation*}
	\begin{aligned}
		u(t,0,x)&\le \S_0 \int_0^{i_\dagger} \omega(i) e^{-\alpha_*c_* i} \K_* *u(t,i,x)\md i=\S_0 \int_0^{i_\dagger} \omega(i) e^{-\alpha_*c_* i} \int_{X^-(t)-R}^{X^+(t)+R} \K_*(x-y) u(t,i,y)\md y\md i\\
		&\le \S_0 \int_0^{i_\dagger} \omega(i) e^{-\alpha_*c_* i} \int_{X^-(t)}^{X^+(t-i)} \K_*(x-y) u(t,i,y)\md y\md i\le C \int_0^{i_\dagger}  \int_{X^-(t)}^{X^+(t-i)} \K_*(x-y) u(t,i,y)\md y\md i,
	\end{aligned}
\end{equation*}
where the second inequality is due to the assumption that $u(t,i,x)\le 0$ for $t\ge i+t_*$, $i\in[0,i_\dagger)$ and $x\in[X^-(t)-R,X^-(t)]\cup[X^+(t-i),X^+(t)+R]$. 
Applying \eqref{2.4-claim} to the right-hand side of the above formula, it follows that  for $t\ge i_\dagger+t_*$ and $x\in[X^-(t), X^+(t)]$,
\begin{align*}
	u(t,0,x)& \le C \int_0^{i_\dagger}  \int_{X^-(t)}^{X^+(t-i)} \K_*(x-y) u(t-i,0,y)\md y\md i\\
	& \le C \int_0^{i_\dagger}  \sup_{X^-(t)\le x\le X^+(t-i)}  u(t-i,0,x)\md i=C \int_{t-i_\dagger}^{t}  \sup_{X^-(t)\le x\le X^+(\tau)}  u(\tau,0,x)\md \tau\\
	&\le C \int_{t_*}^{t}  \sup_{X^-(\tau)\le x\le X^+(\tau)}  u(\tau,0,x)\md \tau\le C \int_{i_\dagger+t_*}^{t}  \sup_{X^-(\tau)\le x\le X^+(\tau)}  u(\tau,0,x)\md \tau,
\end{align*}
by noticing that $X^-(s)$ is nondecreasing in $s\ge t_*$ and that  $u(t,0,x)\le0$ for $t\in [ t_*,i_\dagger+t_*]$ and $x\in[X^-(t), X^+(t)]$. Thus,
\begin{equation*}
\sup_{X^-(t)\le x\le X^+(t)} u(t,0,x)	\le C \int_{i_\dagger+t_*}^{t}  \sup_{X^-(\tau)\le x\le X^+(\tau)}  u(\tau,0,x)\md \tau,~~~~~~t\ge i_\dagger+t_*.
\end{equation*}
Our claim \eqref{2.4-claim2} is an immediate consequence of the Gronwall inequality.

Finally, it follows from \eqref{2.4-claim} and \eqref{2.4-claim2} that
\begin{equation*}
	u(t,i,x)\le u(t-i,0,x)\le 0,~~~~~t\ge i+i_\dagger+t_*,~i\in[0,i_\dagger),~x\in[X^-(t), X^+(t-i)].
\end{equation*}
This completes the proof.
\end{proof}

\section{Heat kernel aysmptotics for the linearized transport equation}
\label{sec: heat kernel estimate}
The goal of this section is to prove that the linear transport problem, in the reference frame moving with the critical velocity,  shares a common  asymptotic behavior  with the heat equation within the diffusive regime. While heat kernel estimates for integral equations with symmetric kernels are not entirely new - see for instance \cite{BC2002}, heat kernel estimates for linear nonlocal equations with the critical velocity  only date back to \cite{BFRZ2023} for discrete equations, and \cite{Roquejoffre2022} for integral equations.
The strategy is to look at the equivalent form of \eqref{moving frame-u}:
\begin{equation}
	\label{sol_u}
	\begin{aligned}
		u(t,i,x)=
		\begin{cases}
			u_0(i-t,x),  &t\le i,~x\in\R,\\
			\Phi(t-i,x), &t>i,~ x\in\R.
		\end{cases}
	\end{aligned}
\end{equation}
The function $\Phi$  satisfies the following linear equation, that we will sometimes call by its standard name ``renewal equation'', for $t>0$ and $x\in\R$:
\begin{equation}\label{renewal eqn=3}
	\begin{aligned}
		\Phi(t,x)&=\S_0  \int_0^{i_\dagger} \omega(i) e^{-\alpha_*c_* i}\K_* * u(t,i,x)\md i\\
		&=\S_0 \int_0^t \omega(i) e^{-\alpha_*c_* i}\K_* * \Phi(t-i,x)\md i+\underbrace{\S_0\int_0^{\infty} \omega(i+t) e^{-\alpha_*c_* (i+t)}\K_* *  u_0(i,x)\md i}_{=:\mathcal{F}(t,x)}
	\end{aligned}
\end{equation}
where we have denoted the last term in the right-hand side by  $\mathcal{F}(t,x)=\mathcal{F}(t,x; u_0)$. It is then clear that 
\begin{equation*}
	\begin{aligned}
		\mathcal{F}(t,x)=\begin{cases}
			\S_0\int_0^{i_\dagger-t} \omega(i+t) e^{-\alpha_*c_* (i+t)}\K_* *  u_0(i,x)\md i,~~~~&\text{if}~~~t<i_\dagger,\\
				0, &\text{if}~~~t\ge i_\dagger.
			\end{cases}
			\end{aligned}
	\end{equation*}

A significant body of this section will contribute to the precise information of the nontrivial part of the function $u$ within the diffusive regime, for which we analyze the equivalent linear renewal equation \eqref{renewal eqn=3} in this section, showing that the solution $\Phi$ to \eqref{renewal eqn=3} performs  a universal Gaussian asymptotic profile as that of the Laplacian operator for which the low frequencies in the Fourier variable play the dominant role.

 Before stating our main result, let us introduce the normalized Gaussian profile:
\begin{equation*}
	\mathscr{G}(x):=\frac{1}{\sqrt{2\pi}}e^{-\frac{x^2}{2}},~~~~~~x\in\R.
\end{equation*}
For  notational simplicity and consistency with previous section, we set
\begin{equation}
	\label{notation}
	a_*:=\int_\R  x\K_*(x)\md x>0,~~~2d_*:=\int_\R  x^2\K_*(x)\md x>0,~~~\upsilon_0:=\frac{c_*\widehat{\K_*}(0)}{a_*},~~~b:=\frac{1}{2}\upsilon_0\partial^2_{\alpha}D_{c_*}(\alpha_*)>0.
\end{equation}

\begin{thm}\label{thm_3.1} 
	Let  $\Phi(t,x)$ be the solution to the renewal equation \eqref{renewal eqn=3} for $t>0$ and $x\in\R$ with an absolutely continuous and compactly supported initial condition  $u_0$ in $[0,i_\dagger)\times\R$. Then for $T>i_\dagger$ sufficiently large, there is a nonnegative kernel $\mathcal{G}(t,x)$ defined  for $t\ge T$ and $x\in\R$ such that for any $\gamma_*\in(3/8,1/2)$, the following estimate holds true:
	\begin{equation}\label{4.2-1}
		\Vert \Phi(t,x)-\mathcal{G}**\mathcal{F}(t,x)\Vert_{L^{\infty}_x(\R)}\le C \big(\Vert u_0\Vert_{L^\infty([0,i_\dagger)\times\R)}+\max_{[0,i_\dagger)}\big\Vert  u_0(i,x) \big\Vert_{H^1_x(\R)} \big)e^{-\frac{b}{2}t^{1-2\gamma_*}},~~~~~t\ge T.
	\end{equation}
	Here, $\mathcal{G}** \mathcal{F}(t,x):=\int_\R\int_0^t \mathcal{G}(t-s,x-y)\mathcal{F}(s,y)\md s\md y$, and $b>0$ is given in \eqref{notation}. Moreover, by  choosing $\varsigma\in(0,1/2-\gamma_*)$,  we have for  $t\ge T$, up to increasing $T$ if needed, 
	\begin{itemize}
		\item[(i)] if $|x-c_*t|\le t^{\frac{1}{2}+\varsigma}$,
		\begin{equation}
			\label{estimate_1_Gaussian}
			\mathcal{G}(t,x)=\bigg[\frac{\upsilon_0}{\sqrt{2bt}}+\frac{k\upsilon_0}{4b^2t}\mathcal{P}\left(\frac{x-c_*t}{\sqrt{2bt}}\right)	\bigg]\mathscr{G}\left(\frac{x-c_*t}{\sqrt{2bt}}\right)	+\mathcal{R}(t,x),
		\end{equation} 
	 with
		$\mathcal{P}(x):=-3x	+x^3,\ x\in\R,$
 and
		\begin{equation}
			\label{remaining term R}
			|\mathcal{R}(t,x)|\lesssim t^{\varsigma-1} e^{-\frac{(x-c_*t)^2}{4bt}	}.
		\end{equation}
	Therefore, $\mathcal{G}(t,x)=\frac{\upsilon_0}{\sqrt{2bt}}\mathscr{G}\left(\frac{x-c_*t}{\sqrt{2bt}}\right)\Big(1+o_{t\to+\infty}(1)\Big)$ for  $|x-c_*t|\le t^{\frac{1}{2}+\varsigma}$;
		\item[(ii)] if $|x-c_*t|\ge  t^{\frac{1}{2}+\varsigma}$,
		\begin{equation}
			\label{estimate_2}
			\mathcal{G}(t,x)\lesssim t^{-\gamma_*}e^{-t^\varsigma}.
		\end{equation}
	\end{itemize}
\end{thm}
\begin{rmk} 
\textnormal{ We point out that it is sufficient for the heat kernel $\mathcal{G}(t,\cdot)$  to be defined for $t>i_\dagger$ sufficiently large. As such,  the convolution of  $\mathcal{G}(t,\cdot)$  and $\mathcal{F}(t,\cdot)$ with respect to time is always well-defined for all $t$ sufficiently large, by noticing from \eqref{renewal eqn=3} that $\mathcal{F}(t,x)\equiv0$ for $t>i_\dagger$,  uniformly in $x\in\R$. 
} 

\end{rmk}

Our analysis carried out in  this section will take advantage of both Fourier and Laplace transforms\footnote{We denote by $\widehat{f}(\xi)=\int_\R e^{-\mathbf{i}x\xi}f(x)\md x$ for $\xi\in\R$ the Fourier transform of function $f(x)$ for $x\in\R$, by $\widetilde{g}(\lambda)=\int_0^\infty e^{-\lambda t}g(t)\md t$ for $\lambda\in\mathbb{C}$ the Laplace transform of function $g(t)$ for $t\in\R_+$, and by $\widetilde h(\lambda,\xi):=\widetilde{\widehat{h}}(\lambda,\xi)=\int_0^\infty e^{-\lambda t}\int_\R e^{-\mathbf{i}x\xi}h(t,x)\md x\md t$ for $\lambda\in\mathbb{C}$ and $\xi\in\R$ the Fourier and Laplace transforms of function $h(t,x)$ for $t\in\R_+$ and $x\in\R$.}. In what follows, we shall prove that  the long time behavior of the renewal function $\Phi(t,x)$, solution to \eqref{renewal eqn=3}, is dominated by  the low frenquency region of $\widehat\Phi(t,\xi)$ - Fourier transform of $\Phi(t,x)$, while the long time behavior of $\widehat\Phi(t,\xi)$ is determined by   
   the singularity of its Laplace transform $\widetilde{\Phi}(\lambda,\xi)$.
   
      First of all, applying  in \eqref{sol_u} the Fourier transform on $u$ with respect to  $x\in\R$,
it then gives
\begin{equation*}
	\begin{aligned}
			\hat u(t,i,\xi)=\begin{cases}
			\hat u(0,i-t,\xi),~~ &t\le i, ~\xi\in\R,\\
		\widehat\Phi(t-i,\xi), &t>i,~ \xi\in\R,
		\end{cases}
	\end{aligned}
\end{equation*}
where
\begin{equation}\label{renewal eqn-Fourier}
\begin{aligned}
	\widehat\Phi(t,\xi)= \S_0\widehat{\K_*}(\xi)\int_0^t \omega(i) e^{-\alpha_*c_* i} \widehat\Phi(t-i,\xi)\md i+\widehat{ \mathcal{F}}(t,\xi),~~~~~~~~~~~~ t>0, ~\xi\in\R,
\end{aligned}
\end{equation}
with 
\begin{equation}\label{F_fourier}
	\widehat{ \mathcal{F}}(t,\xi)=\S_0\widehat{\K_*}(\xi)\int_0^{\infty}\omega(i+t) e^{-\alpha_* c_*(i+t)} \widehat{u}_0(i,\xi)\md i,~~~~~~~~~~~~ t>0, ~\xi\in\R.
\end{equation}

By noticing the convolution in time variable in \eqref{renewal eqn-Fourier},  we continue  applying the Laplace transform in \eqref{renewal eqn-Fourier}  with respect to time variable $t>0$, then  we arrive at
\begin{equation}\label{A1}
	\widetilde{\Phi}(\lambda,\xi)=\S_0\widehat{\K_*}(\xi)
	\widetilde{\omega}(\lambda+\alpha_*c_*)\widetilde{\Phi}(\lambda,\xi)+\widetilde{\mathcal{F}}(\lambda,\xi),~~~~~\lambda\in\mathbb{C},~\xi\in\R,
\end{equation}
with 
\begin{equation*}\label{F_fourier+laplace}
	\widetilde{\mathcal{F}}(\lambda,\xi)=\S_0\widehat{\K_*}(\xi)\int_0^\infty e^{-\lambda t}\int_0^\infty \omega(i+t)e^{-\alpha_* c_*(i+t)}\widehat{u}_0(i,\xi)\md i\md t,~~~~~\lambda\in\mathbb{C},~\xi\in\R.
\end{equation*}
This further implies that
$
	\displaystyle\widetilde{\Phi}(\lambda,\xi)=\frac{\widetilde{\mathcal{F}}(\lambda,\xi)
		}{1-\S_0\widehat{\K_*}(\xi)
			\widetilde{\omega}(\lambda+\alpha_*c_*)},
$
and thus substituting it back into \eqref{A1} yields
\begin{equation}\label{renewal eqn-Laplace}
	\widetilde{\Phi}(\lambda,\xi)=\underbrace{\frac{\S_0\widehat{\K_*}(\xi)
		\widetilde{\omega}(\lambda+\alpha_*c_*)}{1-\S_0\widehat{\K_*}(\xi)
		\widetilde{\omega}(\lambda+\alpha_*c_*)}}_{\widetilde{\mathcal{H}}(\lambda,\xi)}\widetilde{\mathcal{F}}(\lambda,\xi)+\widetilde{\mathcal{F}}(\lambda,\xi), ~~~~~\lambda\in\mathbb{C},~\xi\in\R.
\end{equation}
\subsubsection*{The characteristic equation}
The function $\widetilde{\Phi}$ is meromorphic and its poles are among the roots of the \textit{Lotka's characteristic equation}
\begin{equation}
	\label{A4}
	\S_0\widehat{\K_*}(\xi)
	\widetilde{\omega}(\lambda+\alpha_*c_*)=1, ~~~~~\lambda\in\mathbb{C},~\xi\in\R.
\end{equation}
At this point, it is useful to realize that the dispersion relation \eqref{dr} can be written as 
$\S_0\widehat{\K_*}(0)
\widetilde{\omega}(\alpha_*c_*)=1$.

The  lemma below is devoted to the analysis of the solution $\lambda\in\mathbb{C}$ of  the Lotka's characteristic equation  \eqref{A4} with respect to different ranges of $\xi\in\R$. 

\begin{lem}\label{lem_analysis of lambda}
	The solutions $\lambda(\xi)\in\mathbb{C}$ of \eqref{A4} with respect to $\xi\in\R$ satisfy the following properties:
	\begin{itemize}
		\item[(i)] $\Re\lambda(\xi)< 0$ when $\xi\in\R\backslash\{0\}$, and $\lambda(0)=0$;
		
		\item[(ii)] there exists $\varep\!>\!0$ such that  $\lambda(\xi)$ is uniquely defined for $\xi\in[-\varep,\varep]$, and it has the  expansion:
		\begin{equation}\label{precise lambda}
			\lambda(\xi)=-\mathbf{i}c_*\xi- b\xi^2+\mathbf{i}k\xi^3 +O(\xi^4),~~~~~\text{as}~\xi\to 0,
		\end{equation}
		with $
			b>0$ given in \eqref{notation}, and $k\in\R$.
	Therefore, up to decreasing $\varep>0$,  
	\begin{equation}
		\label{rough lambda}
			\big|\lambda(\xi)-(-\mathbf{i}c_*\xi- b\xi^2)\big|\le \frac{b}{2}\xi^2~~~~~~\text{for}~\xi\in[-\varep,  \varep].
	\end{equation}
		
		\item[(iii)] for $\varep>0$ chosen in (ii), there exists  some $\bar\lambda<0$ such that $\lambda(\xi)$ for each $\xi\in\R\backslash[-\varep,\varep]$ satisfies $\Re\lambda(\xi)\le \bar \lambda<0$.
	\end{itemize}
\end{lem}
\begin{proof}[Proof of Lemma \ref{lem_analysis of lambda}]
\textit{Proof of (i)}.	Let us first claim that $\Re\lambda(\xi)\le 0$ for all $\xi\in\R$.
	Assume by contradiction that the solution $\lambda_0:=\lambda(\xi_0)\in\mathbb{C}$ to \eqref{A4} associated with some $\xi_0\in\R$ satisfies $\Re\lambda_0>0$. It is easily seen that $|\widehat{\K_*}(\xi_0)|=\Big|\int_\R e^{-\mathbf{i}x\xi_0}\K_*(x)\md x\Big|\le\widehat{\K_*}(0)$, and
	\begin{equation*}
		|\widetilde{\omega}(\lambda_0+\alpha_*c_*)|=\Big|\int_0^\infty e^{-(\Re\lambda_0+\mathbf{i}\Im\lambda_0)t-\alpha_*c_*t}\omega(t)\md t\Big|<\int_0^\infty e^{-\alpha_*c_*t}\omega(t)\md t=\widetilde{\omega}(\alpha_*c_*).
	\end{equation*}
	Therefore, combining with the dispersion relation $\S_0\widehat{\K_*}(0)
	\widetilde{\omega}(\alpha_*c_*)=1$, we arrive at
	\begin{equation*}
		1=	|	\S_0\widehat{\K_*}(\xi_0)
		\widetilde{\omega}(\lambda_0+\alpha_*c_*)|< 	\S_0\widehat{\K_*}(0)
		\widetilde{\omega}(\alpha_*c_*)=1,
	\end{equation*}
	a contradiction. Our claim is therefore achieved.
	
	To prove that $\Re\lambda(\xi)<0$ when $\xi\in\R\backslash\{0\}$, we assume towards contradiction that there exists $\xi_1\in\R\backslash\{0\}$ such that the associated solution $\lambda_1:=\lambda(\xi_1)\in\mathbb{C}$ of \eqref{A4} has the form $\lambda_1=\mathbf{i}\eta$ with $\eta\in\R$. Then, 
	\begin{align*}
		\S_0\widehat{\K_*}(\xi_1)
		\widetilde{\omega}(\mathbf{i}\eta+\alpha_*c_*)&=\S_0\int_\R e^{-\mathbf{i}x\xi_1}\K_*(x)\md x\int_0^\infty e^{-\mathbf{i}\eta t-\alpha_*c_*t}\omega(t)\md t\\
		&=\S_0 \int_\R\int_0^\infty \big(\cos(x\xi_1)-\mathbf{i}\sin(x\xi_1)\big)\big(
		\cos(\eta t)-\mathbf{i}\sin(\eta t)
		\big)\K_*(x) e^{-\alpha_*c_*t}\omega(t)\md t  \md x.~~~~~~~
	\end{align*}
	Therefore, the real part of $	\S_0\widehat{\K_*}(\xi_1)
	\widetilde{\omega}(\mathbf{i}\eta+\alpha_*c_*)$ is
	\begin{align*}
		&\S_0 \int_\R\int_0^\infty \big(\cos(x\xi_1)\cos(\eta t)-\sin(x\xi_1)\sin(
		\eta t)\big)\K_*(x) e^{-\alpha_*c_*t}\omega(t)\md t  \md x \\
		&= 	\S_0 \int_\R\int_0^\infty \cos(x\xi_1+\eta t)\K_*(x) e^{-\alpha_*c_*t}\omega(t)\md t  \md x.
	\end{align*}
	Since the real numbers $\xi_1$ and $\eta$ are bounded and the cosine function  is no larger than 1 and not identically 1, we then derive that  $\Re\big(\S_0\widehat{\K_*}(\xi_1)
	\widetilde{\omega}(\mathbf{i}\eta+\alpha_*c_*)\big)$ is less than but not identically equal to 1, by using again the dispersion relation $\S_0\widehat{\K_*}(0)
	\widetilde{\omega}(\alpha_*c_*)=1$. This  contradicts the fact that $\S_0\widehat{\K_*}(\xi_1)
	\widetilde{\omega}(\mathbf{i}\eta+\alpha_*c_*)=1$. Hence, we conclude that $\Re\lambda(\xi)<0$ when $\xi\in\R\backslash\{0\}$.  Moreover, it is obvious to see from the  dispersion relation $\S_0\widehat{\K_*}(0)
	\widetilde{\omega}(\alpha_*c_*)=1$,  that $\lambda=0$ is the solution to \eqref{A4} corresponding to $\xi=0$. Therefore, (i) is proved.
	
	\vspace{2mm}
	
	\noindent 
	{\it Proof of (ii)}.  It suffices to prove \eqref{precise lambda}, then \eqref{rough lambda} will be an immediate consequence, up to decreasing $\varep>0$ if necessary.
By setting $$\mathcal{J}(\lambda,\xi):=\S_0\widehat{\K_*}(\xi)\widetilde \omega(\lambda+\alpha_* c_*)-1,~~~~~~~\lambda\in\mathbb{C},~\xi\in\R,$$
	we notice that $\mathcal{J}$ is a continuous function of $\lambda\in\mathbb{C}$ and $\xi\in\R$, and depends holomorphically on $\lambda\in\mathbb{C}$ for each fixed $\xi\in\R$. Moreover, $\mathcal{J}(0,0)=0$ and $\partial_\lambda \mathcal{J}(0,0)=-\S_0\widehat{\K_*}(0)\int_0^\infty t e^{-\alpha_* c_* t}\omega(t)\md t\neq 0$. By the holomorphic implicit function theorem, it immediately follows that there exists $\varep>0$ such that there is a unique holomorphic function $\lambda(\xi)$ in $\mathbb{C}$ defined for $|\xi|\le \varep$ satisfying $\lambda(0)=0$ and $\mathcal{J}(\lambda(\xi),\xi)=0$. Again, since $\xi\in[-\varep,\varep]\mapsto\lambda(\xi)\in\mathbb{C}$ is holomorphic, $\lambda(\xi)$ then has a series expansion with respect to $\xi$ in $[-\varep,\varep]$, up to decreasing $\varep$ if necessary. To get proceed with the expansion of $\lambda(\xi)$, we  notice from the notation \eqref{notation} that
	the derivatives of $\widehat{\K_*}(\xi)$ valued at $\xi=0$ can be written as:
	\begin{equation}\label{ad}
		\widehat{\K_*}'(0)=-\mathbf{i}\int_\R  x\K_*(x)\md x=-\mathbf{i}a_*,~~~\widehat{\K_*}''(0)=-\int_\R  x^2\K_*(x)\md x =-2d_*.
	\end{equation}
	
	First of all, differentiating the equation \eqref{A4}, namely,
	\begin{equation*}
		\S_0\widehat{\K_*}(\xi)
		\widetilde{\omega}\big(\lambda(\xi)+\alpha_*c_*\big)=	\S_0\widehat{\K_*}(\xi)\int_0^\infty e^{-\lambda(\xi)t-\alpha_*c_*t}\omega(t)\md t=1
	\end{equation*}
	with respect to $\xi$ implies
	\begin{equation}\label{3.12'}
		\widehat{\K_*}'(\xi)\widetilde{\omega}\big(\lambda(\xi)+\alpha_*c_*\big)- \widehat{\K_*}(\xi)\lambda'(\xi)\int_0^\infty t e^{-\lambda(\xi)t-\alpha_*c_*t}\omega(t)\md t=0.
	\end{equation}
	Letting $\xi=0$, the above formula gives 
	\begin{equation*}
		\widehat{\K_*}'(0)\widetilde{\omega}(\alpha_*c_*)=\widehat{\K_*}(0)\lambda'(0)\int_0^\infty t e^{-\alpha_*c_*t}\omega(t)\md t,
	\end{equation*}
in which it follows from the property   $\partial_\alpha D_{c_*}(\alpha_*)=0$ (see \eqref{dr_1}) that
	\begin{equation*}
	\int_0^\infty t\omega(t) e ^{-\alpha_* c_*t}\md t=\frac{ a_*}{c_*  \widehat{\K_*}(0)}	\widetilde{\omega}(\alpha_*c_*)=\frac{ \mathbf{i}\widehat{\K_*}'(0)}{c_*  \widehat{\K_*}(0)}	\widetilde{\omega}(\alpha_*c_*).
\end{equation*}
Therefore, $$\lambda'(0)=-\mathbf{i}c_*.$$

	By further differentiating  \eqref{3.12'} with respect to $\xi$, it follows that
	\begin{equation}
		\label{3.12''}
	\begin{aligned}	&\widehat{\K_*}''(\xi)\widetilde{\omega}\big(\lambda(\xi)+\alpha_*c_*\big)- 2\widehat{\K_*}'(\xi)\lambda'(\xi)\int_0^\infty t e^{-\lambda(\xi)t-\alpha_*c_*t}\omega(t)\md t\\
		&	~~~~~~-\widehat{\K_*}(\xi)\lambda''(\xi)\int_0^\infty t e^{-\lambda(\xi)t-\alpha_*c_*t}\omega(t)\md t+\widehat{\K_*}(\xi)(\lambda'(\xi))^2\int_0^\infty t^2 e^{-\lambda(\xi)t-\alpha_*c_*t}\omega(t)\md t =0.
	\end{aligned}
\end{equation}
	The above formula at $\xi=0$ is reduced to 
	\begin{align*}
		\widehat{\K_*}''(0)\widetilde{\omega}(\alpha_*c_*)-\big( 2\widehat{\K_*}'(0)\lambda'(0)+\widehat{\K_*}(0)\lambda''(0)\big)\int_0^\infty t e^{-\alpha_*c_*t}\omega(t)\md t+\widehat{\K_*}(0)(\lambda'(0))^2\int_0^\infty t^2 e^{-\alpha_*c_*t}\omega(t)\md t =0.
	\end{align*}
Combining
\begin{align*}
	\int_0^\infty t^2\omega(t) e ^{-\alpha_* c_*t}\md t&=
	\frac{1}{\S_0 c_*^2 \widehat{\K_*}(0)}\left(
	\partial^2_{\alpha}D_{c_*}(\alpha_*)+ \S_0 \widehat{\K_*}''(0)\widetilde{\omega}(\alpha_*c_*)+2\S_0 c_*\mathbf{i} \widehat{\K_*}'(0)
	\int_0^\infty t\omega(t) e ^{-\alpha_* c_*t}\md t
	\right)\\
	&=\frac{1}{\S_0 c_*^2 \widehat{\K_*}(0)}\left(
	\partial^2_{\alpha}D_{c_*}(\alpha_*)+\frac{  \widehat{\K_*}''(0)}{ \widehat{\K_*}(0)}-
	\frac{2(\widehat{\K_*}'(0))^2}{\ (\widehat{\K_*}(0))^2}
	\right)
\end{align*}
	with \eqref{ad} and \eqref{notation} eventually leads to
	\begin{equation*}
		\lambda''(0)=-\frac{c_*}{a_*}\widehat{\K_*}(0)\partial^2_{\alpha}D_{c_*}(\alpha_*)=-\upsilon_0\partial^2_{\alpha}D_{c_*}(\alpha_*)=:-2b<0.
	\end{equation*}

	Let us continue differentiating  \eqref{3.12''}  with respect to $\xi$. After an elementary calculation, we derive that
	\begin{align*}
		&\widehat{\K_*}'''(\xi)\widetilde \omega\big(\lambda(\xi)+\alpha_*c_*\big)-\left(3\widehat{\K_*}''(\xi)\lambda'(\xi)+3\widehat{\K_*}'(\xi)\lambda''(\xi)+\widehat{\K_*}(\xi)\lambda'''(\xi)\right)\int_0^\infty t \omega(t) e ^{-(\lambda(\xi)+\alpha_* c_*)t}\md t\\
		&	\!+\!\left(\!3\widehat{\K_*}'(\xi)(\lambda'(\xi))^2\!+\!3\widehat{\K_*}(\xi)\lambda'(\xi)\lambda''(\xi)\!\right)\!\int_0^\infty\!\! t^2\omega(t) e ^{-(\lambda(\xi)+\alpha_* c_*)t}\md t\!-\!\widehat{\K_*}(\xi)(\lambda'(\xi))^3\!\int_0^\infty\!\! t^3\omega(t) e ^{-(\lambda(\xi)+\alpha_* c_*)t}\md t\!=\!0.
	\end{align*}
	Again taking the values at $\xi=0$, we have
	\begin{align*}
-6a_*b\frac{\widehat{\K_*}'(0)}{(\widehat{\K_*}(0))^2}\mathbf{i}	+	\frac{a_*}{\widehat{\K_*}(0)}\lambda'''(0)\mathbf{i}+6b\left(
		\partial^2_{\alpha}D_{c_*}(\alpha_*)+\frac{  \widehat{\K_*}''(0)}{ \widehat{\K_*}(0)}-
		\frac{2(\widehat{\K_*}'(0))^2}{\ (\widehat{\K_*}(0))^2}\right)+c_*	\partial^3_{\alpha}D_{c_*}(\alpha_*)=0,
	\end{align*}
	in which we use 
	\begin{align*}
		\int_0^\infty t^3\omega(t) e ^{-\alpha_* c_*t}\md t&=- \frac{1}{\S_0c_*^3\widehat{\K_*}(0)}\bigg(
		\partial^3_{\alpha}D_{c_*}(\alpha_*)+\S_0\widehat{\K_*}'''(0)\widetilde{\omega}(\alpha_*c_*)\mathbf{i}\\
		&~~~~~~~~~-3\S_0 c_*\widehat{\K_*}''(0) 	\int_0^\infty t\omega(t) e ^{-\alpha_* c_*t}\md t -3\S_0c_*^2\widehat{\K_*}'(0) \mathbf{i}	\int_0^\infty t^2\omega(t) e ^{-\alpha_* c_*t}\md t 
		\bigg).
	\end{align*}
	Therefore, together with \eqref{ad}, we derive that
	\begin{align*}
		\lambda'''(0)=\frac{\mathbf{i}6b\widehat{\K_*}(0)}{a_*}\left(
		\partial^2_{\alpha}D_{c_*}(\alpha_*)-\frac{2d_*}{\widehat{\K_*}(0)}+\frac{a_*^2}{(\widehat{\K_*}(0))^2}+\frac{c_*}{6b}  \partial^3_{\alpha}D_{c_*}(\alpha_*)
		\right)=:\mathbf{i}6k ,\ \hbox{with $k\in\R$.}
	\end{align*}
	
	Consequently, we derive the following asymptotic expansion of $\lambda(\xi)$:
	\begin{equation*}
		\lambda:=\lambda(\xi)=-\mathbf{i}c_*\xi- b\xi^2+\mathbf{i}k\xi^3 +O(\xi^4),~~~~~~\text{as}~\xi\to 0.
	\end{equation*}

	\vspace{2mm}
	
	\noindent
	\textit{Proof of (iii)}.
	  Let $\varep>0$ be as given  in statement (ii).  Since $\K_*\in L^1(\R)$, the Riemann-Lebesgue Lemma implies that  $|\widehat{\K_*}(\xi)|$ converges to 0 as $|\xi|\to\infty$. This implies that $|\widehat{\K_*}(\xi)|\le \frac{1}{2}\widehat{\K_*}(0)$ for $|\xi|\ge R$ with some large $R>0$. Regarding $\varep< |\xi|\le R$, since cosine function is less than 1 but not identically 1, we have 
	  \begin{equation*}
	  	|\widehat{\K_*}(\xi)|\le \int_\R \cos(x\xi)K_*(x)\md x\le (1-\delta)\widehat{\K_*}(0)
	  \end{equation*}
	  for some small $\delta\in(0,\frac{1}{2})$. Therefore, we conclude that
	\begin{equation}
		\label{A5}
		|\widehat{\K_*}(\xi)|\le (1-\delta)\widehat{\K_*}(0)~~~~~~\text{for all}~|\xi|>\varep.
	\end{equation}  
	Define
	\begin{equation*}
		\lambda\in\R\mapsto G(\lambda):=(1-\delta)\S_0\widehat{\K_*}(0)\int_0^\infty e^{-\lambda t-\alpha_* c_*t}\omega(t)\md t -1.
	\end{equation*}
	We notice that $G$ is a strictly decreasing function in $\lambda\in\R$, $G(-\infty)=+\infty$ and $G(0)=-\delta$. Thus, there exists a unique real number $\bar\lambda<0$ such that $G(\bar \lambda)=0$.
	Together with \eqref{A5}, we infer that the solution $\lambda(\xi)$ to \eqref{A4} will  satisfy $\Re\lambda(\xi)\le\bar\lambda<0$ for all $\xi\in\R\backslash[-\varep,\varep]$. We then achieve (iii). This completes the proof.
\end{proof}

\subsubsection*{Reduction of the computation}
This goal of this section is to show that it is sufficient to compute the inverse Fourier transform of a function that is more amenable than $\widehat\Phi$. Here is the result.
\begin{prop}\label{prop_4.1}  Let $\lambda(\xi)\in\mathbb{C}$ be the  unique solution of \eqref{A4} for  $\xi\in[-\varep,\varep]$, with $\varep>0$ as given in Lemma  \ref{lem_analysis of lambda} (ii). Then, the Fourier transform $\widehat\Phi$ of the function $\Phi$ given in \eqref{renewal eqn=3} satisfies the following property:
\begin{equation*}
	\Big|\widehat{\Phi}(t,\xi)- e^{\lambda(\xi) t} \widetilde{\mathcal{F}}(\lambda(\xi),\xi)r(\xi)\Big|	\le C e^{ \boldsymbol{\delta} t},~~~~~~t>0,~~\xi\in[-\varep,\varep].
\end{equation*}
Here,  $\boldsymbol{\delta}<0$ is chosen such that $-\frac{\alpha_*c_*}{2}<\boldsymbol{\delta}<\min_{\xi\in[-\varep,\varep]}\Re\lambda(\xi)(<0)$, and the function   $\xi\in[-\varep,\varep]\mapsto r(\xi)\in\mathbb{C}\backslash\{0\}$ is  holomorphic and has the asymptotic expansion
		\begin{equation}\label{r(xi)}
		r(\xi)=\upsilon_0+\mathbf{i}\upsilon_1\xi+\upsilon_2\xi^2+\mathbf{i}\upsilon_3\xi^3+O(\xi^4)~~~~\text{as}~~\xi\to 0, 
\end{equation}
 with $\upsilon_0>0$ given in \eqref{notation} and $\upsilon_i\in\R\backslash\{0\}$ for $i=1,2,3$. 
\end{prop}

The main ingredient of the proof of Proposition \ref{prop_4.1}
is \cite[Theorem A.7]{IM2017}, relating the long time behavior of a function to the singularity of its Laplace transform. 
\begin{lem}[\cite{IM2017}]
	\label{lem_residue}
	Assume that the Laplace transform $\widetilde f(\lambda)$  of a function $f(t)$ has an isolated pole at $\lambda_0$, with the Laurent series around $\lambda_0$:
	\begin{equation*}
		\widetilde f(\lambda)=\sum_{n=-m}^{+\infty}c_n(\lambda-\lambda_0)^n,~~~~~\lambda\in\mathbb{C}.
	\end{equation*} 
	Also assume  the existence of $\boldsymbol{\delta}< \textnormal{Re}\lambda_0$ such that  $\lambda_0$ is the only  pole of $\widetilde{f}(\lambda)$ with  $\textnormal{Re}\lambda>\boldsymbol{\delta}$ and such that 
	\begin{equation}\label{4.9}
		\lim_{|\lambda|\to+\infty, \boldsymbol{\delta}<\textnormal{Re}\lambda}\widetilde f(\lambda)=0,~~~~	\int_{-\infty}^{+\infty} |\widetilde f(\boldsymbol{\delta}+\mathbf{i} y)|\md y<+\infty,
	\end{equation}
	then, 
	\begin{align}\label{eqn-residue}
		f(t)&=e^{\lambda_0 t}\sum_{n=1}^mc_{-n}\frac{t^{n-1}}{(n-1)!}+\frac{1}{2\pi \mathbf{i}}\int_{\boldsymbol{\delta}-\mathbf{i}\infty}^{\boldsymbol{\delta}+\mathbf{i}\infty}e^{\lambda t}\widetilde f(\lambda)\md \lambda,~~~~~t>0.
	\end{align}
\end{lem}

\begin{proof}[Proof of Proposition \ref{prop_4.1}]
	 Let $\lambda(\xi)\in\mathbb{C}$ be the  unique solution of \eqref{A4} for  $\xi\in[-\varep,\varep]$, with
	 $\varep>0$ as given in Lemma \ref{lem_analysis of lambda} (ii), and,  up to decreasing $\varep$, such that
$
	 \displaystyle	-\frac{\alpha_*c_*}{2}<\min_{\xi\in[-\varep,\varep]}\Re\lambda(\xi).
$

Consider the function $\widetilde{\mathcal{F}}(\lambda,\xi)\widetilde{\mathcal{H}}(\lambda,\xi)$ for $\lambda\in\mathbb{C}$ and $\xi\in\R$, with $\widetilde{\mathcal{H}}(\lambda,\xi)$ given in \eqref{renewal eqn-Laplace}.
First, we have that
\begin{equation}\label{W-1}
		\lim_{|\lambda|\to+\infty, -\frac{1}{2}\alpha_* c_*<\textnormal{Re}\lambda}\widetilde{\mathcal{F}}(\lambda,\xi)\widetilde{\mathcal{H}}(\lambda,\xi)=0,~~~~\text{uniformly in}~\xi\in\R.
\end{equation}
In fact, we observe that $|\widehat{\K_*}(\xi)|\le \widehat{\K_*}(0)$ for all $\xi\in\R$, thus
\begin{align*}
|\widetilde{\mathcal{F}}(\lambda,\xi)|&= \bigg|\S_0\widehat{\K_*}(\xi)\int_0^\infty e^{-\lambda t}\int_0^\infty \omega(i+t)e^{-\alpha_* c_*(i+t)}\widehat{u}_0(i,\xi)\md i\md t\bigg|\\
&\le C	\int_0^\infty e^{-(\Re\lambda+\alpha_* c_*) t}\int_0^\infty \omega(i+t)e^{-\alpha_* c_*i}\big|\widehat{u}_0(i,\xi)\big|\md i\md t\\
&\le C 	\int_0^\infty e^{-\frac{1}{2}\alpha_* c_* t}\md t \int_0^\infty \big|\widehat{u}_0(i,\xi)\big|\md i=C \int_0^\infty \bigg|\int_\R e^{-\mathbf{i}x\xi}u_0(i,x)\md x\bigg|\md i\\
&\le C \int_0^\infty \int_\R |u_0(i,x)|\md x\md i<+\infty,~~~~~~~\Re\lambda>-\frac{\alpha_* c_*}{2},~\xi\in\R,
\end{align*}
thanks to $u_0(i,x)\in L^1([0,i_\dagger)\times\R)$.
  Moreover, since $\widetilde\omega(\lambda+\alpha_*c_*)=\int_0^\infty e^{-(\lambda+\alpha_* c_*)t}\omega(t)\md t$ converges absolutely for $\Re\lambda> -\alpha_* c_*$, we
derive from  \cite[Theorem 23.7]{Doetsch1974} that    
\begin{equation}\label{4.25}
	\lim_{|\lambda|\to+\infty, -\alpha_* c_*<  \textnormal{Re}\lambda}\widetilde\omega(\lambda+\alpha_*c_*)=0.
\end{equation}
This implies that $\lim_{|\lambda|\to+\infty, -\frac{1}{2}\alpha_* c_*\le  \textnormal{Re}\lambda}\widetilde{\mathcal{H}}(\lambda,\xi)=0$, uniformly in $\xi\in\R$. Then, \eqref{W-1} follows.

	Next, let us show that 
\begin{equation}\label{W-2}
	\int_{-\infty}^{+\infty} |\widetilde{\mathcal{F}}(\boldsymbol{\sigma}+\mathbf{i} y,\xi)\widetilde{\mathcal{H}}(\boldsymbol{\sigma}+\mathbf{i} y,\xi)|\md y<+\infty,~~~ \text{uniformly in}~\xi\in[-\varep,\varep], 
\end{equation}	
with  $\boldsymbol{\sigma}>-\alpha_*c_*$ such that there is 
no root of \eqref{A4} on the line $\Re\lambda=\boldsymbol{\sigma}$ for all $\xi\in[-\varep,\varep]$. 
Indeed, we infer from the choice of $\boldsymbol{\sigma}$, \eqref{4.25} together with $|\widehat{\K_*}(\xi)|\le \widehat{\K_*}(0)$ for all $\xi\in\R$ that
$$m_{\boldsymbol{\sigma}}:=\inf_{y\in\R,\xi\in[-\varep,\varep]}\big|1-\S_0\widehat{\K_*}(\xi)
\widetilde{\omega}(\boldsymbol{\sigma}+\mathbf{i}y+\alpha_*c_*)\big|>0.$$
Define
\begin{equation}
	\label{func def_f g}
\begin{aligned}
	f_{\boldsymbol{\sigma}}(t):=\begin{cases}
		e^{-({\boldsymbol{\sigma}}+\alpha_* c_*)t} 
		~~~&t\ge 0,\\
		0,~~~~~~~&t<0,
	\end{cases}~~~~~~~~g_{\boldsymbol{\sigma}}(t):=\begin{cases}
	e^{-({\boldsymbol{\sigma}}+\alpha_* c_*)t} \omega(t),~~~&t\ge 0,\\
	0,~~~~~~~~~~~~~~~~~~~~~~~&t<0,
\end{cases}
\end{aligned}
\end{equation}
 both of which actually vanish outside $[0,i_\dagger]$. Since $	f_{\boldsymbol{\sigma}}(t)$ and $g_{\boldsymbol{\sigma}}(t)$  belong to $L^1(\R)\cap L^2(\R)$ due to $\boldsymbol{\sigma}>-\alpha_*c_*$, their Fourier transforms $\widehat{f}_{\boldsymbol{\sigma}}(y)$ and	$\widehat{g}_{\boldsymbol{\sigma}}(y)$  belong to $L^2(\R)$. Moreover, 
\begin{equation*}
	\widehat{f}_{\boldsymbol{\sigma}}(y)=\int_\R e^{-\mathbf{i}ty} f_{\boldsymbol{\sigma}}(t) \md t=\int_0^\infty e^{-({\boldsymbol{\sigma}}+\mathbf{i}y+\alpha_* c_*)t}\md t,
\end{equation*}
\begin{equation*}
	\widehat{g}_{\boldsymbol{\sigma}}(y)=\int_\R e^{-\mathbf{i}ty} g_{\boldsymbol{\sigma}}(t) \md t=\int_0^\infty e^{-({\boldsymbol{\sigma}}+\mathbf{i}y+\alpha_* c_*)t} \omega(t)\md t=\widetilde\omega({\boldsymbol{\sigma}}+\mathbf{i}y+\alpha_* c_*).
\end{equation*}
Thus, by noticing that
\begin{equation*}
	\label{4.28}
	\begin{aligned}
		|\widetilde{\mathcal{F}}(\boldsymbol{\sigma}+\mathbf{i} y,\xi)|&=\bigg|\S_0\widehat{\K_*}(\xi)\int_0^\infty e^{-(\boldsymbol{\sigma}+\mathbf{i} y) t}\int_0^\infty \omega(i+t)e^{-\alpha_* c_*(i+t)}\widehat{u}_0(i,\xi)\md i\md t\bigg|\\
		&\le  \S_0 \widehat{\K_*}(0) \max_{[0,i_\dagger]}\omega \int_0^\infty \big|\widehat{u}_0(i,\xi)\big|\md i ~\bigg|\int_0^\infty e^{-(\boldsymbol{\sigma}+\mathbf{i} y+\alpha_* c_*) t}\md t\bigg|\\
		&\le  \S_0 \widehat{\K_*}(0) \max_{[0,i_\dagger]}\omega \int_0^\infty \int_\R\big|u_0(i,x)|\md x\md i ~ \big|\widehat{f}_{\boldsymbol{\sigma}}(y)\big|<C\big|\widehat{f}_{\boldsymbol{\sigma}}(y)\big|,~~~~~\text{uniformly in}~\xi\in\R,
	\end{aligned}
\end{equation*}
and
\begin{equation*}
|\widetilde{\mathcal{H}}(\boldsymbol{\sigma}+\mathbf{i} y,\xi)|=\bigg|\frac{\S_0\widehat{\K_*}(\xi)
	\widetilde{\omega}(\boldsymbol{\sigma}+\mathbf{i} y+\alpha_*c_*)}{1-\S_0\widehat{\K_*}(\xi)
	\widetilde{\omega}(\boldsymbol{\sigma}+\mathbf{i} y+\alpha_*c_*)}
\bigg|\le \frac{\S_0\widehat{\K_*}(0)}{m_{\boldsymbol{\sigma}}} \big| \widehat{g}_{\boldsymbol{\sigma}}(y)\big|,~~~~~\text{uniformly in}~\xi\in[-\varep,\varep],
\end{equation*}
it follows that 
\begin{equation}\label{4.27}
		\int_{-\infty}^{+\infty} |\widetilde{\mathcal{F}}(\boldsymbol{\sigma}+\mathbf{i} y,\xi)\widetilde{\mathcal{H}}(\boldsymbol{\sigma}+\mathbf{i} y,\xi)|\md y\le C
		\big\Vert \widehat{f}_{\boldsymbol{\sigma}}(y)\big\Vert_{L^2_y(\R)}\big\Vert \widehat{g}_{\boldsymbol{\sigma}}(y)\big\Vert_{L^2_y(\R)},~~~~~\text{uniformly in}~\xi\in[-\varep,\varep].
\end{equation}
Therefore, \eqref{W-2} is achieved.

	Thus, from \eqref{W-1} and \eqref{W-2} we have, on the one hand: $
		\displaystyle\lim_{{|\lambda|\to+\infty\,\boldsymbol{\delta}<\textnormal{Re}\lambda}}\widetilde{\mathcal{F}}(\lambda,\xi)\widetilde{\mathcal{H}}(\lambda,\xi)=0,$ {uniformly in} $\xi\in\R,$ and $\displaystyle 
\int_{-\infty}^{+\infty} |\widetilde{\mathcal{F}}(\boldsymbol{\delta}+\mathbf{i} y,\xi)\widetilde{\mathcal{H}}(\boldsymbol{\delta}+\mathbf{i} y,\xi)|\md y<$ is bounded {uniformly in} $\xi\in[-\varep,\varep]$,
with some $\boldsymbol{\delta}<0$ such that
$
	\displaystyle -\frac{\alpha_*c_*}{2}<\boldsymbol{\delta}<\min_{\xi\in[-\varep,\varep]}\Re\lambda(\xi);
$
 and on the other hand, 
 the function
\begin{equation}\label{4.30}
	\widehat{W}(t,\xi):=	\frac{1}{2\pi \mathbf{i}}\int_{\sigma-\mathbf{i}\infty}^{\sigma+\mathbf{i}\infty}e^{\lambda t}\widetilde{\mathcal{F}}(\lambda,\xi)\widetilde{\mathcal{H}}(\lambda,\xi)\md \lambda,~~~~~~~t>0,~\xi\in[-\varep,\varep],
\end{equation}
with some $\sigma>0$, is well-defined, and its Laplace transform for each $\xi\in[-\varep,\varep]$ is $\widetilde{\mathcal{F}}(\lambda,\xi)\widetilde{\mathcal{H}}(\lambda,\xi)$ for $\lambda\in\mathbb{C}$.

In addition, we notice that  $\widetilde{\mathcal{F}}(\lambda,\xi)\widetilde{\mathcal{H}}(\lambda,\xi)$ for each $\xi\in[-\varep,\varep]$ has a simple pole $\lambda(\xi)\in\mathbb{C}$, with residue
\begin{align*}
	c_{-1}(\xi)=\textnormal{Res} (\widetilde{\mathcal{F}}(\lambda,\xi)\widetilde{\mathcal{H}}(\lambda,\xi);\lambda(\xi))&= \textnormal{Res}\left(\widetilde{\mathcal{F}}(\lambda,\xi)\frac{\S_0\widehat{\K_*}(\xi)
		\widetilde{\omega}(\lambda+\alpha_*c_*)}{1-\S_0\widehat{\K_*}(\xi)
		\widetilde{\omega}(\lambda+\alpha_*c_*)};\lambda(\xi)\right)\\
	&=\widetilde{\mathcal{F}}(\lambda(\xi),\xi)\Big(\S_0\widehat{\K_*}(\xi)\int_0^\infty t e^{-(\lambda(\xi)+\alpha_*c_*) t} \omega(t)\md t\Big)^{-1}\\
	&=:\widetilde{\mathcal{F}}(\lambda(\xi),\xi)r(\xi)\in\mathbb{C}\backslash\{0\}.
\end{align*}
We then apply Lemma \ref{lem_residue} and derive that for each $\xi\in[-\varep,\varep]$,
\begin{equation}
	\label{W_hat}
\begin{aligned}
	\widehat{W}(t,\xi)&=e^{\lambda(\xi) t}c_{-1}(\xi)+\frac{1}{2\pi \mathbf{i}}\int_{\boldsymbol{\delta}-\mathbf{i}\infty}^{\boldsymbol{\delta}+\mathbf{i}\infty}e^{\lambda t}\widetilde{\mathcal{F}}(\lambda,\xi)\widetilde{\mathcal{H}}(\lambda,\xi)\md \lambda\\
	&=e^{\lambda(\xi) t}\widetilde{\mathcal{F}}(\lambda(\xi),\xi)r(\xi)+\underbrace{\frac{1}{2\pi \mathbf{i}}\int_{\boldsymbol{\delta}-\mathbf{i}\infty}^{\boldsymbol{\delta}+\mathbf{i}\infty}e^{\lambda t}\widetilde{\mathcal{F}}(\lambda,\xi)\widetilde{\mathcal{H}}(\lambda,\xi)\md \lambda}_{=:\Lambda(t;\xi)},~~~~~t>0,
\end{aligned}
\end{equation}
where $\Lambda(t;\xi)$ for $t>0$ and $\xi\in[-\varep,\varep]$  can be estimated by  following the lines of \eqref{4.27}:
\begin{equation}
	\label{4.32}
\begin{aligned}
	\big|\Lambda(t;\xi)\big|
	=\bigg|\frac{1}{2\pi }\int_{-\infty}^{+\infty}e^{(\boldsymbol{\delta}+\mathbf{i}y) t}\widetilde{\mathcal{F}}(\boldsymbol{\delta}+\mathbf{i}y,\xi)\widetilde{\mathcal{H}}(\boldsymbol{\delta}+\mathbf{i}y,\xi)\md y\bigg|\le \frac{e^{ \boldsymbol{\delta} t}}{2\pi m_{\boldsymbol{\delta}}}
	\Vert \widehat{f}_{\boldsymbol{\delta}}(y)\Vert_{L^2_y(\R)}\Vert \widehat{g}_{\boldsymbol{\delta}}(y)\Vert_{L^2_y(\R)}.
\end{aligned}
\end{equation}
Moreover, since $r(\xi)$ is holomorphic  for $\xi\in[-\varep,\varep]$, it follows that  $r(\xi)$ has the  expansion:
\begin{equation*}
	r(\xi)= \upsilon_0+\mathbf{i}\upsilon_1\xi+\upsilon_2\xi^2+\mathbf{i}\upsilon_3\xi^3+O(\xi^4),~~~~~~~\text{as}~\xi\to 0, 
\end{equation*}
with coefficients $\upsilon_0$ given in \eqref{notation} and $\upsilon_i\in\R\backslash\{0\}$ for $i=1,2,3$.

One also obtains from the inverse Laplace transform of \eqref{renewal eqn-Laplace} as well as \eqref{W_hat} that 
\begin{align*}
	\widehat{\Phi}(t,\xi)=\widehat{ \mathcal{F}}(t,\xi)+\widehat{W}(t,\xi)=\widehat{ \mathcal{F}}(t,\xi)+e^{\lambda(\xi) t}\widetilde{\mathcal{F}}(\lambda(\xi),\xi)r(\xi)+\Lambda(t,\xi),~~~~t>0,~~\xi\in[-\varep,\varep],
\end{align*}
with
\begin{align*}
	\Big|\widehat{ \mathcal{F}}(t,\xi)\Big|&=\Big|\S_0\widehat{\K_*}(\xi)\int_0^\infty\omega(i+t)e^{-\alpha_* c_*(i+t)}\widehat{u}_0(i,\xi)\md i\Big|\\
	&\le C e^{-\alpha_*c_* t}\int_0^\infty \big|\widehat{u}_0(i,\xi)\big|\md i= C e^{-\alpha_*c_* t}\int_0^\infty \bigg|\int_\R e^{-\mathbf{i}x\xi}u_0(i,x)\md x\bigg|\md i\le C e^{-\alpha_*c_* t}\Vert u_0\Vert_{L^1([0,i_\dagger)\times\R)}.
\end{align*}
Consequently,  we have  
\begin{align*}
\Big|\widehat{\Phi}(t,\xi)- e^{\lambda(\xi) t} \widetilde{\mathcal{F}}(\lambda(\xi),\xi)r(\xi)\Big|	\le C \big(e^{ \boldsymbol{\delta} t}+ e^{-\alpha_*c_* t}\big)\le C e^{ \boldsymbol{\delta} t},~~~~~~~t>0,~~\xi\in[-\varep,\varep].
\end{align*}
The proof of Proposition \ref{prop_4.1} is thereby complete.
\end{proof}

\subsubsection*{Proof of Theorem \ref{thm_3.1}}

	We first apply the inverse Laplace transform and the inverse Fourier transform in  \eqref{renewal eqn-Laplace}, it then follows that 
\begin{align*}
	\Phi(t,x)- \mathcal{F}(t,x)&=\frac{1}{2\pi}\int_\R e^{\mathbf{i}x\xi}\frac{1}{2\pi\mathbf{i}}\int_{\Gamma}e^{\lambda t}\widetilde{\mathcal{H}}(\lambda,\xi)  \widetilde{\mathcal{F}}(\lambda,\xi)\md \lambda\md\xi\\
	&=\frac{1}{2\pi}\int_{|\xi|\le\varep} e^{\mathbf{i}x\xi}\frac{1}{2\pi\mathbf{i}}\int_{\Gamma}e^{\lambda t}\widetilde{\mathcal{H}}(\lambda,\xi)  \widetilde{\mathcal{F}}(\lambda,\xi)\md \lambda\md\xi+  \frac{1}{2\pi}\int_{|\xi|>\varep} e^{\mathbf{i}x\xi}\frac{1}{2\pi\mathbf{i}}\int_{\Gamma}e^{\lambda t}\widetilde{\mathcal{H}}(\lambda,\xi)  \widetilde{\mathcal{F}}(\lambda,\xi)\md \lambda\md\xi\\
	&=: \I_1(t,x)+\I_2(t,x),~~~~~~~~~~~~~~t>0,~x\in\R,
\end{align*}
for some $\varep>0$. 
We observe that
\begin{align}
	\label{F1}
\mathcal{F}(t,x)=\S_0\int_0^\infty \omega(i+t) e^{-\alpha_*c_* (i+t)}\K_* *  u_0(i,x)\md i\le C\Vert u_0\Vert_{L^{\infty}([0,i_\dagger)\times\R)} e^{-\alpha_*c_*t},~~~~~t>0,~x\in\R.
\end{align}

To estimate $\Phi$, our main task then is to estimate $\I_1$ and $\I_2$. From now on, let $\varep>0$ be fixed as given in Lemma \ref{lem_analysis of lambda} (ii). Let $\lambda(\xi)\in\mathbb{C}$
solve equation \eqref{A4} with respect to $\xi\in\R$. In particular, \eqref{A4} admits a unique solution $\lambda(\xi)\in\mathbb{C}$ for each $\xi\in[-\varep,\varep]$.

\vskip 2mm

\noindent
\textbf{Step 1}. We first claim that there exists some constant $\eta<0$ such that
\begin{equation}
	\label{I_2-1}
	|\I_2(t,x)|\le C   \max_{[0,i_\dagger)}\big\Vert  u_0(i,x) \big\Vert_{H^1_x(\R)} e^{\eta t}~~~~~t>0,~~~x\in\R.
\end{equation}
 Indeed, consider $\xi\in\R\backslash[-\varep,\varep]$ and we  deduce from Lemma \ref{lem_analysis of lambda} (iii) that $\Re\lambda(\xi)<\bar\lambda<0$ for some $\bar{\lambda}<0$. Choose the integral line $\Gamma$ as 
$
	\Gamma=\big\{\eta+\mathbf{i}z,~z\in\R\big\},
$
with  $\max(\bar\lambda,-\alpha_* c_*)<\eta<0$,
so that all the singular points $\lambda(\xi)$ of $\widetilde{\mathcal{H}}(\cdot,\xi)$ for $|\xi|>\varep$ lie to the left of the curve $\Gamma$ in the complex plane $\mathbb{C}$. We then get that for $t>0$ and $x\in\R$,
\begin{align*}
\big|\I_2(t,x)\big|&=\bigg|	\frac{1}{2\pi}\int_{|\xi|>\varep} e^{\mathbf{i}x\xi}\bigg(\frac{1}{2\pi\mathbf{i}}\int_{\Gamma}e^{\lambda t}\widetilde{\mathcal{H}}(\lambda,\xi)  \widetilde {\mathcal{F}}(\lambda,\xi)~\md \lambda\bigg)
\md\xi\bigg|\\
&=\bigg|	\frac{1}{2\pi}\int_{|\xi|>\varep} e^{\mathbf{i}x\xi}\bigg(\frac{1}{2\pi}\int_{\R}e^{(\eta+\mathbf{i}z) t}\widetilde{\mathcal{H}}(\eta+\mathbf{i}z,\xi)  \widetilde {\mathcal{F}}(\eta+\mathbf{i}z,\xi)\md z\bigg)
\md\xi\bigg|\\
&\le \frac{e^{\eta t}}{4\pi^2}\int_{|\xi|>\varep} \frac{1}{1+|\xi|}(1+|\xi|)	\left(\int_\R  \big|\widetilde F(\eta+\mathbf{i}z,\xi)\big| \big|\widetilde{\mathcal{H}}(\eta+\mathbf{i}z,\xi)\big|\md z\right)\md\xi\\
&\le \frac{e^{\eta t}}{4\pi^2}\bigg(\int_{|\xi|>\varep} \frac{1}{(1+|\xi|)^2}\md \xi\bigg)^\frac{1}{2} \bigg(\int_{|\xi|>\varep} 
\bigg((1+|\xi|)\underbrace{\int_\R  \big|\widetilde F(\eta+\mathbf{i}z,\xi)\big| \big|\widetilde{\mathcal{H}}(\eta+\mathbf{i}z,\xi)\big|\md z}_{=:\mathcal{B}}\bigg)^2
\md \xi\bigg)^\frac{1}{2},
\end{align*}
where we have applied the H\"older inequality in the last step.  We notice that 
\begin{equation*}
	\bigg(\int_{|\xi|>\varep} \frac{1}{(1+|\xi|)^2}\md \xi\bigg)^\frac{1}{2}<+\infty.
\end{equation*}
It remains to estimate the integral $\mathcal{B}$. First, the H\"older inequality gives that 
\begin{equation}
	\label{estimate B}
	\mathcal{B}\le \bigg(
	\int_\R  \big|\widetilde F(\eta+\mathbf{i}z,\xi)\big|^2 \md z
	\bigg)^{\frac{1}{2}} \bigg(
	\int_\R  \big|\widetilde{\mathcal{H}}(\eta+\mathbf{i}z,\xi)\big|^2\md z
	\bigg)^{\frac{1}{2}}.
\end{equation}
Let $f_{\eta}(t)$ and $g_{\eta}(t)$ be as defined in \eqref{func def_f g}, by taking particularly   $\boldsymbol{\sigma}=\eta$, then their Fourier transforms $\widehat{f}_{\eta}(z)$ and $\widehat{g}_{\eta}(z)$ belong to $L^2_z(\R)$.
We have
\begin{align*}
		|\widetilde{\mathcal{F}}(\eta+\mathbf{i} z,\xi)|&=\bigg|\S_0\widehat{\K_*}(\xi)\int_0^\infty e^{-(\eta+\mathbf{i} z) t}\int_0^\infty \omega(i+t)e^{-\alpha_* c_*(i+t)}\widehat{u}_0(i,\xi)\md i\md t\bigg|\\
		&\le  \S_0 \widehat{\K_*}(0) \max_{[0,i_\dagger]}\omega \int_0^\infty \big|\widehat{u}_0(i,\xi)\big|\md i ~\bigg|\int_0^\infty e^{-(\eta+\mathbf{i} z+\alpha_* c_*) t}\md t\bigg|\\
		&<C\max_{[0,i_\dagger)}\widehat{u}_0(\cdot,\xi)\big|\widehat{f}_{\eta}(z)\big|,~~~~~\text{uniformly in}~\xi\in\R.
\end{align*}
Therefore, for each $\xi\in\R\backslash[-\varep,\varep]$,
\begin{equation}\label{4.34}
	\begin{aligned}
		\bigg(
		\int_\R  \big|\widetilde F(\eta+\mathbf{i}z,\xi)\big|^2 \md z
		\bigg)^{\frac{1}{2}}\le C\max_{[0,i_\dagger)}\widehat{u}_0(\cdot,\xi)\bigg(\int_\R \big|\widehat{f}_{\eta}(z)\big|^2\md z\bigg)^{\frac{1}{2}}=C\max_{[0,i_\dagger)}\widehat{u}_0(\cdot,\xi)\big\Vert  \widehat{f}_{\eta}(z)\big\Vert_{L^2_z(\R)}.
	\end{aligned}
\end{equation}
On the other hand, it follows from $\max(\bar\lambda,-\alpha_* c_*)<\eta<0$ that there is 
no root of \eqref{A4} on the line $\Re\lambda=\eta$ for all $\xi\in\R\backslash[-\varep,\varep]$.
We then infer from  \eqref{4.25} together with $|\widehat{\K_*}(\xi)|\le \widehat{\K_*}(0)$ for all $\xi\in\R$ that 
$$m_{\eta}:=\inf_{y\in\R,\xi\in\R\backslash[-\varep,\varep]}\big|1-\S_0\widehat{\K_*}(\xi)
\widetilde{\omega}(\eta+\mathbf{i}z+\alpha_*c_*)\big|>0.$$
This implies that
\begin{equation*}
	|\widetilde{\mathcal{H}}(\eta+\mathbf{i}z,\xi)|=\bigg|\frac{\S_0\widehat{\K_*}(\xi)
		\widetilde{\omega}(\eta+\mathbf{i}z+\alpha_*c_*)}{1-\S_0\widehat{\K_*}(\xi)
		\widetilde{\omega}(\eta+\mathbf{i}z+\alpha_*c_*)}
	\bigg|\le \frac{\S_0\widehat{\K_*}(0)}{m_{\eta}} \big| \widehat{g}_{\eta}(z)\big|,~~~~~\text{uniformly in}~\xi\in\R\backslash[-\varep,\varep],
\end{equation*}
which gives
\begin{equation}\label{4.35}
\begin{aligned}
	\bigg(
	\int_\R  \big|\widetilde{\mathcal{H}}(\eta+\mathbf{i}z,\xi)\big|^2\md z
	\bigg)^{\frac{1}{2}}\le \frac{\S_0\widehat{\K_*}(0)}{m_{\eta}}\bigg(\int_\R \big|\widehat{g}_{\eta}(z)\big|^2\md z\bigg)^{\frac{1}{2}}=\frac{\S_0\widehat{\K_*}(0)}{m_{\eta}}\big\Vert  \widehat{g}_{\eta}(z)\big\Vert_{L^2_z(\R)}.
\end{aligned}
\end{equation}
Plugging \eqref{4.34} and \eqref{4.35} into the estimate \eqref{estimate B} yields
\begin{equation*}
	\mathcal{B}\le C\max_{[0,i_\dagger)}\widehat{u}_0(\cdot,\xi)\big\Vert  \widehat{f}_{\eta}(z)\big\Vert_{L^2_z(\R)}  \big\Vert  \widehat{g}_{\eta}(z)\big\Vert_{L^2_z(\R)}\le C\max_{[0,i_\dagger)}\widehat{u}_0(\cdot,\xi),
\end{equation*}
whence
\begin{equation*}
	\bigg(\int_{|\xi|>\varep} (1+|\xi|)^2 \mathcal{B}^2\md\xi\bigg)^{\frac{1}{2}}\le C  \bigg(\int_{|\xi|>\varep} \Big((1+|\xi|)\big|\max_{[0,i_\dagger)}\widehat{u}_0(\cdot,\xi)\big|\Big)^2\md\xi\bigg)^{\frac{1}{2}}\le C\max_{[0,i_\dagger)}\big\Vert u_0(i,x)\big\Vert_{H^1_x(\R)}.
\end{equation*}
Our claim \eqref{I_2-1} is achieved.
\vskip 2mm

\noindent
\textbf{Step 2}. Let us now estimate the integral $\I_1$. We note from \eqref{4.30} and \eqref{W_hat} that
\begin{align*}
	\I_1(t,x)&=\frac{1}{2\pi}\int_{|\xi|\le\varep} e^{\mathbf{i}x\xi}\frac{1}{2\pi\mathbf{i}}\int_{\Gamma}e^{\lambda t}\widetilde{\mathcal{H}}(\lambda,\xi)  \widetilde {\mathcal{F}}(\lambda,\xi)\md \lambda\md\xi=\frac{1}{2\pi}\int_{|\xi|\le\varep} e^{\mathbf{i}x\xi}\widehat{W}(t,\xi)\md \xi\\
	&=\underbrace{\frac{1}{2\pi}\int_{|\xi|\le\varep} r(\xi) e^{\mathbf{i}x\xi} e^{\lambda(\xi) t}\widetilde{\mathcal{F}}(\lambda(\xi),\xi)\md \xi}_{=:\mathcal{I}_{1,1}(t,x)}+\underbrace{\frac{1}{2\pi}\int_{|\xi|\le\varep} e^{\mathbf{i}x\xi} \Lambda(t,\xi)\md \xi}_{\mathcal{I}_{1,2}(t,x)},~~~~~~~~t>0,~~x\in\R.
\end{align*}
Here, it is easy to infer from \eqref{4.32} that  
\begin{align}\label{I_12}
	\big|\mathcal{I}_{1,2}(t,x)\big|=\bigg|\frac{1}{2\pi}\int_{|\xi|\le\varep} e^{\mathbf{i}x\xi} \Lambda(t,\xi)\md \xi\bigg|\le C|\Lambda(t,\xi)|\le C e^{\boldsymbol{\delta}t},~~~~~~t>0,~\text{uniformly in}~x\in\R, 
\end{align}
for $\boldsymbol{\delta}<0$ given in Proposition \ref{prop_4.1}. Moreover, we observe from  \eqref{F_fourier} that $\widehat{ \mathcal{F}}(t,\xi)\equiv0$ for $t>i_\dagger$, uniformly in $\xi\in\R$, therefore
\begin{align*}
	\mathcal{I}_{1,1}(t,x)&=\frac{1}{2\pi}\int_{|\xi|\le\varep}  r(\xi) e^{\mathbf{i}x\xi} e^{\lambda(\xi) t}\widetilde{\mathcal{F}}(\lambda(\xi),\xi)\md \xi\\
	&=\frac{1}{2\pi}\int_{|\xi|\le\varep}  r(\xi) e^{\mathbf{i}x\xi}  e^{\lambda(\xi) t}\bigg(\int_0^\infty e^{-\lambda(\xi)s}\widehat{ \mathcal{F}}(s,\xi)\md s\bigg) \md \xi\\
	&=\frac{1}{2\pi}\int_{|\xi|\le\varep}  r(\xi)  \int_0^{i_\dagger} e^{\lambda(\xi)(t-s)}\int_\R e^{\mathbf{i}(x-y)\xi} \mathcal{F}(s,y)\md y\md s \md \xi,~~~~~~t>0,~x\in\R.
\end{align*}
For $t>i_\dagger$, we further derive from $\mathcal{F}(t,x)=0$ for $t\ge i_\dagger$, uniformly in $x\in\R$,  and  Fubini theorem that 
\begin{align*}
	\mathcal{I}_{1,1}(t,x)&=\frac{1}{2\pi}\int_{|\xi|\le\varep}  r(\xi)  \int_0^t e^{\lambda(\xi)(t-s)}\int_\R e^{\mathbf{i}(x-y)\xi} \mathcal{F}(s,y)\md y\md s \md \xi\\	&=\frac{1}{2\pi}\int_0^t \int_\R \int_{|\xi|\le\varep}  r(\xi)  e^{\lambda(\xi)(t-s)}e^{\mathbf{i}(x-y)\xi}    \mathcal{F}(s,y) \md \xi\md y\md s\\
	&= \bigg(\underbrace{ \frac{1}{2\pi}\int_{|\xi|\le\varep} r(\xi) e^{\lambda(\xi)t} e^{\mathbf{i}x\xi}   \md \xi}_{=:\mathcal{G}_0(t,x)}\bigg) *_x *_t \mathcal{F}(t,x)
	,~~~~~~t>i_\dagger,~~x\in\R.
\end{align*}

For some small parameter $\gamma_*\in(0,1/2)$ to be determined later, we choose $T>i_\dagger$ sufficiently large such that $T^{-\gamma_*}<\varep$. For all $t\ge T$, we  divide the domain of integration into two situations: either $|\xi|\le t^{-\gamma_*}$ or $t^{-\gamma_*}\le |\xi|\le\varep$, then
\begin{align*}
	\mathcal{G}_0(t,x)&=\underbrace{\frac{1}{2\pi}\int_{|\xi|\le t^{-\gamma_*}}   r(\xi) e^{\mathbf{i}x\xi} e^{\lambda(\xi)t} \md \xi}_{=:\mathcal{G}(t,x)}+ \underbrace{\frac{1}{2\pi}\int_{t^{-\gamma_*}\le |\xi|\le\varep}  e^{\mathbf{i}x\xi} r(\xi) e^{\lambda(\xi)t} \md \xi}_{=:\mathcal{G}_1(t,x)},~~~~~~t\ge T,~x\in\R.
\end{align*} 

Let us now estimate $\mathcal{G}_1$. Set $\lambda^*(\xi):=-\mathbf{i}c_*\xi- b\xi^2$. By noticing that $|r(\xi)|$ is bounded for $|\xi|\le \varep$, we then infer from \eqref{rough lambda} and from the change of variable $\zeta=\xi\sqrt{t}$ that, up to increasing $T$,
\begin{equation*}
	\begin{aligned} 
		|\mathcal{G}_1(t,x)|&=\frac{1}{2\pi}\left|\int_{t^{-\gamma_*}\le |\xi|\le\varep} r(\xi) e^{\mathbf{i}x\xi} e^{\lambda(\xi)t} \md \xi \right|\lesssim\int_{t^{-\gamma_*}\le |\xi|\le\varep}  \left| e^{(\lambda(\xi)-\lambda^*(\xi))t+\lambda^*(\xi)t} \right| \md \xi \\
		&\lesssim	\int_{t^{-\gamma_*}\le |\xi|\le\varep}   e^{-\frac{b}{2}\xi^2t} \md \xi= t^{-\frac{1}{2}} \int_{t^{\frac{1}{2}-\gamma_*}\le |\zeta|\le \varep t^{\frac{1}{2}}} e^{-\frac{b}{2}\zeta^2}\md\zeta \lesssim e^{-\frac{b}{2}t^{1-2\gamma_*}},~~~~~t\ge T,~~\text{uniformly in}~x\in\R.
	\end{aligned}
\end{equation*}
Since   $\K_*\in L^1(\R)$ and  $u_0$ is continuous and compactly supported in $[0,i_\dagger)\times\R$, since $\omega$ is nonnegative, bounded and belongs to $ L^1([0,i_\dagger])$, it follows from the Young's convolution inequality that, up to increasing $T$, 
\begin{equation}\label{G1}
	\begin{aligned}
	\mathcal{G}_1**\mathcal{F}(t,x)	=&\int_\R\int_0^t \mathcal{G}_1(t-s,x-y)\mathcal{F}(s,y)\md s\md y \\
		\lesssim &\int_\R\int_0^{i_\dagger} e^{-\frac{b}{2}(t-s)^{1-2\gamma_*}} \bigg(\S_0\int_0^\infty \omega(i+s) e^{-\alpha_*c_* (i+s)}\K_* *  u_0(i,y)\md i\bigg) \md s\md y\\
		\lesssim&\int_0^{i_\dagger} e^{-\frac{b}{2}(t-s)^{1-2\gamma_*}} \bigg(\S_0\int_0^\infty \omega(i+s) e^{-\alpha_*c_* (i+s)}\md i\bigg) \md s \int_\R\K_* *  \max_{[0,i_\dagger)}u_0(i,y)\md y\\
		\lesssim & \Vert \max_{[0,i_\dagger)}u_0(i,x)\Vert_{L^{1}_x(\R)}\int_0^{i_\dagger} e^{-\frac{b}{2}(t-s)^{1-2\gamma_*}}e^{-\alpha_*c_* s}\md s\\
		\lesssim& \Vert u_0\Vert_{L^{\infty}([0,i_\dagger)\times\R)}e^{-\frac{b}{2}t^{1-2\gamma_*}},~~~~~~~t\ge T,~\text{uniformly in}~x\in\R. 
	\end{aligned}
\end{equation}

It is left to estimate $\mathcal{G}$ for all $t\ge T$ and $x\in\R$, up to increasing $T$ if needed. In the range of $|\xi|\le t^{-\gamma_*}$,  we derive from \eqref{precise lambda} and \eqref{r(xi)} that
\begin{equation*}
	\lambda(\xi)+\mathbf{i}c_*\xi=- b\xi^2+\mathbf{i}k\xi^3 +O(\xi^4),~~~~~\text{with}~b>0,~k\in\R,
\end{equation*}
and 
\begin{equation*}	r(\xi)=\upsilon_0+\mathbf{i}\upsilon_1\xi+\upsilon_2\xi^2+\mathbf{i}\upsilon_3\xi^3+O(\xi^4),~~~~~\text{with}~\upsilon_0>0,~\upsilon_i\in\R\backslash\{0\}.
\end{equation*}
By the change of variable $\zeta=\xi\sqrt{t}$, we have
\begin{equation*}\label{G-kernel}
	\begin{aligned}
		\mathcal{G}(t,x)&=\frac{1}{2\pi}\int_{|\xi|\le t^{-\gamma_*}} r(\xi) e^{\mathbf{i}x\xi} e^{\lambda(\xi)t} \md \xi=\frac{1}{2\pi}\int_{|\xi|\le t^{-\gamma_*}} r(\xi) e^{\mathbf{i}(x-c_*t)\xi} e^{(\lambda(\xi)+\mathbf{i}c_*\xi )t} \md \xi\\
		&=\frac{1}{2\pi\sqrt{t}}\int_{-t^{\frac{1}{2}-\gamma_*}}^{t^{\frac{1}{2}-\gamma_*}}r\Big(\frac{\zeta}{\sqrt{t}}\Big)\exp\bigg(\underbrace{\mathbf{i}\frac{\zeta(x-c_*t)}{\sqrt{t}}-b\zeta^2+\mathbf{i}\frac{k\zeta^3}{\sqrt{t}}+O\Big(\frac{\zeta^4}{t}\Big)}_{=:\Psi(t,x-c_*t,\zeta)}\bigg)
		\md\zeta.
	\end{aligned}
\end{equation*}
In order to proceed the computation, let us define the complex line
\begin{equation}\label{Gamma*}
	\Gamma_*:=\left\{\zeta=\eta+\mathbf{i}\frac{x-c_*t}{2b\sqrt{t}},~|\eta|\le t^{\frac{1}{2}-\gamma_*}\right\},
\end{equation}
along which $\Psi$ can be reformulated as
\begin{align}\label{Psi_psi}
	\Psi(t,x-c_*t,\eta)= -b\eta^2-\frac{(x-c_*t)^2}{4bt}+\underbrace{\mathbf{i}\frac{k\big( \eta+\mathbf{i}\frac{x-c_*t}{2b\sqrt{t}}\big)^3}{\sqrt{t}}+O\bigg(\frac{\big(\eta+\mathbf{i}\frac{x-c_*t}{2b\sqrt{t}}\big)^4}{t}\bigg)}_{=:\psi(t,x-c_*t,\eta)}. 
\end{align}
We observe that $\Psi(t,x-c_*t,\eta)$ can be dominated by $-b\eta^2-\frac{(x-c_*t)^2}{4bt}$, provided that $\psi(t,x-c_*t,\eta)$ is negligible, that is,
\begin{equation}\label{cdn}
	\frac{k\big( \eta+\mathbf{i}\frac{x-c_*t}{2b\sqrt{t}}\big)^3}{\sqrt{t}}\ll 1,~~~~\frac{\big(\eta+\mathbf{i}\frac{x-c_*t}{2b\sqrt{t}}\big)^4}{t}\ll1,~~~~~\text{for}~|\eta|\le t^{\frac{1}{2}-\gamma_*}.
\end{equation}
In other words, it is necessary to set some constraints on the size of $x-c_*t$.
To achieve \eqref{cdn}, we require
\begin{equation*}
\frac{\eta}{ t^{\frac{1}{6}}}\ll 1,~~~~\frac{|x-c_*t|}{t^{\frac{2}{3}}}\ll1, ~~~~~\text{for}~|\eta|\le t^{\frac{1}{2}-\gamma_*},
\end{equation*}
which can be fulfilled by imposing
\begin{equation}
	\label{cdn_gamma delta}
	\frac{3}{8}<\gamma_*<\frac{1}{2},~~~~~~\frac{|x-c_*t|}{\sqrt{t}}\le t^{\varsigma}~~~\text{with}~~\varsigma\in\left(0,\frac{1}{2}-\gamma_*\right).
\end{equation}
In what follows, our analysis will be divided into two cases: $|x-c_*t|\le t^{\frac{1}{2}+\varsigma}$ and $|x-c_*t|\ge t^{\frac{1}{2}+\varsigma}$. 

\vspace{2mm}
\noindent
{\bf Case 1. $|x-c_*t|\le t^{\frac{1}{2}+\varsigma}$.} Define 
\begin{equation*}
	\Gamma_\pm:=\left\{\zeta=\pm t^{\frac{1}{2}-\gamma_*}+\mathbf{i}\beta,~0\le \beta \le \frac{x-c_*t}{2b\sqrt t}\right\}.
\end{equation*} 
Therefore, together with the line $\Gamma_*$ given in \eqref{Gamma*}, we arrive at 
$$	\mathcal{G}(t,x)= \frac{1}{2\pi\sqrt{t}}\int_{-t^{\frac{1}{2}-\gamma_*}}^{t^{\frac{1}{2}-\gamma_*}}r(\zeta/\sqrt{t})\exp\big(\Psi(t,x-c_*t,\zeta)\big)\md \zeta= \frac{1}{2\pi\sqrt{t}}\int_{\Gamma_-\cup\Gamma_*\cup\Gamma_+}r(\zeta/\sqrt{t})\exp\big(\Psi(t,x-c_*t,\zeta)\big)\md \zeta.
$$
On $\Gamma_\pm$, since 
\begin{equation*}
	\Psi(t,x-c_*t,\zeta)=\pm\mathbf{i}\frac{x-c_*t}{t^\gamma}-\frac{\beta(x-c_*t)}{\sqrt{t}}-b(\pm t^{\frac{1}{2}-\gamma_*}+\mathbf{i}\beta)^2+O\left(t^{1-3\gamma_*}\right),
\end{equation*}
it follows from our choices of $\gamma_*$ and $\varsigma$ in \eqref{cdn_gamma delta} that
\begin{equation*}
	\Re 	\Psi(t,x-c_*t,\zeta)=-\frac{\beta(x-c_*t)}{\sqrt{t}}-b\left(t^{1-2\gamma_*}-\beta^2\right)+O\left(t^{1-3\gamma_*}\right)\le-b t^{1-2\gamma_*}\big(1+o(1)\big),
\end{equation*}
Moreover,  from $\displaystyle
		r\Big(\frac{\zeta}{\sqrt{t}}\Big)=\upsilon_0+\mathbf{i}\upsilon_1\frac{\zeta}{\sqrt{t}}+O\left(\frac{\zeta^2}{t}\right),\ |\zeta|=|\xi\sqrt{t}|\le t^{\frac{1}{2}-\gamma_*},\ t\ge T,$
up to increasing $T$,  we deduce that $|r(\zeta/\sqrt{t})|=|\upsilon_0+O(t^{-\gamma_*})|$ is uniformly bounded for $\zeta\in\Gamma_\pm$. As a consequence,
\begin{equation}\label{case1_1}
	\frac{1}{2\pi\sqrt{t}}\left|\int_{\Gamma_-\cup\Gamma_+}r(\zeta/\sqrt{t})\exp\big(\Psi(t,x-c_*t,\zeta)\big)\md \zeta\right|\lesssim t^{-\gamma_*} e^{-b t^{1-2\gamma_*}}.
\end{equation}

It remains to look at $\Gamma_*=\{\zeta=\eta+\mathbf{i}\frac{x-c_*t}{2b\sqrt{t}},~|\eta|\le t^{\frac{1}{2}-\gamma_*}\}$. We notice that on $\Gamma_*$, 
\begin{align*}
	r\Big(\frac{\zeta}{\sqrt{t}}\Big)=\upsilon_0+\mathbf{i}\upsilon_1\frac{\eta+\mathbf{i}\frac{x-c_*t}{2b\sqrt{t}}}{\sqrt{t}}+\upsilon_2\left(\frac{\eta+\mathbf{i}\frac{x-c_*t}{2b\sqrt{t}}}{\sqrt{t}}\right)^2+\mathbf{i}\upsilon_3\left(\frac{\eta+\mathbf{i}\frac{x-c_*t}{2b\sqrt{t}}}{\sqrt{t}}\right)^3+O\bigg(\frac{\big(\eta+\mathbf{i}\frac{x-c_*t}{2b\sqrt{t}}\big)^4}{t^2}\bigg),
\end{align*}
which, by denoting $$\mathcal{A}(t,x):=\frac{x-c_*t}{2b\sqrt{t}}\in[-t^\varsigma, t^\varsigma],$$
can be recast as
 \begin{align*}
 	r\Big(\frac{\zeta}{\sqrt{t}}\Big)\!&=\!\upsilon_0\!-\!\frac{\upsilon_1}{\sqrt{t}} \mathcal{A}\!-\!\frac{\upsilon_2}{t}\mathcal{A}^2+\frac{\upsilon_3}{t^{\frac{3}{2}}}\mathcal{A}^3
 	\!+\!\mathbf{i}\eta\Big(\frac{\upsilon_1}{\sqrt{t}}+\frac{2\upsilon_2}{t}\mathcal{A}\!-\!\frac{3\upsilon_3}{t^{\frac{3}{2}}}\mathcal{A}^2\Big)+\eta^2\Big(\frac{\upsilon_2}{t}\!-\!\frac{3\upsilon_3}{t^{\frac{3}{2}}}\mathcal{A}
 	\Big)+\mathbf{i}\eta^3\frac{\upsilon_3}{t^{\frac{3}{2}}}+O(t^{-4\gamma_*})\\
 	&=\!\upsilon_0\!+\!\mathbf{i}\eta\frac{\upsilon_1}{\sqrt{t}}+O\big(t^{\varsigma-\frac{1}{2}}\big).
 \end{align*}
Moreover, the function $\psi$ given in \eqref{Psi_psi} has the following expression on $\Gamma_*$:
\begin{align*}
	\psi(t,x-c_*t,\eta)&=\mathbf{i}\frac{k\big( \eta+\mathbf{i}\frac{x-c_*t}{2b\sqrt{t}}\big)^3}{\sqrt{t}}+O\bigg(\frac{\big(\eta+\mathbf{i}\frac{x-c_*t}{2b\sqrt{t}}\big)^4}{t}\bigg)\\
	&=\frac{k}{\sqrt{t}}\mathcal{A}^3-\mathbf{i}\eta\frac{3k}{\sqrt{t}}\mathcal{A}^2-\eta^2\frac{3k}{\sqrt{t}}\mathcal{A}+\mathbf{i}\eta^3\frac{k}{\sqrt{t}}+O(t^{1-4\gamma_*})\lesssim O(t^{1-3\gamma_*})=o(1)
\end{align*}
uniformly in $|x-c_*t|\le t^{\frac{1}{2}+\varsigma}$ and $|\eta|\le t^{\frac{1}{2}-\gamma_*}$, due to $\gamma_*\in(3/8,1/2)$ and $0<\varsigma<1/2-\gamma_*$.

Therefore, we have
\begin{align*}
		\frac{1}{2\pi\sqrt{t}}\int_{\Gamma_*}&	r\Big(\frac{\zeta}{\sqrt{t}}\Big)\exp\big(\Psi(t,x-c_*t,\zeta)\big)\md \zeta\\
		=&\frac{1}{2\pi\sqrt{t}}\int_{-t^{\frac{1}{2}-\gamma_*}}^{t^{\frac{1}{2}-\gamma_*}}r\left(\frac{\eta+\mathbf{i}\frac{x-c_*t}{2b\sqrt{t}}}{\sqrt{t}}\right)
		\exp\big(-b\eta^2-\frac{(x-c_*t)^2}{4bt}+\psi(t,x-c_*t,\eta)\big)\md \eta\\
		=&\frac{1}{2\pi\sqrt{t}}\int_{-t^{\frac{1}{2}-\gamma_*}}^{t^{\frac{1}{2}-\gamma_*}}\bigg(\!\upsilon_0\!+\!\mathbf{i}\eta\frac{\upsilon_1}{\sqrt{t}}+O\big(t^{\varsigma-\frac{1}{2}}\big)\!\bigg)\bigg(1+
	\frac{k}{\sqrt{t}}\mathcal{A}^3-\mathbf{i}\eta\frac{3k}{\sqrt{t}}\mathcal{A}^2-\eta^2\frac{3k}{\sqrt{t}}\mathcal{A}\\
	&~~~~~~~~~~~~~~~~~~~~~~~~~~~~~~~~~~~~~  ~~~~~~~~~~~~~~~~~~~~+\mathbf{i}\eta^3\frac{k}{\sqrt{t}}+O(t^{1-4\gamma_*})\bigg)e^{-b\eta^2-\frac{(x-c_*t)^2}{4bt}}\md \eta\\
		=&\frac{1}{2\pi\sqrt{t}}  \int_{-t^{\frac{1}{2}-\gamma_*}}^{t^{\frac{1}{2}-\gamma_*}}\bigg( \!\upsilon_0+\frac{k\upsilon_0}{\sqrt{t}}\mathcal{A}^3 -\frac{3k\upsilon_0}{\sqrt{t}}\mathcal{A}\eta^2 +O(t^{\varsigma-\frac{1}{2}})
		\bigg) e^{-b\eta^2}\md \eta~ e^{-\frac{(x-c_*t)^2}{4bt}	} 
		\\
		=&\frac{1}{2\pi\sqrt{t}}\bigg( \!\upsilon_0+\frac{k\upsilon_0}{\sqrt{t}}\mathcal{A}^3 -\frac{3k\upsilon_0}{2b\sqrt{t}}\mathcal{A} +O(t^{\varsigma-\frac{1}{2}})
		\bigg)\sqrt{\frac{	\pi}{ b }} e^{-\frac{(x-c_*t)^2}{4bt}	}\\
	=&\frac{1}{\sqrt{2bt}}\bigg( \!\upsilon_0\!-\frac{3k\upsilon_0}{2b\sqrt{t}}\mathcal{A}+\frac{k\upsilon_0}{\sqrt{t}}\mathcal{A}^3 
	+O(t^{\varsigma-\frac{1}{2}})\bigg) \frac{1}{\sqrt{2\pi}} e^{-\frac{(x-c_*t)^2}{4bt}	}\\
	=&\left[\frac{\upsilon_0}{\sqrt{2bt}}+\frac{1}{{2bt} }\left(-\frac{3k\upsilon_0}{2b}\left(\frac{x-c_*t}{\sqrt{2bt}}\right)	+\frac{k\upsilon_0}{2b}\left(\frac{x-c_*t}{\sqrt{2bt}}\right)	^3\right)\right]\frac{1}{\sqrt{2\pi}}e^{-\frac{(x-c_*t)^2}{4bt}	}+O(t^{\varsigma-1}) e^{-\frac{(x-c_*t)^2}{4bt}	}\\
	=&:\bigg[\frac{\upsilon_0}{\sqrt{2bt}}+\frac{k\upsilon_0}{4b^2t}\mathcal{P}\left(\frac{x-c_*t}{\sqrt{2bt}}\right)	\bigg]\mathscr{G}\left(\frac{x-c_*t}{\sqrt{2bt}}\right)	+O(t^{\varsigma-1}) e^{-\frac{(x-c_*t)^2}{4bt}	},
\end{align*}
 where we have defined 
 $\mathcal{P}(x):=-3x	+x^3$ for $x\in\R$.  
Consequently, combining the computation above with \eqref{case1_1}, we derive that
\begin{align*}
	\mathcal{G}(t,x)&= \frac{1}{2\pi\sqrt{t}}\int_{\Gamma_-\cup\Gamma_*\cup\Gamma_+}r(\zeta/\sqrt{t})\exp\big(\Psi(t,x-c_*t,\zeta)\big)\md \zeta\\
	&=\bigg[\frac{\upsilon_0}{\sqrt{2bt}}+\frac{k\upsilon_0}{4b^2t}\mathcal{P}\left(\frac{x-c_*t}{\sqrt{2bt}}\right)	\bigg]\mathscr{G}\left(\frac{x-c_*t}{\sqrt{2bt}}\right)	+\mathcal{R}(t,x).
\end{align*}	
This gives \eqref{estimate_1_Gaussian} and \eqref{remaining term R}.

\vspace{2mm}
\noindent
{\bf Case 2. $|x-c_*t|\ge t^{\frac{1}{2}+\varsigma}$.} In this case, the main ingredient of the proof is very similar to the previous case. However, as the term $\psi(t,x-c_*t,\eta)$ in \eqref{Psi_psi} cannot be negligible, the line $\Gamma_*$ defined in \eqref{Gamma*} is not helpful anymore. Instead, we shall make other options of the integration lines.
First of all, let us deal with the situation of $x-c_*t\ge t^{\frac{1}{2}+\varsigma}$. Define the vertical segments:
$
	\Gamma_\pm:=\left\{\zeta=\pm t^{\frac{1}{2}-\gamma_*}+\mathbf{i}\beta,~0\le\beta\le 1\right\},
$
and the horizontal segment:
$
	\overline \Gamma:=\left\{\zeta=\eta+\mathbf{i},~|\eta|\le t^{\frac{1}{2}-\gamma_*}
	\right\}.
$
On $\Gamma_\pm$, we have 
\begin{equation*}
	\Psi(t,x-c_*t,\zeta)=\pm\mathbf{i}\frac{x-c_*t}{t^{\gamma_*}}-\frac{\beta(x-c_*t)}{\sqrt{t}}-b(\pm t^{\frac{1}{2}-\gamma_*}+\mathbf{i}\beta)^2+O(t^{1-3\gamma_*}),
\end{equation*}
whence 
\begin{equation}\label{eqn-1}
\Re\Psi(t,x-c_*t,\zeta)\le -b t^{1-2\gamma_*}\left(1+o(1)\right).	
\end{equation}
On $\overline \Gamma$, we notice
\begin{equation*}
		\Psi(t,x-c_*t,\zeta)=-\frac{x-c_*t}{\sqrt{t}}+\mathbf{i}\frac{\eta(x-c_*t)}{\sqrt{t}}-b(\eta+\mathbf{i})^2+\mathbf{i}\frac{k(\eta+\mathbf{i})^3}{\sqrt{t}}+O\Big(\frac{(\eta+\mathbf{i})^4}{t}\Big),
\end{equation*}
therefore
\begin{equation}\label{eqn-2}
	\Re	\Psi(t,x-c_*t,\zeta)\le -\frac{x-c_*t}{\sqrt{t}}+O(1)\le -t^{\varsigma}+O(1).
\end{equation}
We also observe that
\begin{equation*}
	|r(\zeta/\sqrt{t})|\le \upsilon_0+O(t^{-\gamma_*})~~~~\text{for}~\zeta\in\Gamma_-\cup\overline\Gamma\cup\Gamma_+.
\end{equation*}
As a consequence, it follows from \eqref{eqn-1} and \eqref{eqn-2} that
\begin{equation}
	\label{eqn-12}
	\begin{aligned}
		\left|\mathcal{G}(t,x)\right|&= \frac{1}{2\pi\sqrt{t}}\bigg|\int_{-t^{\frac{1}{2}-\gamma_*}}^{t^{\frac{1}{2}-\gamma_*}}r(\zeta/\sqrt{t})\exp\big(\Psi(t,x-c_*t,\zeta)\big)\md \zeta\bigg|\\
		&= \frac{1}{2\pi\sqrt{t}}\left|\int_{\Gamma_-\cup\overline\Gamma\cup\Gamma_+}r(\zeta/\sqrt{t})\exp\big(\Psi(t,x-c_*t,\zeta)\big)\md \zeta\right|
		\lesssim t^{-\gamma_*} e^{-t^{\varsigma}}.
	\end{aligned}
\end{equation}

It remains to discuss the situation of $x-c_*t\le -t^{\frac{1}{2}+\varsigma}$. The analysis in the previous case will work by a slight modification. We outline the details for the sake of completeness. Define now the vertical segments:
\begin{equation*}
	\Gamma'_\pm:=\left\{\zeta=\pm t^{\frac{1}{2}-\gamma_*}-\mathbf{i}\beta,~0\le\beta\le 1\right\},
\end{equation*}
and the horizontal segment:
\begin{equation*}
	\underline \Gamma:=\left\{\zeta=\eta-\mathbf{i},~|\eta|\le t^{\frac{1}{2}-\gamma_*}
	\right\}.
\end{equation*}
On $\Gamma_\pm'$, we have 
\begin{equation*}
	\Psi(t,x-c_*t,\zeta)=\pm\mathbf{i}\frac{x-c_*t}{t^{\gamma_*}}+\frac{\beta(x-c_*t)}{\sqrt{t}}-b(\pm t^{\frac{1}{2}-\gamma_*}-\mathbf{i}\beta)^2+O(t^{1-3\gamma_*}),
\end{equation*}
whence 
\begin{equation}\label{eqn-3}
	\Re\Psi(t,x-c_*t,\zeta)\le -b t^{1-2\gamma_*}\left(1+o(1)\right).	
\end{equation}
On $\underline \Gamma$, since
\begin{equation*}
	\Psi(t,x-c_*t,\zeta)=\frac{x-c_*t}{\sqrt{t}}+\mathbf{i}\frac{\eta(x-c_*t)}{\sqrt{t}}-b(\eta-\mathbf{i})^2+\mathbf{i}\frac{k(\eta-\mathbf{i})^3}{\sqrt{t}}+O\Big(\frac{(\eta-\mathbf{i})^4}{t}\Big),
\end{equation*}
it follows that
\begin{equation}\label{eqn-4}
	\Re	\Psi(t,x-c_*t,\zeta)\le \frac{x-c_*t}{\sqrt{t}}+O(1)\le -t^{\varsigma}+O(1).
\end{equation}
Also, there holds $
	|r(\zeta/\sqrt{t})|\le \upsilon_0+O(t^{-\gamma_*})$ for $\zeta\in\Gamma'_-\cup\underline\Gamma\cup \Gamma'_+$.
Eventually, \eqref{eqn-3} and \eqref{eqn-4} imply that
\begin{equation}
	\label{eqn-34}
	\begin{aligned}
		\left|\mathcal{G}(t,x)\right|&=  \frac{1}{2\pi\sqrt{t}}\left|\int_{\Gamma'_-\cup\underline\Gamma\cup\Gamma'_+}r(\zeta/\sqrt{t})\exp\big(\Psi(t,x-c_*t,\zeta)\big)\md \zeta\right|
		\lesssim t^{-\gamma_*} e^{-t^{\varsigma}}.
	\end{aligned}
\end{equation}
Therefore, \eqref{estimate_2} is achieved, thanks to \eqref{eqn-12} and \eqref{eqn-34}. Consequently, it follows from \eqref{I_12} and \eqref{G1} that, up to increasing $T$,
\begin{equation}\label{I_1}
	\big|\mathcal{I}_{1}(t,x)\!-\!\mathcal{G}**\mathcal{F}(t,x)\big|\!=\!\big|\mathcal{G}_1**\mathcal{F}(t,x)+\mathcal{I}_{1,2}(t,x)\big|\!\le\! C \Vert u_0\Vert_{L^{\infty}([0,i_\dagger)\times\R)}e^{-\frac{b}{2}t^{1-2\gamma_*}}
\end{equation}
for $t\ge T$, uniformly in $x\in\R$.
\vskip 2mm

\noindent
\textbf{Conclusion}. Combining  \eqref{F1}, \eqref{I_2-1} and \eqref{I_1},  we  arrive at \eqref{4.2-1}. 
 The proof of Theorem \ref{thm_3.1} is  complete. 

\section{Approximation within and beyond the diffusive scale} 
\label{sec4}

The main step in the proof of Theorem \ref{thm-1'} is the study of the solution $\rho(t,i,x)$ in the diffusive zone ahead of the moving boundary $x\approx c_*t$, which entails the study of  the function $v(t,i,x)$ defined by (\ref{def_v}). As $\rho(t,i,x)$ is a bounded function, the function $v(t,i,x)$ is very close to 0 as soon as $x-c_*t$ is a little below 0, that is, $x-c_*t\sim -t^{\delta}$ with $\delta>0$ small. This observation is the key to all the PDE proofs of the logarithmic delay in Fisher-KPP type equations, and the present model is no exception. What it means is that, in order to investigate what $v(t,i,x)$ looks like ahead of the moving boundary, it is a good idea to try to approximate it with  the solution, denoted by $\mathbf\Phi$, to the linear renewal  problem with an imposed Dirichlet  boundary condition set around $x=c_*t$. The unfortunate feature of this class of models is that asymptotics of the Dirichlet heat kernel is, in general, not known. This was already an important  issue that had to be overcome in \cite{BFRZ2023} or \cite{Roquejoffre2022}. 

Another important issue is to control $v(t,i,x)$ sufficiently far at infinity, as the heat kernel estimates are valid only up to a certain distance of the moving boundary, say, $t^{1/2+\varsigma}$ with $\varsigma>0$ small.  In the nonlocal Fisher-KPP equation -- see \cite{Roquejoffre2022} --  it was  easy, here, a new idea is required.  
\subsubsection*{Dirichlet type approximation in the diffusive zone}
The inspiration comes here from \cite{Roquejoffre2022}, where an approximation of the solution of the equation in the diffusion zone ahead of the moving boundary was approximated by extending the initial data in an odd fashion across the moving boundary, and applying the the key estimate on the heat kernel -- in the present work, provided by Theorem \ref{thm_3.1}. This idea will work here, but what one should do is not an odd extension of the initial datum.
\begin{thm}
	\label{thm_Dirichlet heat eqn}

	Assume that $u_0(i,x)$ is a nontrivial, absolutely continuous and compactly supported function  in $[0,i_\dagger)\times\R$, which is  nonnegative for $x\ge 0$ and nonpositive  for $x\le 0$, uniformly in $i\in[0,i_\dagger)$, such that
	\begin{equation}
	\label{extension}
	\int_0^{+\infty}\Big(  \K_* *  u_0(i,x)+ \K_* *  u_0(i,-x)\Big)\md x=0~~~~~ \hbox{for all $i\in[0,i_\dagger)$.}
\end{equation}
	 Let $\mathbf\Phi$ be the solution   to the linear renewal equation \eqref{renewal eqn=3} for $t>0$ and $x\in\R$ starting with initial condition $u_0$. Choose $T>i_\dagger$ sufficiently large so that Theorem \ref{thm_3.1} holds true for all $t\ge T$. Then, up to increasing $T$ if necessary, the function $\mathbf\Phi$ has the following asymptotics  
	\begin{equation*}
		\mathbf\Phi(t,x)\sim \frac{x-c_*t}{t^{\frac{3}{2}}}e^{-\frac{(x-c_*t)^2}{4bt}},~~~~~~~~t\ge T,~~|x-c_*t|\le O(t^{\frac{1}{2}+\varsigma}),
	\end{equation*}
	for $\varsigma\in(0,\frac{1}{2})$ small enough. 
\end{thm}
The assumption \eqref{extension} on $u_0$ in Theorem \ref{thm_Dirichlet heat eqn} can be fulfilled by requiring
		\begin{equation*}
	u_0(i,-x)=-\frac{	\int_0^\infty\big(\K_* (z-x)+  \K_* (-z-x)\big) \md z}{\int_0^\infty\big(\K_*(z+x) +\K_*(-z+x)\big)\md z} u_0(i,x),~~~i\in[0,i_\dagger),~~x\in\R_+.
		\end{equation*}
	This is in fact an analogue of assuming odd-type, compactly supported initial data in the classical KPP equation.
	The detailed computation leading to this condition is elementary, but given because it is important. 
By noticing that
\begin{align*}
			\K_* *  u_0(i,x)&=\int_\R\K_* (x-y)  u_0(i,y)\md y=\int_0^\infty\K_* (x-y)  u_0(i,y)+\K_*(x+y)u_0(i,-y) \md y,\\
		\K_* *  u_0(i,-x)&=\int_\R\K_* (-x-y)  u_0(i,y)\md y=\int_0^\infty \K_* (-x-y)  u_0(i,y)+\K_*(-x+y)u_0(i,-y)\md y,
	\end{align*}
it then follows from Fubini theorem that
\begin{align*}
	0&=\int_0^{\infty} \Big(  \K_* *  u_0(i,x)+ \K_* *  u_0(i,-x) \Big) \md x\\
	&=\int_0^{\infty}  \int_0^\infty\K_* (x-y)  u_0(i,y)+\K_*(x+y)u_0(i,-y) +  \K_* (-x-y)  u_0(i,y)+\K_*(-x+y)u_0(i,-y)\md y\md x\\
	&=\int_0^{\infty}  \int_0^\infty\Big(\K_* (x-y)+  \K_* (-x-y)\Big)  u_0(i,y)+\Big(\K_*(x+y) +\K_*(-x+y)\Big)u_0(i,-y)\md y\md x\\
	&=\int_0^{\infty} \bigg[ u_0(i,y) \int_0^\infty\Big(\K_* (x-y)+  \K_* (-x-y)\Big) \md x + u_0(i,-y)\int_0^\infty\Big(\K_*(x+y) +\K_*(-x+y)\Big)dx \bigg]\md y.
\end{align*}

\begin{proof}[Proof of Theorem \ref{thm_Dirichlet heat eqn}] 
	Let $u_0$ and $T$ be given as in the statement.
Let $\mathbf\Phi(t,x)$ be the solution to  equation \eqref{renewal eqn=3} for $t>0$ and $x\in\R$ associated with initial datum $u_0$.
Since   $\mathcal{F}(t,x)\equiv0$ for $t\ge i_\dagger$, uniformly in $x\in\R$,  a direct conclusion from Theorem \ref{thm_3.1} is that for  $t\ge T$ and $|x-c_*t|=\! O( t^{\frac{1}{2}+\varsigma})$ with $\varsigma>0$ small enough,
	\begin{align*}
	\mathbf	\Phi(t,x)&=\mathcal{G}**\mathcal{F}(t,x)+O(e^{-\frac{b}{2}t^{1-2\gamma}})
	= \int_{t-i_\dagger}^t \int_\R \mathcal{G}(s,x-y) \mathcal{F}(t-s,y)\md y\md s+O(e^{-\frac{b}{2}t^{1-2\gamma}})\\
	&=\big(1+o(1)\big)\!\underbrace{\int_{t-i_\dagger}^t \int_\R \frac{\upsilon_0}{\sqrt{2bs}}\mathscr{G}\Big(\frac{x-y-c_*s}{\sqrt{2bs}}\Big) \mathcal{F}(t-s,y)\md y\md s}_{=:P(t,x)}+O(e^{-\frac{b}{2}t^{1-2\gamma}}).
	\end{align*}
By Fubini theorem, the function $P(t,x)$  can be rewritten as
\begin{align*}
	P(t,x)&=\S_0\int_{t-i_\dagger}^t \frac{\upsilon_0}{\sqrt{2bs}}\int_\R \mathscr{G}\Big(\frac{x-y-c_*s}{\sqrt{2bs}}\Big) \int_0^{+\infty}\omega(i+t-s) e^{-\alpha_*c_* (i+t-s)}\K_* *  u_0(i,y)\md i\md y\md s\\
	&=\S_0\int_{t-i_\dagger}^t \frac{\upsilon_0}{\sqrt{2bs}}\int_0^{+\infty}\omega(i+t-s) e^{-\alpha_*c_* (i+t-s)}\!\underbrace{\int_\R \mathscr{G}\Big(\frac{x-y-c_*s}{\sqrt{2bs}}\Big) \K_* *  u_0(i,y)\md y}_{=:Q(s,i,x)}\md i\md s
\end{align*}
 for $t\ge T$ and $|x-c_*t|=\! O( t^{\frac{1}{2}+\varsigma})$.
Noticing that the integral variable $s$ in the above formula takes values in $(t-i_\dagger,t)$, this together with   $|x-c_*t|=O( t^{\frac{1}{2}+\varsigma})$ implies that  $|x-c_*s|= O(s^{\frac{1}{2}+\varsigma})$. Up to increasing $T$, it follows from series expansion and the assumption on $u_0$ that
\begin{align*}
	Q(s,i,x)&=\int_0^{+\infty} \mathscr{G}\Big(\frac{x-y-c_*s}{\sqrt{2bs}}\Big) \K_* *  u_0(i,y)+\mathscr{G}\Big(\frac{x+y-c_*s}{\sqrt{2bs}}\Big) \K_* *  u_0(i,-y)\md y\\
	&=\frac{1}{\sqrt{2\pi}}e^{-\frac{(x-c_*s)^2}{4bs}}\int_0^{+\infty} e^{\frac{(x-c_*s)y}{2bs}-\frac{y^2}{4bs}} \K_* *  u_0(i,y)+e^{-\frac{(x-c_*s)y}{2bs}-\frac{y^2}{4bs}} \K_* *  u_0(i,-y)\md y\\
	&=\frac{1}{\sqrt{2\pi}}e^{-\frac{(x-c_*s)^2}{4bs}}\left(\int_0^{+\infty} \frac{(x-c_*s)y}{2bs}\Big( \K_* *  u_0(i,y)- \K_* *  u_0(i,-y)\Big)\md y+O\Big(\frac{1}{s}\Big)\right).
\end{align*}
Substituting the above  into the expression of $P(t,x)$ yields that,  up to increasing $T$,
$$
	P(t,x)=C\int_{t-i_\dagger}^t \frac{1}{\sqrt{s}}\Big(\frac{x-c_*s}{s}+O\Big(\frac{1}{s}\Big)\Big)e^{-\frac{(x-c_*s)^2}{4bs}}\md s
	\approx C\frac{x-c_*t}{t^{\frac{3}{2}}}e^{-\frac{(x-c_*t)^2}{4bt}}
	$$
for $t\ge T$ and $|x-c_*t|=\! O( t^{\frac{1}{2}+\varsigma})$.
Consequently, $\displaystyle{\mathbf	{\Phi}}(t,x)\approx C\frac{x-c_*t}{t^{\frac{3}{2}}}e^{-\frac{(x-c_*t)^2}{4bt}}$  for $t\ge T$ and $|x-c_*t|=\! O( t^{\frac{1}{2}+\varsigma})$. 
\end{proof} 

\subsubsection*{Control beyond the diffusive regime} 

At this point, it is worth recalling from the dispersion relation \eqref{dr} that  the function $\alpha\in[0,+\infty)\mapsto D_c(\alpha)$ is convex and satisfies $D_{c_*}(\alpha_*)=0$, $\partial_\alpha D_{c_*}(\alpha_*)=0$ and $\partial^2_{\alpha}D_{c_*}(\alpha_*)>0$, thereby 
\begin{equation}\label{dr_1}
	\S_0 \widehat{\K_*}(0)\int_0^{i_\dagger} \omega(i) e ^{-\alpha_* c_*i}\md i=1,~~~~
	\S_0 \int_0^{i_\dagger} \omega(i) e ^{-\alpha_* c_*i}\int_\R \K_*(z) (z-c_*i)\md z\md i=0,
\end{equation}
with $\widehat{\K_*}(0)=\int_\R \widehat{\K_*}(x)\md x$, and
\begin{align}\label{dr_2}
	\delta_*:=\partial^2_{\alpha}D_{c_*}(\alpha_*)=\S_0 \int_0^{i_\dagger} \omega(i) e ^{-\alpha_* c_*i}\int_\R \K_*(z) (z-c_*i)^2\md z\md i>0.
\end{align}

\begin{lem} 
	\label{lem_initial esti} 
	Let $u$ be the bounded and continuous solution of the linear transport problem
	\eqref{moving frame-u} associated with an absolutely continuous, bounded and compactly supported  initial datum
	$u_0$ in $[0,i^\dagger)\times\mathbb R$. 
	Then, for every $T>0$ and $\boldsymbol{a}>0$,
	\begin{equation}
		\label{finite-time-exp-tail}
	|u(t,0,x)|\le C_{\boldsymbol{a},T}e^{-\boldsymbol{a} x}~~~~~\text{for all}~~(t,x)\in [0,T]\times\R,
	\end{equation}
	with $C_{\boldsymbol{a},T}:=
		\mathscr{R}_0\K_{\boldsymbol{a}} M_{u_0}
	e^{
		\K_{\boldsymbol{a}}
		\S_0\tau_\infty T}$, where $\K_{\boldsymbol{a}}:=
		\int_{\mathbb R}
		\K_*(x)e^{\boldsymbol{a} x}\md x$ and $M_{u_0}:=
		\max_{(i,x)\in[0,i_\dagger)\times\R}
		e^{\boldsymbol{a} x}|u_0(i,x)|.$
\end{lem}
 
\begin{proof}
	
Based on \eqref{sol_u}, we have $u(t,0,x)=\Phi(t,x)$ for $t>0$ and $x\in\R$, where we recall  that
\begin{equation}\label{renewal sol_linear}
	\begin{aligned}
	\Phi(t,x)=\S_0 \int_0^t \omega(i) e^{-\alpha_*c_* i}\K_* * \Phi(t-i,x)\md i+\underbrace{\S_0\int_0^{\infty} \omega(i+t) e^{-\alpha_*c_* (i+t)}\K_* *  u_0(i,x)\md i}_{=\mathcal{F}(t,x)}.
	\end{aligned}
\end{equation}

	We first  estimate $\mathcal{F}(t,x)$.  For all $(t,x)\in\R_+\times\R$ and $\boldsymbol{a}>0$, we have
	\begin{equation*}
		\begin{aligned}
			e^{\boldsymbol{a} x}|\mathcal{F}(t,x)|&\le \S_0\int_0^{\infty} \omega(i+t) e^{-\alpha_*c_* (i+t)}\int_\R\K_* (z) e^{\boldsymbol{a} z} e^{\boldsymbol{a}(x-z)} |u_0(i,x-z)|\md z\md i\\
			& \le \S_0\int_\R\K_* (z) e^{\boldsymbol{a} z}\md z\int_0^{\infty} \omega(i+t) e^{-\alpha_*c_* (i+t)} \max_{x\in\R}\big(e^{\boldsymbol{a} x} |u_0(i,x)|\big)\md i\le \mathscr{R}_0\K_{\boldsymbol{a}} M_{u_0}<+\infty,
		\end{aligned}
	\end{equation*}
	where $\K_{\boldsymbol{a}}$ and $M_{u_0}$ are given in the statement, and $\mathscr{R}_0=\S_0\int_0^{i_\dagger}\omega(i)\md i\in(1,+\infty)$.
Consequently, 
\begin{equation*}
	\label{esti_F_tx}
	e^{\boldsymbol{a} x}|\mathcal{F}(t,x)| \le \mathscr{R}_0\K_{\boldsymbol{a}} M_{u_0}=:\mathcal{F}_{\boldsymbol{a}}<+\infty~~~\text{for all}~(t,x)\in\R_+\times\R.
\end{equation*}

Define for $t>0$ and $x\in\R$,
\begin{align*}
	\Phi_0(t,x)&:=\mathcal{F}(t,x),\\
		\Phi_{n+1}(t,x)&:=\S_0 \int_0^t \omega(i) e^{-\alpha_*c_* i}\K_* * \Phi_n(t-i,x)\md i,~~~n\ge 0.
\end{align*}
	We claim that
	\begin{equation}
		\label{volterra-iterate-bound}
e^{\boldsymbol{a} x}|\Phi_n(t,x)|\le \mathcal{F}_{\boldsymbol{a}} \frac{\big(\K_{\boldsymbol{a}}\S_0\tau_\infty t\big)^n}{n!}~~~\text{for all}~~x\in\R.
	\end{equation}
	This is immediate for $n=0$. Suppose that it holds for some $n$.
	Then,
	\begin{equation*}
			\begin{aligned}
		e^{\boldsymbol{a} x}|\Phi_{n+1}(t,x)|&\le \S_0 \int_0^t \omega(i) e^{-\alpha_*c_* i}\int_\R\K_*(z)e^{\boldsymbol{a} z} e^{\boldsymbol{a}(x-z)}  \Phi_n(t-i,x-z)\md z\md i\\
		&\le \S_0\int_\R\K_*(z)e^{\boldsymbol{a} z}\md z \int_0^t \omega(i) e^{-\alpha_*c_* i} \sup_{x\in\R}\big(e^{\boldsymbol{a} x}  \Phi_n(t-i,x)\big)\md i\\
			&\le
			\S_0\K_{\boldsymbol{a}} \tau_\infty   
		 \mathcal{F}_{\boldsymbol{a}} \frac{\big(\K_{\boldsymbol{a}}\S_0\tau_\infty \big)^n}{n!} 	\int_0^t (t-i)^n \md i\\
			&=
			\mathcal{F}_{\boldsymbol{a}}
			\frac{
			\big(\K_{\boldsymbol{a}}\S_0\tau_\infty t \big)^{n+1}
			}{(n+1)!}.
		\end{aligned}
	\end{equation*}
 By taking the supremum over $x\in\R$, our claim \eqref{volterra-iterate-bound} follows. Therefore, for any fixed $T>0$, we have
 \begin{equation*}
 		e^{\boldsymbol{a} x}|\Phi_{n+1}(t,x)|\le 	\mathcal{F}_{\boldsymbol{a}}
 		\frac{
 			\big(\K_{\boldsymbol{a}}\S_0\tau_\infty T \big)^{n+1}
 		}{(n+1)!}~~~~\text{for all}~(t,x)\in[0,T]\times\R.
 \end{equation*}
	Furthermore, since 
	\begin{equation*}
		\sum_{n=0}^\infty \mathcal{F}_{\boldsymbol{a}}
		\frac{
			\big(\K_{\boldsymbol{a}}\S_0\tau_\infty T \big)^{n+1}
		}{(n+1)!} = \mathcal{F}_{\boldsymbol{a}} e^{\K_{\boldsymbol{a}}\S_0\tau_\infty T}<+\infty,
	\end{equation*}
    the Weierstrass M-test implies that the series
	$\sum_{n=0}^{\infty}\Phi_n(t,\cdot)$ converges absolutely in
	$L^\infty_{\boldsymbol{a}}(\mathbb R)$, uniformly in
	$t\in[0,T]$, to some function $\Phi(t,\cdot)$, where
	$$
	L^\infty_{\boldsymbol a}(\mathbb R)
	:=
	\Big\{
	\phi:\mathbb R\to\mathbb R~\big|~
	\|\phi\|_{L^\infty_{\boldsymbol a}(\mathbb R)}
	:=
	\sup_{x\in\mathbb R}
	e^{\boldsymbol a x}|\phi(x)|<+\infty
	\Big\}.
	$$
    To be more precise, we have 
	\begin{equation}\label{esti_Phi}
		e^{\boldsymbol{a} x}|\Phi(t,x)|\le \sum_{n=0}^\infty 	e^{\boldsymbol{a} x}|\Phi_{n}(t,x)| \le \mathcal{F}_{\boldsymbol{a}} e^{\K_{\boldsymbol{a}}\S_0\tau_\infty T}=:C_{\boldsymbol{a}, T}~~~\text{for all}~(t,x)\in [0,T]\times\R.
	\end{equation}
	
	We now show that  $\Phi$ satisfies the renewal equation \eqref{renewal sol_linear}. In fact,
	for any $T>0$ and $\boldsymbol{a}>0$, define the operator $\mathcal{T}$ on $X_{\boldsymbol{a},T}:= L^\infty((0,T);L^\infty_{\boldsymbol{a}}(\R))$ as follows:
\begin{equation*}
	\mathcal{T}\Phi(t,x):= \S_0 \int_0^t \omega(i) e^{-\alpha_*c_* i}\K_* * \Phi(t-i,x)\md i
\end{equation*}
Then it follows that 
\begin{equation*}
	\sup_{(t,x)\in [0,T]\times\R} e^{\boldsymbol{a} x} 	|\mathcal{T}\Phi(t,x)|\le \S_0\K_{\boldsymbol{a}} \tau_\infty T \sup_{(t,x)\in [0,T]\times\R} e^{\boldsymbol{a} x} 	|\Phi(t,x)|.
\end{equation*}
Thus, $\mathcal T: X_{\boldsymbol{a},T}\to X_{\boldsymbol{a},T}$ is a bounded linear operator. Therefore, from
\begin{equation*}
	S_N:=\sum_{n=0}^N\Phi_n=\mathcal{F}+\sum_{n=1}^N \Phi_n =\mathcal{F}+\sum_{n=1}^N \mathcal{T}\Phi_{n-1}= \mathcal{F}+ \mathcal{T}\sum_{n=1}^N\Phi_{n-1}=  \mathcal{F}+ \mathcal{T}	S_{N-1},
\end{equation*}
and from $S_N$ converges to $\Phi$ in $X_{\boldsymbol{a},T}$, 
passing to the limit and using the continuity of $\mathcal T$ gives $\Phi=\mathcal{F}+\mathcal{T}\Phi$. Therefore, $\Phi$ satisfies the renewal equation \eqref{renewal sol_linear}.
	
Suppose that 
	$\widetilde\Phi$ is another solution of \eqref{renewal sol_linear}. Define
	$D:=\Phi-\widetilde\Phi$ in $(0,+\infty)\times\R$, then
	\begin{equation*}
		D(t,x)= \S_0 \int_0^t \omega(i) e^{-\alpha_*c_* i}\K_* * D(t-i,x)\md i.
	\end{equation*}
 Thus, we have 
$
	\Vert D(t,\cdot)\Vert_{L^\infty(\R)}
	\le
		\Vert K_*	\Vert_{L^1(\R)}
\S_0\tau_\infty
	\int_0^t
		\Vert D(s,\cdot)	\Vert_{L^\infty(\R)}\md s$ for $t\in(0,T]$ with any fixed $T>0$.
It then follows from the Gronwall inequality that  $D\equiv0$ on $(0,T]\times\R$. Therefore, $\Phi$ is the unique bounded and continuous solution of \eqref{renewal sol_linear} on $(0,+\infty)\times\R$. Together with \eqref{esti_Phi}, we arrive at \eqref{finite-time-exp-tail}.
\end{proof}

	\begin{thm}
		\label{thm_exp decay beyond diffusive scale}
		Let $u$ be the solution of the linear transport problem \eqref{moving frame-u} in $\R_+\times[0,i_\dagger)\times\R$ with initial datum $u_0$ defined in $[0,i_\dagger)\times\R$ and satisfying the hypothesis of Theorem~\ref{thm_Dirichlet heat eqn}. 
		Then, for $\kappa_1,\kappa_2>0$ satisfying\footnote{This is possible by choosing  $0<\kappa_2<\frac{a_*}{8c_*^2\delta_* \widehat{\K_*}(0)}$.} $0<4c_*\kappa_2<\kappa_1<\sqrt{2\kappa_2 a_*/\big(\delta_*\widehat{\K_*}(0)\big)}$,  with $\delta_*$ given in \eqref{dr_2} and  $a_*=\int_\R  z\K_*(z)\md z$,  there exist $t^*>i_\dagger$ sufficiently large $($depending on $\kappa_1$$)$, and $\kappa_3\ge 1+\frac{4}{\kappa_1}$  such that the function $t^{-2}U(t,i,x)$ defined for $t\ge i+t^*$, $i\in[0,i_\dagger)$ and $x\ge X(t,i)-R:= c_*(t-i)-\frac{\kappa_2}{\kappa_1}\frac{c_*i}{\sqrt{t}}+
		\kappa_3\sqrt{t}-R$, with 
			\begin{equation}\label{def_expo}
			U(t,i,x):= e^{
				-\kappa_1\big(
				\frac{
					x-c_*(t-i)+
					\frac{\kappa_2}{\kappa_1} \frac{c_*i}{\sqrt{t}}
				}{\sqrt{t}}-\kappa_3\big)
			},
		\end{equation} 
		 is a supersolution of \eqref{moving frame-u} for $t\ge i+i_\dagger+t^*$, $i\in[0,i_\dagger)$ and $x\ge X(t,i)$, in the sense of Definition \ref{Def_super+sub}. Moreover, let $\varsigma\in(0,\frac{1}{2})$ be given in Theorem \ref{thm_Dirichlet heat eqn}, then for any fixed $0<\epsilon<\varsigma$, there exists $M>1$ such that
		\begin{equation*}
			|	u(t-t^*,i,x)|\le M t^{-2} U(t,i,x)
		\end{equation*}
		for $t\ge i+i_\dagger+ t^*$, $i\in[0,i_\dagger)$ and $x\ge X(t,i)+t^{\frac{1}{2}+\epsilon}$. 
	\end{thm} 
	\begin{proof}[Proof of Theorem \ref{thm_exp decay beyond diffusive scale}]
		Fix $t^*>i_\dagger$ sufficiently large, and define $\overline u(t,i,x):=t^{-2}U(t,i,x)$ with $U$ given in \eqref{def_expo} for $t\ge i+t^*$, $i\in[0,i_\dagger)$ and $x\ge X(t,i)-R:= c_*(t-i)-\frac{\kappa_2}{\kappa_1}\frac{c_*i}{\sqrt{t}}+
		\kappa_3\sqrt{t}-R$.
	
		We first prove that $\overline u(t,i,x)$ is a supersolution to \eqref{moving frame-u}  for $t\ge i+i_\dagger+t^*$, $i\in[0,i_\dagger)$ and  $x\ge X(t,i):= c_*(t-i)-\frac{\kappa_2}{\kappa_1}\frac{c_*i}{\sqrt{t}}+
		\kappa_3\sqrt{t}$, in the sense of Definition \ref{Def_super+sub}. In fact, we have
		\begin{equation}
			\label{eqn_thm4.3_1}
			\begin{aligned}
				(\partial_t+\partial_i)\overline u(t,i,x)
				&=\bigg(\frac{\kappa_2c_*i}{2t^{2}}+\frac{\kappa_1\big(	x-c_*(t-i)+
					\frac{\kappa_2}{\kappa_1} \frac{c_*i}{\sqrt{t}}\big)}{2t^{\frac{3}{2}}}-\frac{\kappa_2c_*}{t}-\frac{2}{t}\bigg) \overline u(t,i,x)\\
				&\ge\bigg(\frac{\kappa_2c_*i}{2t^{2}}+\frac{\kappa_1	\kappa_3\sqrt{t}}{2t^{\frac{3}{2}}}-\frac{\kappa_2c_*}{t}- \frac{2}{t}\bigg)  \overline u(t,i,x)\\
				&\ge\frac{\kappa_1\kappa_3-2\kappa_2c_*- 4}{2t} \overline u(t,i,x)\ge 0   
			\end{aligned}
		\end{equation}
		for $t\ge i+t^*$, $i\in(0,i_\dagger)$ and $x\ge X(t,i)-R$, by noticing that $\kappa_3\ge 1+\frac{4}{\kappa_1}$ and $\kappa_1>4\kappa_2c_*>0$. 
		
		Next, we shall check 
		\begin{equation*}
			\S_0 \int_0^{i_\dagger} \omega(i) e^{-\alpha_*c_* i} \K_* * \overline u(t,i,x)\md i\le \overline u(t,0,x)= t^{-2} U(t,0,x), ~~~~~~~~~ t\ge i_\dagger+t^*, ~ ~ x\ge X(t,0).
		\end{equation*}
		For $ t\ge i_\dagger+t^*$ and $x\ge X(t,0)$, we actually have
		\begin{align}\label{4.44}
			\S_0 \int_0^{i_\dagger} \omega(i) &e^{-\alpha_*c_* i}\K_* * \overline u(t,i,x)\md i\nonumber\\
			&=\S_0 \int_0^{i_\dagger} \omega(i) e^{-\alpha_*c_* i}\int_{X(t,0)-R}^{+\infty}\K_*(x-y) \overline u(t,i,y)\md y\md i\nonumber\\
			& \le \S_0 \int_0^{i_\dagger} \omega(i) e^{-\alpha_*c_* i} \int_\R\K_*(x-y) e^{\kappa_1\frac{x-y}{\sqrt{t}}}\md y ~t^{-2}e^{
				-\kappa_1\big(
				\frac{
					x-c_*(t-i)+
					\frac{\kappa_2}{\kappa_1} \frac{c_*i}{\sqrt{t}}
				}{\sqrt{t}}-\kappa_3\big)
			}
			\md i~~~~~~~\nonumber\\
			&= \S_0 \int_0^{i_\dagger} \omega(i) e^{-\alpha_*c_* i} \int_\R\K_*(z) e^{\frac{\kappa_1(z-c_*i)}{\sqrt{t}}- \frac{\kappa_2 c_*i}{t}}\md z
			\md i~t^{-2} e^{
				-\kappa_1\big(
				\frac{
					x-c_*t
				}{\sqrt{t}}-\kappa_3\big)
			},
		\end{align}
		in which we recall that both $i\in[0,i_\dagger)$ and $z\in[-R,R]$ are bounded, thus  by means of a Taylor expansion, we get  that up to increasing $t^*$,    
		\begin{equation}
			\label{exp}
			\exp\Big({\frac{\kappa_1(z-c_*i)}{\sqrt{t}}- \frac{\kappa_2 c_*i}{t}}\Big)=1+ \frac{\kappa_1(z-c_*i)}{\sqrt{t}}- \frac{\kappa_2 c_*i}{t} +\frac{\kappa_1^2}{2}\frac{(z-c_*i)^2}{t}+O\Big(t^{-\frac{3}{2}}\Big).
		\end{equation}
		Substituting \eqref{exp} into \eqref{4.44},  with  \eqref{dr_1} and \eqref{dr_2} altogether, we deduce  that, up to increasing $t^*$,
		\begin{equation}\label{eqn_thm4.3_2}
			\begin{aligned}
				\S_0 \int_0^{i_\dagger} \omega(i) e^{-\alpha_*c_* i}\K_* * \overline u(t,i,x)\md i
				\le \bigg(1-\frac{1}{t}\Big(\frac{\kappa_2a_*}{\widehat{\K_*}(0)}-\frac{\kappa_1^2\delta_*}{2}\Big)+O\Big(t^{-\frac{3}{2}}\Big) \bigg)  \overline u(t,0,x)
				\le \Big(1-\frac{C}{t}\Big)  \overline u(t,0,x)
			\end{aligned}
		\end{equation}
		for $t\ge i_\dagger+ t^*$ and $x\ge X(t,0)$,  since the coefficient $\frac{\kappa_2a_*}{\widehat{\K_*}(0)}-\frac{\kappa_1^2\delta_*}{2}$ is positive by virtue of the choice of $\kappa_1$ and $\kappa_2$. Consequently, $\overline u(t,i,x)$ is a  supersolution  for problem \eqref{moving frame-u} for $t\ge i+i_\dagger+t^*$, $i\in[0,i_\dagger)$ and $x\ge X(t,i)$, in the sense of Definition \ref{Def_super+sub}. 
		
		Let $\varsigma\in(0,\frac{1}{2})$ be given in Theorem \ref{thm_Dirichlet heat eqn}. Now, we are going to show that  for any fixed $0<\epsilon<\varsigma$, there exists $M>0$ large enough such that  $u(t,i,x)\le M\overline u(t,i,x)$ for 
		$t\ge i+i_\dagger+ t^*$, $i\in[0,i_\dagger)$ and $x\ge X_\epsilon(t,i):=X(t,i)+t^{\frac{1}{2}+\epsilon}$. To this end, we first  notice that $X_\epsilon(t,i)$ for each $i\in[0,i_\dagger)$ is increasing in $t\in[ i+t^*,+\infty)$, and 
		$X_\epsilon(t-i+s,s)=c_*(t-i)-\frac{\kappa_2}{\kappa_1}\frac{c_*s}{\sqrt{t-i+s}}+\kappa_3\sqrt{t-i+s}+(t-i+s)^{\frac{1}{2}+\epsilon}\le c_*t+\kappa_3\sqrt{t}+t^{\frac{1}{2}+\epsilon}=X_\epsilon(t,0)$ for all $s\in[0,i]$. Moreover, for any fixed $\boldsymbol{a}>\frac{2\kappa_1}{\sqrt{t^*}}$, it follows from Lemma~\ref{lem_initial esti} that  $u(t-t^*,0,x)\le C_{\boldsymbol{a},i_\dagger}e^{-\boldsymbol{a} x}$ for all $(t,x)\in[t^*,t^*+i_\dagger]\times\R$ and for some constant $C_{\boldsymbol{a},i_\dagger}>0$. Therefore, we obtain 
	    $u(t-t^*,0,x)\le C_{\boldsymbol{a},i_\dagger}e^{- \frac{2\kappa_1}{\sqrt{t^*}}x}\le t^{-2}e^{-\kappa_1(\frac{x-c_*t}{\sqrt{t}}-\kappa_3)}= \overline u(t,0,x)$  for all $t\in[t^*,t^*+i_\dagger]$ and $x\ge X_\epsilon(t,0)=c_*t+\kappa_3\sqrt{t}+t^{\frac{1}{2}+\epsilon}$. In addition, Theorem~\ref{thm_Dirichlet heat eqn} along with \eqref{sol_u} implies that $u(t-t^*,0,x)\le C (t-t^*)^{-1+\epsilon}e^{-\frac{(t-t^*)^{2\epsilon}}{4b}}< t^{-2}e^{-\kappa_1 t^\epsilon}\le \overline u(t,0,x)$ for $t\ge 2t^*$ and $x\in[X_\epsilon(t,0)-R, X_\epsilon(t,0)]$, and $u(t-t^*,i,X_\epsilon(t,i))\le C(t-t^*)^{-1+\epsilon}e^{-\frac{(t-t^*)^{2\epsilon}}{4b}}< t^{-2}e^{-\kappa_1 t^\epsilon}\le \overline u(t,i,X_\epsilon(t,i))$  for all $t\ge 2t^*$ and $i\in[0,i_\dagger)$. Moreover, one can choose $M>1$ large such that 
	    $u(t-t^*,0,x)\le M \overline u(t,0,x)$ for $t\in[t^*,2t^*]$ and $x\in[X_\epsilon(t,0)-R, X_\epsilon(t,0)]$, and such that  $u(t-t^*,i,X_\epsilon(t,i))\le M \overline u(t,i,X_\epsilon(t,i))$
	     for all $t\in[ i+i_\dagger+t^*, 2t^*]$ and $i\in[0,i_\dagger)$. We therefore summarize and conclude from the maximum principle Proposition~\ref{prop_mp} that $u(t,i,x)\le M\overline u(t,i,x)$ for 
		$t\ge i+i_\dagger+ t^*$, $i\in[0,i_\dagger)$ and $x\ge X_\epsilon(t,i)$.

		On the other hand, due to the linearity of problem \eqref{moving frame-u}, it can also be easily checked that $-\overline u$ is a  subsolution to \eqref{moving frame-u}  for $t\ge i+i_\dagger+t^*$, $i\in[0,i_\dagger)$ and $x\ge X(t,i)$, in the sense of Definition \ref{Def_super+sub}, such that $-M\overline u(t,i,x)\le u(t,i,x)$ for	$t\ge i+i_\dagger+ t^*$, $i\in[0,i_\dagger)$ and $x\ge X_\epsilon(t,i)$.  The conclusion then follows. 
		\end{proof}

\section{Construction of the upper and lower barriers}
\label{sec5}

  This section is devoted to  the construction of a pair of super- and subsolutions for the nonlinear transport problem \eqref{moving frame-nonlinear}, roughly speaking, ahead of $x-c_*t\approx 0$ for $t$ sufficiently large. 
While the nature of the problem under study differs much from bistable type models, these barrier functions are deeply related to those devised by Fife-McLeod in \cite{FMcL} .

   Hereafter, 
   we shall denote by $u(t,i,x)$ the solution to the linear transport problem \eqref{moving frame-u} for $t>0$, $i\in[0,i_\dagger)$ and $x\in\R$ starting from  
   a nontrivial, absolutely continuous and compactly supported initial function $u_0$  in $[0,i_\dagger)\times\R$, which is  nonnegative for $x\ge 0$ and nonpositive  for $x\le 0$, uniformly in $i\in[0,i_\dagger)$, such that
    	$\int_0^{+\infty}\big( \K_* *  u_0(i,x)+ \K_* *  u_0(i,-x)\big)\md x=0$ for all $i\in[0,i_\dagger)$.

 \vskip 3mm
 
 We begin with introducing some parameters.  Let $\varsigma>0$ be as given in Theorem \ref{thm_Dirichlet heat eqn},  and $t^*$, $\kappa_1$, $\kappa_2$ and $U(t,i,x)$ be as stated in Theorem \ref{thm_exp decay beyond diffusive scale}. Recall that $a_*=\int_R\K_*(z)z\md z>0$ and $\delta_*$ is given by \eqref{dr_2}.  Fix any $\epsilon\in(0,\varsigma)$, then Theorem \ref{thm_exp decay beyond diffusive scale} gives the existence of $M>1$ such that 	
 $|u(t-t^*,i,x)|\le M t^{-2} U(t,i,x)$
 for $t\ge i+i_\dagger+ t^*$, $i\in[0,i_\dagger)$ and $x\ge X(t,i)+t^{\frac{1}{2}+\epsilon}$.
 Moreover, we fix  parameters  $\delta$, $\beta$, $\alpha$, $\gamma$, $\mu$ and $k$   such that
\begin{equation}
	\label{parameters}
	0<\delta<\gamma<\beta<\frac{4}{25}<\frac{7}{15}<\alpha<\frac{1}{2}<2\alpha<k<1<\mu<\frac{26}{25}.
\end{equation}
Choose $1 +\frac{4}{\kappa_1}<\ell_1<\ell_2<+\infty$, and let $\chi:\R_+\to [0,1]$ be a smooth cut-off function such that  $\chi\equiv0$ in $[0,\ell_1]$, $0<\chi<1$ in $(\ell_1,\ell_2)$ and $\chi\equiv1$ in $[\ell_2,+\infty)$, with bounded derivatives of all orders in $[\ell_1,\ell_2]$.
Finally, take some $T>i_\dagger$ sufficiently large so that the conclusions of Theorems~\ref{thm_Dirichlet heat eqn} and \ref{thm_exp decay beyond diffusive scale} are valid for all $t\ge T$. It is then worth observing,  together with \eqref{sol_u}, that  the function $u$ has the following asymptotics:
\begin{equation}
	\label{u-asympt}
	u(t,i,x)\sim \frac{x-c_*t}{t^{\frac{3}{2}}},~~~~~~~~t\ge T,~|x-c_*t|= O(t^{\frac{1}{2}+\varsigma}),
\end{equation}
 uniformly in $i\in[0,i_\dagger)$.

\subsection{The upper barrier} 
For $t\ge i+T$, $i\in[0,i_\dagger)$ and $x\ge X(t,i)-R:=c_*(t-i) -t^\delta-R$, we define
\begin{equation}
	\label{upper}
		\overline v(t,i,x)=
			\overline\xi(t,i)u(t,i,x)+\mathcal{V}_1(t,i,x)+\mathcal{V}_2(t+t^*,i,x),
\end{equation}
in which 
\begin{equation*}
	\label{xi(t,i)}
	\overline \xi(t,i)=1-\frac{1+\frac{c_*i}{t^\mu}}{t^{\gamma}},
\end{equation*} 
and
\begin{equation*} 
	\mathcal{V}_1(t,i,x)=\frac{1+\frac{c_*i}{t^k}}{t^{\frac{3}{2}-\beta}}\cos\left(\frac{x-c_*(t-i)}{t^\alpha}\right)\mathds{1}_{\big\{x\in\R\big|-t^{\delta}-R\le x-c_*(t-i)\le \frac{3\pi}{2}t^{\alpha}\big\}},
\end{equation*}
and 
\begin{equation}\label{V_2}
	\mathcal{V}_2(t,i,x)=M t^{-2}\mathcal{U}\bigg(	\frac{
		x-c_*(t-i)+
		\frac{\kappa_2}{\kappa_1} \frac{c_*i}{\sqrt{t}}
	}{\sqrt{t}}\bigg), 
\end{equation}
with
\begin{equation*}\label{U}
	\mathcal{U}(s)=\chi(s)	 e^{
		-\kappa_1(s-\ell_2)
	},~~~~~s\ge 0.
\end{equation*}
We first remark that, up to increasing $T$ if necessary, $\overline v$  defined in \eqref{upper} is a positive function  $t\ge i+T$, $i\in[0,i_\dagger)$ and $x\ge X(t,i)-R$, by Theorems~\ref{thm_Dirichlet heat eqn} and \ref{thm_exp decay beyond diffusive scale}.
Since  
$X(t,i)$ given above is increasing in $t$ but decreasing in $i$, we have $X(t,0)\ge X(t-i+s,0)\ge X(t-i+s,s)$ for all $s\in[0,i]$, which implies $X(t,0)\ge\max_{s\in[0,i]}X(t-i+s,s)$ for $t\ge i+T$ and $i\in[0,i_\dagger)$.

In what follows, we shall check that $\overline v$ is a  supersolution to the nonlinear transport problem \eqref{moving frame-nonlinear} for $t> i+i_\dagger+T$, $i\in[0,i_\dagger)$ and $x\ge X(t,i)$ in the sense of Definition \ref{Def_super+sub}.  To this end, one easily observes that it will be sufficient to check that $\overline v$ is a supersolution to the linear problem \eqref{moving frame-u} for   $t> i+i_\dagger+T$, $i\in[0,i_\dagger)$ and $x\ge X(t,i)$, thanks to the comparison principle Proposition \ref{prop_cp_nonlinear transport}.

Before proceeding, let us point out  that the term $\overline\xi(t,i) u(t,i,x)$ itself  is a good candidate for the supersolution within the diffusive regime. As a matter of fact, one can compute that for $t\ge i+T$, $i\in(0,i_\dagger)$ and  $x\ge X(t,i)-R$, 
\begin{equation}
	\label{eqn_xi u}
	(\partial_t+\partial_i)\big(\overline \xi(t,i)u(t,i,x)\big)=u(t,i,x)	(\partial_t+\partial_i)\overline\xi(t,i) 
	= \bigg(\!\frac{\gamma\big(1+\frac{c_*i}{t^\mu}\big)}{t^{1+\gamma}}\!+\!\frac{\mu c_* i}{t^{1+\gamma+\mu}}\!-\!\frac{c_*}{t^{\mu+\gamma}}\!\bigg)u(t,i,x)
	\approx \frac{C}{t^{1+\gamma}}u(t,i,x),
\end{equation}
thanks to $\gamma>0$ and $\mu>1$. In particular, 
we  conclude  that
$(\partial_t+\partial_i)\big(\overline \xi(t,i)u(t,i,x)\big)\ge 0$ for $t\ge i+T$, $i\in (0,i_\dagger)$ and $R\le x-c_*(t-i)\le Ct^{\frac12+\varsigma}$.  Moreover, it follows from  \eqref{dr_1} as well as the asymptotics \eqref{u-asympt} of $u$ in the diffusive regime that,  up to increasing $T$,  
\begin{equation}
	\label{eqn_xi u_diffusive}
\begin{aligned}
		 \S_0\int_0^{i_\dagger}\omega(i)e^{-\alpha_*c_*i}&\overline\xi(t,i)\K_* * u(t,i,x) \md i\\
		 &= \S_0\int_0^{i_\dagger}\omega(i)e^{-\alpha_*c_*i}\Big(\overline\xi(t,0)-\frac{c_*i}{t^{\gamma+\mu}}\Big)\K_* * u(t,i,x) \md i\\
		&\le  \overline\xi(t,0) u(t,0,x)-\S_0\int_0^{i_\dagger}\omega(i)e^{-\alpha_*c_*i}\frac{c_*i}{t^{\gamma+\mu}}\int_\R\K_*(y) \frac{C\big(x-c_*t-y\big)}{t^{\frac{3}{2}}} \md y\md i\\
		&=  \overline\xi(t,0) u(t,0,x)-\frac{C\S_0}{t^{\gamma+\mu+\frac{3}{2}}}\int_0^{i_\dagger}\omega(i)e^{-\alpha_*c_*i}c_*i\md i\Big((x-c_*t)\widehat{\K_*}(0)-a_*\Big) \\
	& \le  \overline\xi(t,0) u(t,0,x)-\frac{C u(t,0,x)}{t^{\gamma+\mu}}+\frac{C}{t^{\gamma+\mu+\frac{3}{2}}}
\end{aligned}
\end{equation}
for $t\ge i_\dagger+T$ and $-t^\delta\le x-c_*t\le Ct^{\frac12+\varsigma}$.

 Let us now turn back to our target $\overline v$. 
According to the definition \eqref{upper} of $\overline v$, we  divide our analysis into several steps.
\vskip 2mm

\noindent
\textbf{Step 1}. Let us first deal with  $t\ge i+T$, $i\in[0,i_\dagger)$ and $-t^\delta-R\le x-c_*(t-i)\le \frac{3\pi}{2}t^\alpha$.  We have
\begin{equation}\label{upper v_cos}
	\overline v(t,i,x)=\overline\xi(t,i) u(t,i,x)+\mathcal{V}_1(t,i,x).
\end{equation}
To simplify the notation, let us denote
$\displaystyle
	\zeta=\zeta(t,i,x):=\frac{x-c_*(t-i)}{t^\alpha},
$
belonging to the interval $\displaystyle[\frac{-t^\delta-R}{t^{\alpha}}, \frac{3\pi}{2}]$. We get that
$$
	\big(\partial_t+\partial_i\big)\overline v=\bigg(\!\frac{\gamma\big(1+\frac{c_*i}{t^\mu}\big)}{t^{1+\gamma}}\!+\!\frac{\mu c_* i}{t^{1+\gamma+\mu}}\!-\!\frac{c_*}{t^{\mu+\gamma}}\!\bigg)u\!+\!\frac{\alpha\big(1+\frac{c_*i}{t^k}\big)}{t^{\frac{5}{2}-\beta}}\zeta\sin \zeta+\bigg(\frac{c_*}{t^{\frac{3}{2}-\beta+k}}-\frac{\big(\frac{3}{2}-\beta\big)\big(1+\frac{c_*i}{t^k}\big)}{t^{\frac{5}{2}-\beta}} -\frac{kc_*i}{t^{\frac{5}{2}-\beta+k}}\bigg) \cos\zeta.
$$

\noindent
\textbf{Step 1.1}. Claim: $\big(\partial_t+\partial_i\big)\overline v(t,i,x)\ge 0$ for  $t\ge i+T$,  $i\in (0,i_\dagger)$ and $-t^\delta-R\le x-c_*(t-i)\le \frac{3\pi}{2}t^\alpha$. Indeed,
\begin{itemize}
	\item   $\frac{-t^\delta-R}{t^{\alpha}}\le \zeta\le \frac{R}{t^\alpha}$. Here, we find that $u(t,i,x)\ge\frac{-C}{(t-i)^{\frac{3}{2}}t^{-\delta}}\ge \frac{-C}{t^{\frac{3}{2}-\delta}}$, $\cos\zeta$ is almost 1, whereas $\sin\zeta$ is close to $\zeta$. Combining with the fact that $1+\gamma+\beta>k+\delta$, we see that, up to increasing $T$,
	\begin{equation*}
		\big(\partial_t+\partial_i\big)\overline v(t,i,x)\ge  \frac{c_*}{2t^{\frac{3}{2}-\beta+k}}-\frac{C}{t^{1+\gamma+\frac{3}{2}-\delta}}>0.
	\end{equation*} 
	
\item $ \frac{R}{t^\alpha} \le \zeta\le \frac{\pi}{4}$. In this situation, the conclusion is obvious, based upon the observation that both $u(t,i,x)$ and $\cos\zeta$ are positive, together with our assumption \eqref{parameters} of the parameters.
	
\item $\frac{\pi}{4} \le \zeta\le \frac{3\pi}{2}$. In this case, the cosine term and the sine term do not have a definite sign. Nevertheless, since $1+\gamma+\beta<1+\frac{8}{25}<\frac{7}{5}<3\alpha<k+\alpha$, it turns out that, up to increasing $T$,  $\big(\partial_t+\partial_i\big)\overline v(t,i,x)$ is dominated by  $\displaystyle \frac{\gamma}{t^{1+\gamma}}u(t,i,x)\ge  \frac{C}{t^{1+\gamma+\frac{3}{2}-\alpha}}>0.$
\end{itemize}

\noindent
\textbf{Step 1.2}. Let us prove that 
\begin{equation}\label{step1_i=0}
	\S_0\int_0^{i_\dagger}\omega(i)e^{-\alpha_*c_*i}\K_* * \overline v(t,i,x) \md i\le \overline v(t,0,x)
\end{equation}
for  $t\ge i_\dagger+T$ and $-t^\delta\le x-c_*t\le \frac{3\pi}{2}t^\alpha$. Based on  \eqref{eqn_xi u_diffusive}, the left-hand side in \eqref{step1_i=0} can be written as
\begin{align*}
		&\S_0\int_0^{i_\dagger}\omega(i)e^{-\alpha_*c_*i}\K_* * \overline v(t,i,x) \md i\\
		&=\S_0\int_0^{i_\dagger}\omega(i)e^{-\alpha_*c_*i}\overline\xi(t,i)\K_* * u(t,i,x) \md i+\S_0\int_0^{i_\dagger}\omega(i)e^{-\alpha_*c_*i}\int_{c_*t-t^\delta-R}^{c_*t+\frac{3\pi}{2}t^\alpha}\K_*(x-y)  \mathcal{V}_1(t,i,y) \md y\md i\\
		&\le \overline\xi(t,0) u(t,0,x)-\frac{C u(t,0,x)}{t^{\gamma+\mu}}+\frac{C}{t^{\gamma+\mu+\frac{3}{2}}}+\underbrace{S_0\int_0^{i_\dagger}\omega(i)e^{-\alpha_*c_*i}\K_* * \left(\frac{1+\frac{c_*i}{t^k}}{t^{\frac{3}{2}-\beta}}\cos(\zeta(t,i,x))\right) \md i}_{=:\mathcal{D}(t,x)}
\end{align*}
for  $t\ge i_\dagger+T$ and $\displaystyle -t^\delta\le x-c_*t\le \frac{3\pi}{2}t^\alpha$.
To establish \eqref{step1_i=0}, it then will be sufficient to show that for  $t\ge i_\dagger+T$ and $\displaystyle-t^\delta\le x-c_*t\le \frac{3\pi}{2}t^\alpha$,
\begin{equation*}
		\overline\xi(t,0) u(t,0,x)-\frac{C\zeta(t,0,x)}{t^{\gamma+\mu+\frac{3}{2}-\alpha}}+\frac{C}{t^{\gamma+\mu+\frac{3}{2}}}+\mathcal{D}(t,x)\le \overline v(t,0,x)=	\overline\xi(t,0) u(t,0,x)+\frac{\cos\big(\zeta(t,0,x)\big)}{t^{\frac{3}{2}-\beta}},
\end{equation*}
that is, 
\begin{equation}
	\label{aim}
\frac{C}{t^{\gamma+\mu+\frac{3}{2}}}+	\mathcal{D}(t,x)\le \frac{\cos\big(\zeta(t,0,x)\big)}{t^{\frac{3}{2}-\beta}}+\frac{C\zeta(t,0,x)}{t^{\gamma+\mu+\frac{3}{2}-\alpha}}.
\end{equation}
To this end, we first note that
\begin{equation}
	\label{eqn-D}
	\begin{aligned}
		\mathcal{D}(t,x):=&\S_0\int_0^{i_\dagger}\omega(i)e^{-\alpha_*c_*i}\K_* * \left(\frac{1+\frac{c_*i}{t^k}}{t^{\frac{3}{2}-\beta}}\cos\left(\frac{x-c_*(t-i)}{t^\alpha}\right)\right)\md i\\
		=&\frac{1}{t^{\frac{3}{2}-\beta}}\S_0\int_0^{i_\dagger}\omega(i)e^{-\alpha_*c_*i}\bigg(1+\frac{c_*i}{t^k}\bigg)\int_\R \K_*(y) \cos\left(\frac{x-c_*t-(y-c_*i)}{t^\alpha}\right)\md y\md i,
	\end{aligned}
\end{equation}
where we see that, up to increasing $T$ if necessary,
\begin{align*}
	\int_\R \K_*(y) &\cos\left(\frac{x-c_*t-(y-c_*i)}{t^\alpha}\right)\!\md y\!\\
	&~~~~=\int_\R \K_*(y)\! \Bigg(\cos\left(\zeta(t,0,x)\right)\cos\left(\frac{y-c_*i}{t^\alpha}\right)\!-\!\sin\left(\zeta(t,0,x)\right)\sin\left(\frac{y-c_*i}{t^\alpha}\right)\Bigg)\!\md y\\
		&~~~~=\cos\left(\zeta(t,0,x)\right)\int_\R \K_*(y)\bigg(\underbrace{ 1-\frac{(y-c_*i)^2}{2t^{2\alpha}}+O(t^{-4\alpha})}_{=:J_1(t,i,y)}\bigg)\md y\\
		&~~~~~~~~~~~~~~~~~~~~~~~~~~~~-\sin\left(\zeta(t,0,x)\right)\int_\R \K_*(y) \bigg(\underbrace{
		\frac{y-c_*i}{t^\alpha}-\frac{(y-c_*i)^3}{6t^{3\alpha}}+O(t^{-5\alpha})}_{=:J_2(t,i,y)}
		\bigg)\md y.
\end{align*}
In addition, since  $m>2\alpha$, it follows that
\begin{align*}
	\left(1+\frac{c_*i}{t^k}\right) J_1(t,i,y)&=1-\frac{(y-c_*i)^2}{2 t^{2\alpha}}\!+\!\frac{c_*i}{t^k} \!-\!\frac{c_*i(y-c_*i)^2}{2t^{2\alpha+k}}\!+\!o(t^{-2\alpha-k})
	\le 1-\frac{(y-c_*i)^2}{4 t^{2\alpha}}-O(t^{-2\alpha-k})\\
	\left(1+\frac{c_*i}{t^k}\right) J_2(t,i,y)&=\frac{y-c_*i}{ t^{\alpha}}+\frac{c_*i(y-c_*i)}{t^{k+\alpha}} -\frac{(y-c_*i)^3}{6t^{3\alpha}}+o(t^{-3\alpha})
	\le \frac{y-c_*i}{ t^{\alpha}} -\frac{(y-c_*i)^3}{12t^{3\alpha}}+o(t^{-3\alpha}).
\end{align*}
As a consequence,
\begin{align*}
		\mathcal{D}(t,x)	 =& 	  \frac{\cos\left(\zeta(t,0,x)\right)}{t^{\frac{3}{2}-\beta}}\S_0\!\int_0^{i_\dagger}\!\omega(i)e^{-\alpha_*c_*i}\! \int_\R\! \K_*(y)\! \left(1+\frac{c_*i}{t^k}\!\right)\! J_1(t,i,y)\md y\md i\\
	&~~~~~~~~+\frac{\sin\left(\zeta(t,0,x)\right)}{t^{\frac{3}{2}-\beta}}\!\S_0\!\int_0^{i_\dagger}\!\omega(i)e^{-\alpha_*c_*i}\! \int_\R\! \K_*(y) \left(1+\frac{c_*i}{t^k}\right)\! J_2(t,i,y)\md y\md i\\
	=&: \mathcal{W}_1(t,x)+\mathcal{W}_2(t,x),
\end{align*}
and, along with \eqref{dr_1} and \eqref{dr_2}, we have that
\begin{align*}
	\mathcal{W}_1(t,x)&\le \frac{\cos\left(\zeta(t,0,x)\right)}{t^{\frac{3}{2}-\beta}}\S_0\int_0^{i_\dagger}\omega(i)e^{-\alpha_*c_*i} \int_\R \K_*(y) \left( 1-\frac{(y-c_*i)^2}{4 t^{2\alpha}}-O(t^{-2\alpha-k})\right)\md y\md i\\ 
	&\le \frac{\cos\left(\zeta(t,0,x)\right)}{t^{\frac{3}{2}-\beta}} \left( 1-\frac{C\delta_*}{4 t^{2\alpha}}\right)		
	\end{align*}
\begin{align*}
	\mathcal{W}_2(t,x)&\le \frac{\sin\left(\zeta(t,0,x)\right)}{t^{\frac{3}{2}-\beta}}\S_0\int_0^{i_\dagger}\omega(i)e^{-\alpha_*c_*i} \int_\R \K_*(y) \left(\frac{y-c_*i}{ t^{\alpha}} -\frac{(y-c_*i)^3}{12t^{3\alpha}}+o(t^{-3\alpha})\right) \md y\md i\\
	&\le \frac{C\sin\left(\zeta(t,0,x)\right)}{t^{\frac{3}{2}-\beta+3\alpha}}.
\end{align*}
Therefore, we have the following estimate for $\mathcal{D}$:
\begin{align}\label{estimate_D}
	\mathcal{D}(t,x)
		=\mathcal{W}_1(t,x)+\mathcal{W}_2(t,x)
		\le  \frac{\cos\left(\zeta(t,0,x)\right)}{t^{\frac{3}{2}-\beta}} -\frac{C\delta_*\cos\big(\zeta(t,0,x)\big)}{4 t^{\frac{3}{2}-\beta+2\alpha}}+	\frac{C\sin\left(\zeta(t,0,x)\right)}{t^{\frac{3}{2}-\beta+3\alpha}}.
\end{align}
Therefore, in order for \eqref{aim} to be satisfied, it will be enough to prove that, up to increasing $T$, 
\begin{equation}\label{step1_1}
	\frac{C}{t^{\gamma+\mu+\frac{3}{2}}}+ 	\frac{C\sin\left(\zeta(t,0,x)\right)}{t^{\frac{3}{2}-\beta+3\alpha}}\ll \frac{C\zeta(t,0,x)}{t^{\gamma+\mu+\frac{3}{2}-\alpha}}+\frac{C\delta_*\cos\big(\zeta(t,0,x)\big)}{4 t^{\frac{3}{2}-\beta+2\alpha}}
\end{equation}
 for  $t\ge i_\dagger+T$ and $-t^\delta\le x-c_*t\le \frac{3\pi}{2}t^\alpha$.
In fact, we can achieve  \eqref{step1_1}  by distinguishing three cases:
\begin{itemize}
	\item  $-t^{\delta-\alpha}\le \zeta(t,0,x)\le t^{-\alpha}$. Since $\cos(\zeta(t,0,x))$ takes values close to 1, we observe from  $\gamma+\mu-\delta>2\alpha-\beta$ that the last term in the right-hand side of \eqref{step1_1} plays the dominant role, which is obviously greater than the left-hand side as long as $T$ is sufficiently large. 

	\item $t^{-\alpha}\le \zeta(t,0,x)\le \frac{\pi}{4}$. We notice that both terms in right-hand side of \eqref{step1_1} are positive and $\frac{1}{\sqrt{2}}\le\cos(\zeta(t,0,x))< 1$, and the last term in the right-hand side of   \eqref{step1_1} solely can control the left-hand side for $t\ge i_\dagger+T$ if $T$ is sufficiently large. 
	
	\item  $\frac{\pi}{4} \le \zeta(t,0,x)\le \frac{3\pi}{2}$. In this situation, the cosine term may have a negative sign, however one notices from $\gamma+\mu-\alpha<2\alpha-\beta$ that the first term in the right-hand side of \eqref{step1_1} dominates the sign.
\end{itemize}
Consequently,  \eqref{aim} and thus \eqref{step1_i=0} are reached. The region where $\mathcal{V}_1$ is defined has been exhausted.

\vspace{2mm}

\noindent
\textbf{Step 2}. We now look at $t\ge i+T$, $i\in[0,i_\dagger)$ and $ \frac{3\pi}{2}t^\alpha\le  x-c_*(t-i)\le c_*t^* -
\frac{\kappa_2}{\kappa_1} \frac{c_*i}{\sqrt{t+t^*}}+\ell_1\sqrt{t+t^*}$. Since
\begin{equation*}
	\overline v(t,i,x)=\overline\xi(t,i) u(t,i,x),
\end{equation*}
it is trivial to get the conclusion up to increasing $T$, based upon our discussion \eqref{eqn_xi u}-\eqref{eqn_xi u_diffusive}.

\vspace{2mm}

\noindent
\textbf{Step 3}. In this last step, it remains to consider $t\ge i+T$, $i\in[0,i_\dagger)$ and $x\ge 	c_*(t+t^*-i)-
\frac{\kappa_2}{\kappa_1} \frac{c_*i}{\sqrt{t+t^*}}+\ell_1\sqrt{t+t^*}$, where we recall that the function $\overline v$ has the following expression:  
\begin{equation*}
	\overline v(t,i,x)= \overline\xi(t,i)u(t,i,x)+\mathcal{V}_2(t+t^*,i,x),
\end{equation*}
with 
\begin{equation*}
	\mathcal{V}_2(t+t^*,i,x)=\frac{M}{(t+t^*)^2} \mathcal{U}\bigg(	\frac{
		x-c_*(t+t^*-i)+
		\frac{\kappa_2}{\kappa_1} \frac{c_*i}{\sqrt{t+t^*}}
	}{\sqrt{t+t^*}}\bigg).
\end{equation*}
 Roughly speaking,  $\mathcal{V}_2$ shall be counted upon in this situation  so as to ensure $\overline v$ to keep positive outside the diffusive regime. For a shorter notation, we will use  
 \begin{equation*}
 	\eta=\eta(t,i,x):=	\frac{
 		x-c_*(t+t^*-i)+
 		\frac{\kappa_2}{\kappa_1} \frac{c_*i}{\sqrt{t+t^*}}
 	}{\sqrt{t+t^*}}.
 \end{equation*}

\noindent
\textbf{Step 3.1}. Let us first deal with the case that $\eta\ge\ell_2$, i.e. when $t\ge i+T$, $i\in[0,i_\dagger)$ and $x-c_*(t+t^*-i)+
\frac{\kappa_2}{\kappa_1} \frac{c_*i}{\sqrt{t+t^*}}\ge \ell_2\sqrt{t+t^*}$, we have simply $$	\mathcal{V}_2(t+t^*,i,x)=\frac{M}{(t+t^*)^2} \mathcal{U}(\eta)= \frac{M}{(t+t^*)^2} e^{-\kappa_1(\eta-\ell_2)}.$$
 An immediate consequence of \eqref{eqn_xi u} and \eqref{eqn_thm4.3_1} is that, up to increasing $T$,      
\begin{align*}
(\partial_t+\partial_i)\overline v(t,i,x)&=	(\partial_t+\partial_i)\big(\overline \xi(t,i)u(t,i,x)\big)\!+	(\partial_t+\partial_i)\mathcal{V}_2(t+t^*,i,x)\\
&= u(t,i,x)(\partial_t+\partial_i)\overline\xi(t,i)+	(\partial_t+\partial_i)\mathcal{V}_2(t+t^*,i,x)\\
&\ge  \frac{C}{t^{1+\gamma}}u(t,i,x)+\frac{C}{t+t^*}\mathcal{V}_2(t+t^*,i,x)\ge 0 
\end{align*}     
for $t\ge i+T$, $i\in(0,i_\dagger)$ and  $\eta\ge \ell_2$.  Moreover, 
 by recalling \eqref{eqn_thm4.3_2}, we also have that on the one hand, 
  \begin{align*}
 	&\S_0\int_0^{i_\dagger}\omega(i)e^{-\alpha_*c_*i}\K_* * \overline v(t,i,x) \md i
 	\le 	\S_0\int_0^{i_\dagger}\omega(i)e^{-\alpha_*c_*i}\K_* *\left(\overline\xi(t,i) u(t,i,x)+\mathcal{V}_2(t+t^*,i,x)\right) \md i\\
 	&\le  
 	\S_0\int_0^{i_\dagger}\omega(i)e^{-\alpha_*c_*i}\bigg(\Big(\overline\xi(t,0)-\frac{c_*i}{t^{\gamma+\mu}}\Big)\K_* *u(t,i,x)+ \K_* *\mathcal{V}_2(t+t^*,i,x)\bigg)\md i\\
 	&\le \overline\xi(t,0) u(t,0,x)+\S_0\int_0^{i_\dagger}\omega(i)e^{-\alpha_*c_*i}\K_* *\mathcal{V}_2(t+t^*,i,x)\md i\\
 	&\le \overline\xi(t,0) u(t,0,x)
 	+\mathcal{V}_2(t+t^*,0,x)=\overline v(t,0,x)
 \end{align*}
for $t\ge i_\dagger+T$ and $\ell_2\sqrt{t+t^*}\le x-c_*(t+t^*)\le \ell_2\sqrt{t+t^*}+(t+t^*)^{\frac{1}{2}+\epsilon}$, where we have used the positivity of $u$ obtained in Theorem \ref{thm_Dirichlet heat eqn},  and on the other hand,
 \begin{align*}
 &\S_0\int_0^{i_\dagger}\omega(i)e^{-\alpha_*c_*i}\K_* * \overline v(t,i,x) \md i
 \le 	\S_0\int_0^{i_\dagger}\omega(i)e^{-\alpha_*c_*i}\K_* *\left(\overline\xi(t,i) u(t,i,x)+\mathcal{V}_2(t+t^*,i,x)\right) \md i\\
 &\le  
 \S_0\int_0^{i_\dagger}\omega(i)e^{-\alpha_*c_*i}\bigg(\Big(\overline\xi(t,0)-\frac{c_*i}{t^{\gamma+\mu}}\Big)\K_* *u(t,i,x)+ \K_* *\mathcal{V}_2(t+t^*,i,x)\bigg)\md i\\
 &\le \overline\xi(t,0) u(t,0,x)+\S_0\int_0^{i_\dagger}\omega(i)e^{-\alpha_*c_*i}\Big(\frac{c_*i}{t^{\gamma+\mu}}+1\Big)\K_* *\mathcal{V}_2(t+t^*,i,x)\md i\\
 &\le \overline\xi(t,0) u(t,0,x)
 +\Big(\frac{C}{t^{\gamma+\mu}}+1\Big)\Big(1-\frac{C}{t+t^*}\Big)\mathcal{V}_2(t+t^*,0,x)\\
 &\le \overline\xi(t,0) u(t,0,x)
 +\mathcal{V}_2(t+t^*,0,x)=\overline v(t,0,x)
 \end{align*}
for $t\ge i_\dagger+T$ and $x-c_*(t+t^*)\ge\ell_2\sqrt{t+t^*}+(t+t^*)^{\frac{1}{2}+\epsilon}$, 
where we have used  $\mu>1$ and $\gamma>0$ together with the application of Theorem \ref{thm_exp decay beyond diffusive scale} that $|u(t,i,x)|\le \mathcal{V}_2(t+t^*,i,x)$ for all $t\ge i_\dagger+T$, $i\in[0,i_\dagger)$ and $\eta\ge\ell_2+(t+t^*)^{\epsilon}$.

\noindent
\textbf{Step 3.2}. Finally, we focus on the most complicated case that $t\ge i+T$, $i\in[0,i_\dagger)$ and  $\ell_1\sqrt{t+t^*}\le x-c_*(t+t^*-i)+
\frac{\kappa_2}{\kappa_1} \frac{c_*i}{\sqrt{t+t^*}}\le 	\ell_2\sqrt{t+t^*}$, namely when $\ell_1\le \eta\le \ell_2$, where we notice that  $u(t,i,x)$ and $\mathcal{U}(\eta(t,i,x))$ are positive. In this case,
 \begin{align*}
 	\overline v(t,i,x)=\overline \xi(t,i)u(t,i,x)+\mathcal{V}_2(t+t^*,i,x)= \overline \xi(t,i)u(t,i,x)+ \frac{M}{(t+t^*)^2}\chi(\eta(t,i,x))e^{-\kappa_1(\eta(t,i,x)-\ell_2)}.
 \end{align*}
  We first have, up to increasing $T$ if necessary, 
\begin{align*}
	(\partial_t+\partial_i)\Big(\mathcal{V}_2(t+t^*,i,x)\Big)
	&=(\partial_t+\partial_i)\Big( \frac{M}{(t+t^*)^2} \chi(\eta(t,i,x))e^{-\kappa_1(\eta(t,i,x)-\ell_2)}\Big)\\
	&~~~~=(\chi'-\kappa_1\chi)(\eta)(\partial_t+\partial_i)\eta(t,i,x) \frac{M}{(t+t^*)^2}  e^{-\kappa_1(\eta-\ell_2)}- \frac{2M}{(t+t^*)^3} \chi(\eta) e^{-\kappa_1(\eta-\ell_2)}\\
	~~~~=\Bigg((\chi'-\kappa_1\chi)(\eta)&\bigg(\displaystyle
	\frac{\kappa_2 c_*}{\kappa_1(t+t^*)}
	-\frac{\eta}{2(t+t^*)}-\frac{\kappa_2c_*i}{2\kappa_1(t+t^*)^2}\bigg)- \frac{2}{t+t^*}\chi (\eta)\Bigg) \frac{M}{(t+t^*)^2}  e^{-\kappa_1(\eta-\ell_2)}\\
	&~~~~
	\ge - \frac{C}{(t+t^*)^3}
\end{align*} 
for  $t\ge i+T$, $i\in(0,i_\dagger)$ and  $\ell_1\le \eta\le \ell_2$, by noticing that $\chi$ and $\chi'$ are both bounded in $[\ell_1,\ell_2]$. Combining this and \eqref{eqn_xi u}, together with \eqref{u-asympt} as well as \eqref{parameters},  we have that, up to increasing $T$,  
\begin{align*}
	(\partial_t+\partial_i)\overline v(t,i,x)&=	(\partial_t+\partial_i)\big(\overline \xi(t,i)u(t,i,x)\big)\!+	(\partial_t+\partial_i)\mathcal{V}_2(t+t^*,i,x)\\
	&=u(t,i,x)(\partial_t+\partial_i)\overline\xi(t,i) +M	(\partial_t+\partial_i)\Big(\frac{1}{(t+t^*)^2} \mathcal{U}\big(\eta(t,i,x)\big)\Big)\ge \frac{C}{t^{2+\gamma}} - \frac{C}{(t+t^*)^3}>0.
\end{align*}

 On the other hand, based on \eqref{u-asympt} and \eqref{eqn_xi u_diffusive}, we have      
 \begin{equation}
 	\label{E-integ-def}
 \begin{aligned}
	\S_0\int_0^{i_\dagger}&\omega(i)e^{-\alpha_*c_*i}\K_* * \overline v(t,i,x) \md i\\
	&= 	\S_0\int_0^{i_\dagger}\omega(i)e^{-\alpha_*c_*i}\K_* *\left(\overline\xi(t,i) u(t,i,x)+\mathcal{V}_2(t+t^*,i,x)\right) \md i\\
	&=\S_0\int_0^{i_\dagger}\omega(i)e^{-\alpha_*c_*i}\bigg(\overline\xi(t,0)-\frac{c_*i}{t^{\gamma+\mu}}\bigg)\K_* * u(t,i,x) \md i+\S_0\int_0^{i_\dagger}\omega(i)e^{-\alpha_*c_*i}\K_* * \mathcal{V}_2(t+t^*,i,x) \md i\\
	& \le \overline\xi(t,0) u(t,0,x)-\frac{C}{t^{\gamma+\mu+1}}+ \frac{M}{(t+t^*)^2} \underbrace{\S_0\int_0^{i_\dagger}\omega(i)e^{-\alpha_*c_*i}\K_* * \mathcal{U}(\eta(t,i,x))\md i}_{=:\mathcal{E}(t,x)}
\end{aligned}
 \end{equation}
for $t\ge i_\dagger+T$ and $\ell_1\sqrt{t+t^*}\le x-c_*(t+t^*)\le 	\ell_2\sqrt{t+t^*}$.  

Let us now concentrate on $\mathcal{E}$. Set
\begin{align*}
	&\eta_0:=\eta(t,0,x),~~~~~~~~~~ \varpi(t,i,z):=\frac{z- c_*i
	}{\sqrt{t+t^*}}-\frac{\kappa_2c_*i}{\kappa_1(t+t^*)}.
\end{align*}
We notice that
\begin{align*}
		\K_* * \mathcal{U}(\eta(t,i,x)) &= \int_\R\K_*(z) e^{\kappa_1\varpi(t,i,z)} \chi(\eta(t,i,x-z))\md z  ~ e^{-\kappa_1(\eta_0-\ell_2)}\\
		&=\int_\R\K_*(z) e^{\kappa_1\varpi(t,i,z)} \chi(\eta_0-\varpi(t,i,z))\md z  ~ e^{-\kappa_1(\eta_0-\ell_2)}.
\end{align*}
 Noticing that $\chi$ and its derivatives of all orders are all bounded in $[\ell_1,\ell_2]$, the Taylor expansion   yields  that, up to increasing $T$, 
\begin{align*}
	e^{\kappa_1\varpi(t,i,z)}\chi(\eta_0-\varpi(t,i,z))
	&=\Big(1+\kappa_1\varpi+\frac{\kappa_1^2}{2}\varpi^2+O(\varpi^3)\Big)\Big(\chi(\eta_0)-\chi'(\eta_0)\varpi+\frac{\chi''(\eta_0)}{2}\varpi^2+O(\varpi^3)\Big) \\
	&=\chi(\eta_0)+\big(\kappa_1\chi-\chi'\big)(\eta_0)\varpi+\frac{1}{2}\big(\chi''-2\kappa_1\chi'+\kappa_1^2\chi\big)(\eta_0)\varpi^2+O(\varpi^3)\\ 
&\le \chi(\eta_0)+\frac{(z-c_*i)\big(\kappa_1\chi-\chi'\big)(\eta_0)}{\sqrt{t+t^*}}+\frac{\big|\big(\kappa_1\chi-\chi'\big)(\eta_0)\big|\frac{\kappa_2}{\kappa_1}c_*i}{t+t^*}\\
&~~~~~~~~~~~~~~~~~~~~~~~~~~~+\frac{\frac{1}{2}\big|\big(\chi''-2\kappa_1\chi'+\kappa_1^2\chi\big)(\eta_0)\big|(z-c_*i)^2}{t+t^*}+O\big((t+t^*)^{-\frac{3}{2}}\big).
\end{align*}
It follows from \eqref{dr_1}  that up to increasing $T$,
\begin{equation}\label{estimate of E}
 \begin{aligned}
 \mathcal{E}(t,x) 
 	&=  \S_0\int_0^{i_\dagger}\omega(i)e^{-\alpha_*c_*i}\int_\R\K_*(z) e^{\kappa_1\varpi(t,i,z)} \chi(\eta_0-\varpi(t,i,z))\md z \md i  ~ e^{-\kappa_1(\eta_0-\ell_2)}\\
 	&\le \Big(\chi(\eta_0)+\frac{C}{t+t^*}\Big) e^{-\kappa_1(\eta_0-\ell_2)}=\mathcal{U}(\eta_0)+  \frac{C}{t+t^*}
 \end{aligned}
\end{equation}
for $t\ge i_\dagger+T$ and  $\ell_1\sqrt{t+t^*}\le x-c_*(t+t^*)\le \ell_2\sqrt{t+t^*}$. Eventually,  substituting \eqref{estimate of E} into \eqref{E-integ-def} gives that, up to increasing $T$,
 \begin{align*}
	\S_0\int_0^{i_\dagger}\omega(i)e^{-\alpha_*c_*i}\K_* * \overline v(t,i,x) \md i
	&\le \overline\xi(t,0) u(t,0,x)-\frac{C}{t^{\gamma+\mu+1}}+ \frac{M}{(t+t^*)^2}\mathcal{E}(t,x)\\
	&\le  \overline\xi(t,0) u(t,0,x)-\frac{C}{t^{\gamma+\mu+1}}+ \frac{M}{(t+t^*)^2}\mathcal{U}(\eta_0)+ \frac{C}{(t+t^*)^3}\\
		&< \overline\xi(t,0) u(t,0,x)+\mathcal{V}_2(t+t^*,0,x)=\overline v(t,0,x)
\end{align*} 
for $t\ge i_\dagger+T$ and  $\ell_1\sqrt{t+t^*}\le x-c_*(t+t^*)\le \ell_2\sqrt{t+t^*}$, thanks to $\gamma+\mu<\frac{6}{5}$.

\vspace{2mm}

\noindent
\textbf{Conclusion}.
We have checked that the function $\overline v$ defined in \eqref{upper} is indeed a  supersolution to the linear problem \eqref{moving frame-u}  and thus to the nonlinear transport problem \eqref{moving frame-nonlinear} for $t\ge i+i_\dagger+T$, $i\in[0,i_\dagger)$ and $x-c_*(t-i)\ge -t^\delta$ in the sense of Definition \ref{Def_super+sub}.

In addition, we claim that there exists $A_1>1$ such that 
\begin{equation*}
	A_1\overline v(t,0,x)\ge v(t-T,0,x),~~~~~~~~~~t\in[T,i_\dagger+T],~~x\ge X(t,0).
\end{equation*}
Indeed, Theorem \ref{thm_exp decay beyond diffusive scale} indicates that $\mathcal{V}_2(t+t^*,i,x)\ge v(t-T,i,x)$ for $t\ge T$, $i\in[0,i_\dagger)$ and $x-c_*(t+t^*-i)-
\frac{\kappa_2}{\kappa_1} \frac{c_*i}{\sqrt{t+t^*}}\ge \ell_2\sqrt{t+t^*}+(t+t^*)^{\frac{1}{2}+\epsilon}$. This implies particularly that $\mathcal{V}_2(t+t^*,0,x)\ge v(t-T,0,x)$ for $t\in[T,i_\dagger+T]$ and $x-c_*(t+t^*)\ge \ell_2\sqrt{t+t^*}+(t+t^*)^{\frac{1}{2}+\epsilon}$. For $t\in[T,i_\dagger+T]$ and $X(t,0)\le x\le c_*(t+t^*)+ \ell_2\sqrt{t+t^*}+(t+t^*)^{\frac{1}{2}+\epsilon}$, one can choose $A_1>1$  such that $A_1\overline v(t,0,x)\ge v(t-T,0,x)$, since the function $\overline v$, as mentioned ealier, is positive in its domain. Therefore, our claim is achieved.

 Moreover, for  $t\ge  i+T$, $i\in[0,i_\dagger)$ and $x\in[X(t,i)-R,X(t,i)]$, i.e. $-t^\delta-R\le x-c_*(t-i)\le -t^\delta$, we have that $v(t-T,i,x)= \varrho(t-T,i,x)e^{\alpha_* (x-c_*(t-T-i))}\le \varrho(t-T,i,x)e^{\alpha_* c_*T} e^{-\alpha_* t^\delta} 
 <Ct^{-\frac{3}{2}+\delta}\le A_1 \overline v(t,i,x)$,  up to increasing $A_1$.
  It then follows from  Proposition \ref{prop_cp_nonlinear transport} that
\begin{equation}\label{upper-final}
	v(t-T,i,x)\le A_1\overline v(t,i,x)
\end{equation}
for $t\ge i+i_\dagger+T$, $i\in[0,i_\dagger)$ and $x-c_*(t-i)\ge -t^\delta$.

\subsection{The lower barrier}

  For $t\ge i+T$, $i\in[0,i_\dagger)$ and $x\ge X(t,i)-R:=c_*(t-i)+ t^\delta-R$, we define
\begin{equation}
	\label{lower}
	\underline v(t,i,x)=\max\Big(
			\underline\xi(t,i)u(t,i,x)-\mathcal{V}_3(t,i,x)-\mathcal{V}_2(t+t^*,i,x),0\Big),
\end{equation}
where we have defined 
\begin{equation*}
	\label{xi(t,i)_lower}
	~~~~~~~~~~~\underline \xi(t,i)=1+\frac{1+\frac{c_*i}{t^\mu}}{t^{\gamma}},
\end{equation*}
\begin{equation*}
	\mathcal{V}_3(t,i,x)=\frac{1+\frac{c_*i}{t^k}}{t^{\frac{3}{2}-\beta}}\cos\left(\frac{x-c_*(t-i)}{t^\alpha}\right)\mathds{1}_{\big\{x\in\R\big|t^{\delta}-R\le x-c_*(t-i)\le \frac{3\pi}{2}t^\alpha\big\}},
\end{equation*}
and $\mathcal{V}_2(t,i,x)$ was given in \eqref{V_2}.

Since $X(t,i)$ given above is increasing in $t$ but decreasing in $i$, it follows that $X(t,0)\ge X(t-i+s,0)\ge X(t-i+s,s)$ for all $s\in[0,i]$, whence $X(t,0)\ge \max_{s\in[0,i]}X(t-i+s,s)$. We also observe that $\underline v(t,i,x)\equiv0$ (at least) for $x\ge c_*(t+t^*-i)+
\frac{\kappa_2}{\kappa_1} \frac{c_*i}{\sqrt{t+t^*}}+ \ell_2\sqrt{t+t^*}+(t+t^*)^{\frac{1}{2}+\epsilon}$, thanks to Theorem \ref{thm_exp decay beyond diffusive scale}. We denote for notational convenience 
\begin{equation*}
	X_1(t,i):=c_*(t+t^*-i)+
	\frac{\kappa_2}{\kappa_1} \frac{c_*i}{\sqrt{t+t^*}}+ \ell_2\sqrt{t+t^*}+(t+t^*)^{\frac{1}{2}+\epsilon},~~~~~~t\ge i+T,~i\in[0,i_\dagger).
\end{equation*}
In what follows, as done for the supersolution $\overline v$, we shall check that $\underline v$ defined in \eqref{lower} is a   subsolution to \eqref{moving frame-nonlinear} for $t\ge i+i_\dagger+T$, $i\in[0,i_\dagger)$ and $x\ge X(t,i)$ in the sense of Definition \ref{Def_super+sub}, for which it is enough to look at the region $t\ge i+T$, $i\in[0,i_\dagger)$ and $X(t,i)-R\le x\le X_1(t,i)$.

Before going into the details, let us give the following observation for the right-hand side of the second equation in \eqref{moving frame-nonlinear}, based on  the idea from our work \cite{FRZ2023}:
\begin{equation}\label{eqn_lower_i=0}
\begin{aligned}
&e^{\alpha_*(x-c_*t)}\S_0\!\left(\!1\!-\!\exp\!\Big(\!\!-\!\! e^{-\alpha_*(x-c_*t)}\! \int_0^{i_\dagger}\! \omega(i) e^{-\alpha_*c_* i} \K_* *v(t,i,x)\md i\Big)\!\!\right)\!	\\
&~~~\ge \S_0 \!\int_0^{i_\dagger}\! \omega(i) e^{-\alpha_*c_* i} \K_* *v(t,i,x)\md i-\vartheta\S_0 e^{-\alpha_*(x-c_*t)}\! \Big(\int_0^{i_\dagger}\! \omega(i) e^{-\alpha_*c_* i} \K_* *v(t,i,x)\md i\Big)^2\\
&~~~\ge  \S_0 \!\int_0^{i_\dagger}\! \omega(i) e^{-\alpha_*c_* i} \K_* *v(t,i,x)\md i-\vartheta e^{-\alpha_*(x-c_*t)}\! \int_0^{i_\dagger}\! \omega(i) e^{-\alpha_*c_* i} \K_* *v^2(t,i,x)\md i,~~~~t>0,~x\in\R,
\end{aligned}
\end{equation}
where we have noticed that $1-e^{-s}\ge s- \vartheta s^2$ for $s\ge 0$ with some constant $\vartheta>0$, and that
\begin{align*}
	\int_0^{i_\dagger}\! \omega(i) e^{-\alpha_*c_* i} &\K_* *v(t,i,x)\md i=\int_0^{i_\dagger}\int_\R \omega(i) e^{-\alpha_*c_* i} \K_*(x-y)v(t,i,y)\md y\md i\\
	&\le \Big(\int_0^{i_\dagger}\int_\R \omega(i) e^{-\alpha_*c_* i} \K_*(x-y)\md y\md i\Big)^{\frac{1}{2}}\Big(\int_0^{i_\dagger}\int_\R \omega(i) e^{-\alpha_*c_* i} \K_*(x-y)v^2(t,i,y)\md y\md i\Big)^{\frac{1}{2}}\\
	&\le \frac{1}{\sqrt{\S_0}}\Big(\int_0^{i_\dagger}\! \omega(i) e^{-\alpha_*c_* i} \K_* *v^2(t,i,x)\md i \Big)^\frac{1}{2}.
\end{align*}
It then follows from \eqref{eqn_lower_i=0} that when $\underline v$ faces the second equation of  \eqref{moving frame-nonlinear}, it is sufficient to prove that   for $t\ge i_\dagger+T$ and $x\ge X(t,0)$,
\begin{equation}\label{lower-BC}
	\underline v(t,0,x)\le \S_0 \!\int_0^{i_\dagger}\! \omega(i) e^{-\alpha_*c_* i} \K_* *\underline v(t,i,x)\md i-\vartheta e^{-\alpha_*(x-c_*t)}\! \int_0^{i_\dagger}\! \omega(i) e^{-\alpha_*c_* i} \K_* *\underline v^2(t,i,x)\md i.
\end{equation}

 To start with,  we consider $\underline\xi(t,i) u(t,i,x)$  for $t\ge i+T$, $i\in[0,i_\dagger)$ and $X(t,i)-R\le x\le X_1(t,i)$. One finds that, up to increasing $T$,   
\begin{equation}
	\label{eqn_xi u_lower}
	\begin{aligned}
		(\partial_t+\partial_i)\big(\underline \xi(t,i)u(t,i,x)\big)&=u(t,i,x)(\partial_t+\partial_i)\underline\xi(t,i) \\
		&= \bigg(-\!\frac{\gamma\big(1+\frac{c_*i}{t^\mu}\big)}{t^{1+\gamma}}\!-\!\frac{\mu c_* i}{t^{1+\gamma+\mu}}\!+\!\frac{c_*}{t^{\mu+\gamma}}\!\bigg)u(t,i,x)\le-\frac{C}{t^{1+\gamma}}u(t,i,x)<0
	\end{aligned}
\end{equation}
for $t\ge i+T$, $i\in(0,i_\dagger)$ and $X(t,i)-R\le x\le X_1(t,i)$,
due to $\gamma>0$ and $\mu>1$. Moreover, we claim that 
\begin{equation}
	\label{claim_xi u_lower}
	\begin{aligned}
	\underline\xi(t,0) u(t,0,x)&\le \S_0 \!\int_0^{i_\dagger}\! \omega(i) e^{-\alpha_*c_* i}\underline\xi(t,i) \K_* *u(t,i,x)\md i\\
	&~~~~~~~~~~~~~~-\vartheta e^{-\alpha_*(x-c_*t)}\! \int_0^{i_\dagger}\! \omega(i) e^{-\alpha_*c_* i}\underline\xi^2(t,i) \K_* *u^2(t,i,x)\md i
	\end{aligned}
\end{equation}
 for $t\ge i_\dagger+T$ and $X(t,0)\le x\le X_1(t,0)$.
Indeed, one derives from  \eqref{u-asympt} and the boundedness of $\underline\xi$ and $u$ that
\begin{equation}\label{xi u_i=0_lower}
	\begin{aligned}	
		 \S_0 \int_0^{i_\dagger} \omega(i)& e^{-\alpha_*c_* i} \underline\xi(t,i)\K_* *u(t,i,x)\md i-\vartheta e^{-\alpha_*(x-c_*t)} \int_0^{i_\dagger} \omega(i) e^{-\alpha_*c_* i}\underline\xi^2(t,i) \K_* *u^2(t,i,x)\md i\\
		&\ge \underline\xi(t,0) u(t,0,x)+C\S_0\int_0^{i_\dagger}\omega(i)e^{-\alpha_*c_*i}\frac{c_*i}{t^{\gamma+\mu}}\int_\R\K_*(y) \frac{(x-c_*t-y)}{t^{\frac{3}{2}}}\md y \md i-C e^{-\alpha_*(x-c_*t)}\\
		&= \underline\xi(t,0) u(t,0,x)+	\frac{Cu(t,0,x)}{t^{\gamma+\mu}}-\frac{C}{t^{\gamma+\mu+\frac{3}{2}}}-C e^{-\alpha_*(x-c_*t)}
		\ge \underline\xi(t,0) u(t,0,x)
	\end{aligned}
\end{equation}
for $t\ge i_\dagger+T$ and $X(t,0)\le x\le X_1(t,0)$, up to increasing $T$. Our claim \eqref{claim_xi u_lower} is thus achieved. The rest of the computations resemble those concerning the upper barrier, up to adjustments that are described below.
\vskip 2mm

\noindent
{\bf Step 1}. Let us deal with the region that $t\ge i+T$, $i\in[0,i_\dagger)$ and $t^\delta-R\le x-c_*(t-i)\le \frac{3\pi}{2}t^\alpha$, where $\mathcal{V}_3$ is nontrivial. The function $\underline v$ now reads
\begin{equation*}\label{lower v_cos}
	\underline v(t,i,x)=\underline\xi(t,i) u(t,i,x)-\mathcal{V}_3(t,i,x).
\end{equation*}
Analogous to \eqref{upper v_cos}, we still use for simplification  the notation 
\begin{equation*}
	\zeta=\zeta(t,i,x):=\frac{x-c_*(t-i)}{t^\alpha},
\end{equation*}
which  takes values in the interval $[\frac{t^\delta-R}{t^\alpha}, \frac{3\pi}{2}]$. We have
\begin{align*}
	\big(\partial_t+\partial_i\big)\underline v(t,i,x)&=\bigg(\frac{c_*}{t^{\mu+\gamma}}-\frac{\gamma\big(1+\frac{c_*i}{t^\mu}\big)}{t^{1+\gamma}}-\frac{\mu c_* i}{t^{1+\gamma+\mu}}\bigg)u-\frac{\alpha\big(1+\frac{c_*i}{t^k}\big)}{t^{\frac{5}{2}-\beta}}\zeta\sin \zeta\\
	&~~~~~~~~~~+\bigg(\frac{\big(\frac{3}{2}-\beta\big)\big(1+\frac{c_*i}{t^k}\big)}{t^{\frac{5}{2}-\beta}} +\frac{kc_*i}{t^{\frac{5}{2}-\beta+k}}-\frac{c_*}{t^{\frac{3}{2}-\beta+k}}\bigg) \cos\zeta.
\end{align*}

\noindent
\textbf{Step 1.1}. Claim: $\big(\partial_t+\partial_i\big)\underline v(t,i,x)\le 0$ for  $t\ge i+T$, $i\in(0,i_\dagger)$ and $t^\delta-R\le x-c_*(t-i)\le \frac{3\pi}{2}t^\alpha$. Indeed, 
\begin{itemize}
	\item   $\frac{t^\delta-R}{t^\alpha}\le \zeta\le \frac{\pi}{4}$. The conclusion follows easily from the assumption \eqref{parameters} on the parameters, by noticing that both $u$ and cosine term are positive, and
$\displaystyle		\big(\partial_t+\partial_i\big)\underline v(t,i,x)\le -\frac{C}{t^{1+\gamma+\frac{3}{2}-\delta}}-\frac{C}{t^{\frac{3}{2}-\beta+k}}\le 0.
	$
	
	\item $\frac{\pi}{4} \le \zeta\le \frac{3\pi}{2}$. Noticing that $1+\gamma+\beta<1+\frac{8}{25}<\frac{7}{5}<3\alpha<k+\alpha$ and $k<1$, one derives that $\big(\partial_t+\partial_i\big)\underline v(t,i,x)$ is dominated by  
	$\displaystyle	-\frac{\gamma}{t^{1+\gamma}}u(t,i,x)\le  -\frac{C}{t^{1+\gamma+\frac{3}{2}-\alpha}}<0.
	$
\end{itemize}

\noindent
\textbf{Step 1.2}. 
Let us verify \eqref{lower-BC}
for $t\ge i_\dagger+T$ and $t^\delta\le x-c_*t\le \frac{3\pi}{2}t^\alpha$. As a matter of fact,
\begin{align*}
	\text{RHS of \eqref{lower-BC}}
	&\ge  \S_0 \!\int_0^{i_\dagger}\! \omega(i) e^{-\alpha_*c_* i} \K_* *\left(\underline\xi(t,i) u(t,i,x)-\frac{1+\frac{c_*i}{t^k}}{t^{\frac{3}{2}-\beta}}\cos(\zeta(t,i,x))\right)\md i-C e^{-\alpha_*(x-c_*t)}\\
		&\ge   \underline\xi(t,0) u(t,0,x)\!+\!	\S_0\int_0^{i_\dagger}\!\omega(i)e^{-\alpha_*c_*i}   \frac{c_*i}{t^{\gamma+\mu}}\K_* * u(t,i,x)\md i\\
	&~~~~~~~~~~~~~~~~~~~~~~~~~~~~~~~ -\underbrace{S_0\int_0^{i_\dagger}\omega(i)e^{-\alpha_*c_*i}\K_* * \left(\frac{1+\frac{c_*i}{t^k}}{t^{\frac{3}{2}-\beta}}\cos(\zeta(t,i,x))\right) \md i}_{=:\mathcal{Q}(t,x)}-Ce^{-\alpha_*(x-c_*t)}\\
	&\ge \underline\xi(t,0) u(t,0,x)+\frac{C u(t,0,x)}{t^{\gamma+\mu}}-\frac{C}{t^{\gamma+\mu+\frac{3}{2}}} -\mathcal{Q}(t,x)-Ce^{-\alpha_*(x-c_*t)}\\
	&\ge   \underline\xi(t,0) u(t,0,x)+\frac{C \zeta(t,0,x)}{t^{\gamma+\mu+\frac{3}{2}-\alpha}}-\mathcal{Q}(t,x)
\end{align*}
for  $t\ge i_\dagger+T$ and $t^\delta\le x-c_*t\le \frac{3\pi}{2}t^\alpha$, up to increasing $T$, thanks to \eqref{u-asympt}.
Therefore, it suffices to show that for  $t\ge i_\dagger+T$ and $t^\delta\le x-c_*t\le \frac{3\pi}{2}t^\alpha$,
\begin{equation*}
 \underline\xi(t,0) u(t,0,x)+	\frac{C \zeta(t,0,x)}{t^{\gamma+\mu+\frac{3}{2}-\alpha}}-\mathcal{Q}(t,x)\ge  \underline\xi(t,0) u(t,0,x)	-\frac{\cos\big(\zeta(t,0,x)\big)}{t^{\frac{3}{2}-\beta}},
\end{equation*}
namely,
\begin{equation}
	\label{aim1}
	\mathcal{Q}(t,x)\le \frac{\cos\big(\zeta(t,0,x)\big)}{t^{\frac{3}{2}-\beta}}+\frac{C \zeta(t,0,x)}{t^{\gamma+\mu+\frac{3}{2}-\alpha}}.
\end{equation}
We find that $\mathcal{Q}$ has almost the same expression as $\mathcal{D}$ which has appeared in the analysis for the supersolution $\overline v$ with the form \eqref{upper v_cos}. Consequently, it follows from  \eqref{estimate_D} that
\begin{align*}
	\mathcal{Q}(t,x)
	\le  \frac{\cos\left(\zeta(t,0,x)\right)}{t^{\frac{3}{2}-\beta}} -\frac{C\delta_*\cos\big(\zeta(t,0,x)\big)}{4 t^{\frac{3}{2}-\beta+2\alpha}}+	\frac{C}{t^{\frac{3}{2}-\beta+3\alpha}}
\end{align*}
for  $t\ge i_\dagger+T$ and $t^\delta\le x-c_*t\le \frac{3\pi}{2}t^\alpha$, which implies that in order to establish \eqref{aim1}, it is enough to show that 
\begin{equation}
	\label{aim2}
		\frac{C}{t^{\frac{3}{2}-\beta+3\alpha}} \le \frac{C \zeta(t,0,x)}{t^{\gamma+\mu+\frac{3}{2}-\alpha}}  +\frac{C\delta_*\cos\big(\zeta(t,0,x)\big)}{4 t^{\frac{3}{2}-\beta+2\alpha}}
\end{equation}
for  $t\ge i_\dagger+T$ and $t^\delta\le x-c_*t\le \frac{3\pi}{2}t^\alpha$. Indeed, we find that
\begin{itemize}
	
	\item $\displaystyle t^{\delta-\alpha}\le \zeta(t,0,x)\le \frac{\pi}{4}$. We notice that $\displaystyle \frac{1}{\sqrt{2}}\le\cos(\zeta(t,0,x))< 1$, which, along with the positivity of $\zeta(t,0,x)$, implies \eqref{aim2}, up to increasing $T$.
	
	\item  $\displaystyle \frac{\pi}{4} \le \zeta(t,0,x)\le \frac{3\pi}{2}$. Although the cosine term may be negative, it is the first term in the right-hand side of \eqref{aim2} that plays the dominant role due to $\gamma+\mu-\alpha<2\alpha-\beta$, 
	whence \eqref{aim2} is achieved up to increasing $T$.
\end{itemize}
Consequently,   \eqref{aim2}, and thus \eqref{aim1}  are established. This step is complete.
\vskip 2mm

\noindent
{\bf Step 2}. We cosider the region where $t\ge i+T$, $i\in[0,i_\dagger)$ and  $\frac{3\pi}{2}t^\alpha\le x-c_*(t-i)\le c_*t^*-
\frac{\kappa_2}{\kappa_1} \frac{c_*i}{\sqrt{t+t^*}}+	\ell_1\sqrt{t+t^*}$. Here, since $\underline v(t,i,x)=\underline\xi(t,i)u(t,i,x)$, the conclusion immediately follows from \eqref{eqn_xi u_lower} and \eqref{claim_xi u_lower}.
\vskip 2mm

\noindent
{\bf Step 3}. Finally, we focus on the domain $t\ge i+T$, $i\in[0,i_\dagger)$ and  $\ell_1\sqrt{t+t^*}\le x-c_*(t+t^*-i)+
\frac{\kappa_2}{\kappa_1} \frac{c_*i}{\sqrt{t+t^*}}\le 	\ell_2\sqrt{t+t^*} +(t+t^*)^{\frac{1}{2}+\epsilon}$. For convenience, we denote as before
  \begin{equation*}
	\eta=\eta(t,i,x):=	\frac{
		x-c_*(t+t^*-i)+
		\frac{\kappa_2}{\kappa_1} \frac{c_*i}{\sqrt{t+t^*}}
	}{\sqrt{t+t^*}}\in[\ell_1, \ell_2+(t+t^*)^{\epsilon}].
\end{equation*}
We find that  $\underline v$ in this region reads
\begin{align*}
	\underline v(t,i,x)=\max\Big(\underline \xi(t,i)u(t,i,x)-\mathcal{V}_2(t+t^*,i,x), 0\Big).
\end{align*}

\noindent
{\bf Step 3.1}. Clearly, as long as $\underline v$ is nontrivial,  we deduce from \eqref{eqn_xi u_lower} and \eqref{eqn_thm4.3_1} that, 
\begin{align*}
	\big(\partial_t+\partial_i\big)\underline v(t,i,x)= \big(\partial_t+\partial_i\big)\Big( \underline \xi(t,i)u(t,i,x)-\mathcal{V}_2(t+t^*,i,x)\Big)\le -\frac{C}{t^{1+\gamma}}u(t,i,x) - \frac{C}{t+t^*}\mathcal{V}_2(t+t^*,i,x) \le 0
\end{align*}
for $t\ge i+T$, $i\in(0,i_\dagger)$ and $\ell_2\le \eta\le \ell_2+(t+t^*)^{\epsilon}$.  Moreover, 
combining \eqref{eqn_xi u_lower} 
and the calculation in Step 3.2 for the upper barrier gives that, up to increasing $T$, 
 \begin{align*}
 	\big(\partial_t+\partial_i\big)\underline v(t,i,x)&= \big(\partial_t+\partial_i\big)\Big( \underline \xi(t,i)u(t,i,x)-\mathcal{V}_2(t+t^*,i,x)\Big)\\
 	&\le -\frac{C}{t^{1+\gamma}}u(t,i,x) -(\partial_t+\partial_i)\Big(\frac{M}{(t+t^*)^2} \mathcal{U}(\eta(t,i,x))\Big)\le -\frac{C}{t^{2+\gamma}} + \frac{C}{(t+t^*)^3}\le 0
 \end{align*}
for $t\ge i+T$, $i\in(0,i_\dagger)$ and $\ell_1\le \eta\le \ell_2$. 

\noindent
{\bf Step 3.2}.  We now claim that whenever $\underline v$ is nontrivial, we have
\begin{equation}\label{step3_i=0_lower}
	\begin{aligned}
			\S_0 \!\int_0^{i_\dagger}\! \omega(i) e^{-\alpha_*c_* i} \K_* *\underline v(t,i,x)\md i-\vartheta e^{-\alpha_*(x-c_*t)}\! \int_0^{i_\dagger}\! \omega(i) e^{-\alpha_*c_* i} \K_* *\underline v^2(t,i,x)\md i	\ge \underline v(t,0,x)
	\end{aligned}
\end{equation}
for $t\ge i_\dagger+T$ and $\ell_1\sqrt{t+t^*}\le x-c_*(t+t^*)
\le 	\ell_2\sqrt{t+t^*}+(t+t^*)^{\frac{1}{2}+\epsilon}$. In fact, by following the lines as for \eqref{xi u_i=0_lower}, the left-hand side of    \eqref{step3_i=0_lower}, whenever $\underline v$ is nontrivial, can be written as
\begin{align*}
	\text{LHS of \eqref{step3_i=0_lower}}\!&\ge	\S_0 \!\int_0^{i_\dagger}\! \omega(i) e^{-\alpha_*c_* i} \K_* *\Big(\underline\xi(t,i) u(t,i,x)-\mathcal{V}_2(t+t^*,i,x)\Big)\md i-C e^{-\alpha_*(x-c_*t)}\\
	&\ge \underline\xi(t,0)u(t,0,x)+\frac{C u(t,0,x)}{t^{\gamma+\mu}}-\frac{C}{t^{\gamma+\mu+\frac{3}{2}}}-C e^{-\alpha_*(x-c_*t)}-\Big(1-\frac{C}{t}\Big) \mathcal{V}_2(t+t^*,0,x)\\
	& \ge \underline\xi(t,0)u(t,0,x)+\frac{C}{t^{\gamma+\mu+1}}- \mathcal{V}_2(t+t^*,0,x) \ge \underline v(t,0,x)
\end{align*}
for $t\ge i_\dagger+T$ and $\ell_2\sqrt{t+t^*}< x-c_*(t+t^*)
\le 	\ell_2\sqrt{t+t^*}+(t+t^*)^{\frac{1}{2}+\epsilon}$, where we have also used \eqref{eqn_thm4.3_2}; and
\begin{align*}
	\text{LHS of \eqref{step3_i=0_lower}}\!&\ge	\S_0 \!\int_0^{i_\dagger}\! \omega(i) e^{-\alpha_*c_* i} \K_* *\Big(\underline\xi(t,i) u(t,i,x)-\mathcal{V}_2(t+t^*,i,x)\Big)\md i-C e^{-\alpha_*(x-c_*t)}\\
	&=	\S_0 \!\int_0^{i_\dagger}\! \omega(i) e^{-\alpha_*c_* i} \underline\xi(t,i) \K_* * u(t,i,x)\md i-C e^{-\alpha_*(x-c_*t)} \\
	&~~~~~~~~~~~~~~~~~~~~~~~~~~~~~~- \frac{M}{(t+t^*)^2} \underbrace{\S_0 \int_0^{i_\dagger} \omega(i) e^{-\alpha_*c_* i} \K_* *\mathcal{U}(\eta (t,i,x))\md i}_{=\mathcal{E}(t,x)}\\
	&\ge \underline\xi(t,0)u(t,0,x)+\frac{C u(t,0,x)}{t^{\gamma+\mu}}-\frac{C}{t^{\gamma+\mu+\frac{3}{2}}}-\frac{M}{(t+t^*)^2} \Big(\mathcal{U}(\eta_0)+\frac{C}{t+t^*}\Big)\ge \underline v(t,0,x)
\end{align*}
for $t\ge i_\dagger+T$ and $\ell_1\sqrt{t+t^*}\le x-c_*(t+t^*)
\le 	\ell_2\sqrt{t+t^*}$,
where the integral $\mathcal{E}$, with  the same expression as that of \eqref{E-integ-def}, satisfies \eqref{estimate of E}, i.e. $\mathcal{E}(t,x)\le \mathcal{U}(\eta_0)+\frac{C}{t+t^*}$ in this zone.

\vspace{2mm}

\noindent
\textbf{Conclusion}. The function $\underline v$ given in \eqref{lower} is  a subsolution to the nonlinear transport problem \eqref{moving frame-nonlinear} for $t\ge i+i_\dagger+T$, $i\in[0,i_\dagger)$ and $ x-c_*(t-i)\ge t^\delta$, in the sense of Definition \ref{Def_super+sub}.

On the other hand, we recall that  $v(t,0,x)>0$ for $t>0$ and $x\in\R$ whereas $\underline v(t,0,x)$ for each $t\ge T$ is supported on a bounded interval by its definition, then one can find $A_2>0$ very small such that 
\begin{equation*}
	A_2\underline v(t,0,x)\le v(t,0,x),~~~~~~t\in[T,i_\dagger+T],~~x\ge X(t,0).
\end{equation*}
Moreover, we notice that for $t\ge i+T$, $i\in[0,i_\dagger)$ and $x\in[X(t,i)-R,X(t,i)]$,
\begin{equation*}
	\underline v(t,i,x)=\underline \xi(t,i) u(t,i,x)-\mathcal{V}_3(t,i,x)\le C t^{-\frac{3}{2}+\delta}-\frac{1}{2} t^{-\frac{3}{2}+\beta}\le 0\le v(t,i,x).
\end{equation*}
 The comparison principle Proposition \ref{prop_cp_nonlinear transport} then implies that  
 \begin{equation}\label{lower-final}
 	A_2\underline v(t,i,x)\le v(t,i,x)
 \end{equation}
for $t\ge i+i_\dagger+T$, $i\in[0,i_\dagger)$ and $x-c_*(t-i)\ge t^\delta$.

\section{The final argument}
\label{sec6}
In this section, we aim to show the logarithmic correction in the front position for the nonlinear transport problem \eqref{moving frame-nonlinear}. 
The proof of Theorem \ref{thm-1}  is actually an immediate consequence of  Theorem \ref{thm-1'}, thanks to the transformation $\varrho(t,i,x)=\frac{\rho(t,i,x)}{\pi(i)}$ for $(t,i,x)\in\R_+\times[0,i_\dagger)\times\R$. The rest of this section  contributes to the proof of Theorem \ref{thm-1'}.

Thanks to the transformation above, we notice from \eqref{TW_original} that problem \eqref{fkpp-auxi}  admits  traveling front solution $\S_0\chi_c(x-c(t-i))$ if and only if $c\ge c_*$, which is  unique (modulo translation) and decreasing in $\R$. Moreover, it follows from \eqref{normalization of TW} that the leading edge  of the minimal traveling wave profile of \eqref{fkpp-auxi}  has the asymptotics
\begin{equation}
	\label{chi_c*}
	\S_0\chi_{c_*}(z)\approx \S_0 ze^{-\alpha_* z}~~~~\text{as}~z\to+\infty.
\end{equation} 

\vspace{2mm}

\begin{proof}[Proof of Theorem \ref{thm-1'}.]  The proof strongly rests on the  upper and lower barriers constructed in the preceding subsection as well as comparison arguments.

Set for $t\ge 1$,~$i\in[0,i_\dagger)$ and $x\in\R$, 
$\mathcal{V}(t,i,x):=t^\frac{3}{2}v(t,i,x)$. 
Then, problem \eqref{moving frame-nonlinear} can be recast as 
\begin{equation}
	\label{moving frame-nonlinear_V}
	\begin{aligned}
	\!\!\!\!\!\!\!	\left\{\begin{split}
			&\partial_t \mathcal{V}(t,i,x)  + \partial_i \mathcal{V}(t,i,x)-\frac{3}{2t}\mathcal{V}(t,i,x) =0,~~~~~~~~~~~~~~~~~~~~~~~~~~~~~~~~~~~~~~~~ t\ge 1,~~~~~~ i\in(0,i_\dagger),&x\in\R, \\
			&	\mathcal{V}(t,0,x)=t^\frac{3}{2} e^{\alpha_*(x-c_*t)}\S_0\!\left(\!1\!-\!\exp\!\Big(\!-t^{-\frac{3}{2}} e^{-\!\alpha_*(x-c_*t)}\!\! \int_0^{i_\dagger} \omega(i) e^{-\!\alpha_*c_* i} \K_* *\mathcal{V}(t,i,x)\md i\Big)\!\!\right)\!,~ ~~t\ge 1, &x\in\R,\\
			&\mathcal{V}(1,i,x)=  v(1,i,x),~~~~~~~~~~~~~~~~~~~~~~~~~~~~~~~~~~~~~~~~~~~~~~~~~~~~~~~~~~~~~~~~~~~~~~~~~~~~~~ ~~~~ i\in[0,i_\dagger),&x\in\R.
		\end{split}\right.
	\end{aligned}
\end{equation}

Let $T>0$ be fixed as in the construction of upper and lower barriers.
From the comparison of $v$ with the upper and lower barriers \eqref{upper-final} and \eqref{lower-final}, one deduces
\begin{equation}\label{V-comparison}
	t^{\frac{3}{2}}A_2\underline v(t,i,x)\le	\mathcal{V}(t,i,x)\le t^{\frac{3}{2}}A_1\overline v(t+T,i,x)
\end{equation}
for $t\ge i+i_\dagger+T$, $i\in[0,i_\dagger)$ and $x-c_*(t-i)\ge  t^\delta$.

Fix $T_*\ge i_\dagger+T$ large enough and $\theta\in(\frac{4}{25},\frac{1}{4})$, let us now introduce
$$X^\pm(t):=c_*t\pm t^\theta-\frac{3}{2\alpha_*}\ln t,~~~~~t\ge T_*,~$$
such that $X^-(t)<X^+(t-i_\dagger)$ for all $t\ge T_*$.

Define also for  $n=1,2$, 
\begin{equation*}
	\psi_n(t,i,x)=t^{\frac{3}{2}}e^{\alpha_*(x-c_*(t-i))}\S_0\chi_{c_*}\bigg(x-c_*(t-i)+\frac{3}{2\alpha_*}\ln t+\varsigma_n\bigg)
\end{equation*} 
 for $t\ge i+T_*$, $i\in[0,i_\dagger)$ and $X^-(t)-R\le x\le X^+(t)+R$, with parameters $\varsigma_2<\varsigma_1$ to be chosen such that
\begin{equation}
	\label{psi_n-1}
	\psi_1(t,i,x)\le	\mathcal{V}(t,i,x)\le\psi_2(t,i,x)
\end{equation}
for $t\ge i+T_*$, $i\in[0,i_\dagger)$ and $X^+(t-i)\le x\le X^+(t)+R$, and such that
\begin{equation}\label{psi_n-2}
	\psi_1(t,0,x)\le \mathcal{V}(t,0,x)\le \psi_2(t,0,x)
\end{equation}
for $t\in[T_*,i_\dagger+T_*]$ and $X^-(t)\le x\le X^+(t)$. The constraints \eqref{psi_n-1}-\eqref{psi_n-2} are achievable in that, on the one hand,  it follows from \eqref{V-comparison}, the definition of $\overline v$ and $\underline v$ as well as \eqref{u-asympt} that there exist $C_2>C_1>0$ such that 
\begin{equation*}
	C_1t^\theta\le 	t^{\frac{3}{2}}A_2\underline v(t,i,x)\le	\mathcal{V}(t,i,x)\le t^{\frac{3}{2}}A_1\overline v(t+T,i,x)\le C_2t^\theta
\end{equation*}
for $t\ge i+T_*$, $i\in[0,i_\dagger)$ and $X^+(t-i)\le x\le X^+(t)+R$; on the other hand, we derive from \eqref{chi_c*} that
$\psi_n(t,i,x)\approx\S_0 \big(x-c_*(t-i)+\frac{3}{2\alpha_*}\ln t+\varsigma_n\big)e^{-\alpha_*\varsigma_n}$ for $t\ge i+T_*$, $i\in[0,i_\dagger)$ and $X^+(t-i)\le x\le X^+(t)+R$; thus combining with the monotonicity property of $\chi_{c_*}$ in $\R$ altogether will give \eqref{psi_n-1}. In addition, in view of the strict positivity of $\psi_n(t,0,x)$  and $\mathcal{V}(t,0,x)$ for $t\in[T_*,i_\dagger+T_*]$ and $x\in[X^-(t), X^+(t)]$, it is easy to get \eqref{psi_n-2} up to enlarging $\varsigma_1$ and decreasing $\varsigma_2$.

To proceed with our proof, we need the following result for which the proof will be postponed to the end of this section.
\begin{lem}\label{lem4.1}
	\begin{equation*}\label{claim_psi_n}
	\limsup_{t\to+\infty}\big(\psi_1(t,i,x)-\mathcal{V}(t,i,x)\big)\le 0\le \liminf_{t\to+\infty}\big(\psi_2(t,i,x)-\mathcal{V}(t,i,x)\big),
	\end{equation*}
 uniformly in $i\in[0,i_\dagger)$ and $X^-(t)\le x\le X^+(t-i)$.
\end{lem}

Noticing that $$\varrho(t,i,x)=t^{-\frac{3}{2}}e^{-\alpha_*(x-c_*(t-i))} \mathcal{V}(t,i,x),~~~~~t\ge 1,~i\in[0,i_\dagger),~x\in\R,$$ 
it follows from Lemma \ref{lem4.1} that
\begin{align*}
	\limsup_{t\to+\infty}\bigg(\S_0\chi_{c_*}\Big(x&-c_*(t-i)+\frac{3}{2\alpha_*}\ln t+\varsigma_1\Big)- \varrho(t,i,x)\bigg)\le 0\\
&~~~~~~~~	\le \liminf_{t\to+\infty}\bigg(\S_0\chi_{c_*}\Big(x-c_*(t-i)+\frac{3}{2\alpha_*}\ln t+\varsigma_2\Big)- \varrho(t,i,x) \bigg)
\end{align*}
uniformly in $i\in[0,i_\dagger)$ and $1\le x-c_*(t-i)+\frac{3}{2\alpha_*}\ln t\le (t-i_\dagger)^\theta$. Since $i\in[0,i_\dagger)$ is bounded, one derives that for any $m\in (0,\S_0\rho^*)$,
\begin{equation*}
	E_m(t)=x-c_*t+\frac{3}{2\alpha_*}\ln t+O_{t\to+\infty}(1).
\end{equation*}
 This compltes the proof of Theorem \ref{thm-1'}.
\end{proof}

Finally, we close this section by proving Lemma \ref{lem4.1}.
\begin{proof}[Proof of Lemma \ref{lem4.1}]
	  We  only outline the proof for the first inequality, and the second one can be handled analogously.
	  
First of all, substituting $\psi_n(t,i,x)$ ($n=1,2$) into the first equation of \eqref{moving frame-nonlinear_V}, it follows that 
\begin{equation*}
	\left|\big(\partial_t  + \partial_i\big) \psi_n(t,i,x)-\frac{3}{2t}\psi_n(t,i,x)\right|=\left|\frac{3 }{2\alpha_*t}t^{\frac{3}{2}}e^{\alpha_*(x-c_*(t-i))}\S_0\chi_{c_*}'\bigg(x-c_*(t-i)+\frac{3}{2\alpha_*}\ln t+\varsigma_n\bigg)\right|\le Ct^{\theta-1}
\end{equation*}
for $t\ge i+T_*$, $i\in(0,i_\dagger)$ and $X^-(t)\le x\le X^+(t-i)$.
Moreover, noticing that
\begin{align*}
	\K_* *\psi_n(t,i,x)&=\int_\R \K(x-y)e^{\alpha_*(x-y)}t^{\frac{3}{2}}e^{\alpha_*(y-c_*(t-i))}\S_0\chi_{c_*}\bigg(y-c_*(t-i)+\frac{3}{2\alpha_*}\ln t+\varsigma_n\bigg)\md y\\
	&=t^{\frac{3}{2}}e^{\alpha_*(x-c_*t)}e^{\alpha_*c_*i}\S_0  \int_\R \K(x-y)\chi_{c_*}\bigg(y-c_*(t-i)+\frac{3}{2\alpha_*}\ln t+\varsigma_n\bigg)\md y\\
	&=t^{\frac{3}{2}}e^{\alpha_*(x-c_*t)}e^{\alpha_*c_*i}\S_0 \K*\chi_{c_*}\bigg(x-c_*(t-i)+\frac{3}{2\alpha_*}\ln t+\varsigma_n\bigg),
\end{align*}
together with  the equation  \eqref{chi_TW} satisfied by $\chi_{c_*}$:
\begin{equation*}
	\chi_{c}(z)=1-\exp\Big(-\S_0\int_0^\infty \omega(i) \K*\chi_{c}(z+ci)\md i\Big),~~~z\in\R,
\end{equation*}
 we have that
\begin{align*}
	t^\frac{3}{2} e^{\alpha_*(x-c_*t)}&\S_0\left(\!1-\exp\!\Big(\!-t^{-\frac{3}{2}} e^{-\!\alpha_*(x-c_*t)} \int_0^{i_\dagger} \omega(i) e^{-\!\alpha_*c_* i} \K_* *\psi_n(t,i,x)\md i\Big)\!\right)\\
	&=	t^\frac{3}{2} e^{\alpha_*(x-c_*t)}\S_0\left(\!1-\exp\!\left(- \S_0\int_0^{i_\dagger} \omega(i)  \K *\chi_{c_*}\bigg(x-c_*(t-i)+\frac{3}{2\alpha_*}\ln t+\varsigma_n\bigg)\md i\right)\!\right)\\
	&=	t^\frac{3}{2} e^{\alpha_*(x-c_*t)}\S_0\chi_{c_*}\bigg(x-c_*t+\frac{3}{2\alpha_*}\ln t+\varsigma_n\bigg)\\
	&=\psi_n(t,0,x)
\end{align*}  
 for $t\ge i_\dagger+T_*$ and $X^-(t)\le x\le X^+(t)$. Namely, $\psi_n$ ($n=1,2$) satisfies  the second equation of \eqref{moving frame-nonlinear_V} for $t\ge i_\dagger+T_*$ and $X^-(t)\le x\le X^+(t)$.

Set now $\mathcal{Z}(t,i,x):=\big(\psi_1-\mathcal{V}\big)^+(t,i,x)$ for $t\ge i+T_*$, $i\in[0,i_\dagger)$ and $X^-(t)-R\le x\le X^+(t)+R$. It then follows that the function $\mathcal{Z}$ satisfies
\begin{equation*}
	\label{Z}
	\begin{aligned}
		\!	\left\{\begin{split}
			&\big(\partial_t  + \partial_i\big) \mathcal{Z}(t,i,x)-\frac{3}{2t}\mathcal{Z}(t,i,x) \le Ct^{\theta-1},~~~~~~~~~~ t\ge i+T_*,~ i\in(0,i_\dagger),~X^-(t)\le x\le X^+(t-i), \\
			&	\mathcal{Z}(t,0,x)\!\le \S_0 \int_0^{i_\dagger} \omega(i) e^{-\alpha_*c_* i} \K_* *\mathcal{Z}(t,i,x)\md i, ~~~~~~~~~~~~~~~~~~~~~t\ge i_\dagger+T_*,~X^-(t)\le x \le X^+(t),
		\end{split}\right.
	\end{aligned}
\end{equation*}
where the second equation follows from 
\begin{align*}
	\mathcal{Z}(t,0,x)&=\big(\psi_1-\mathcal{V}\big)^+(t,0,x)\\
	&=t^\frac{3}{2} e^{\alpha_*(x-c_*t)}\S_0\Bigg(1-\exp\!\left(\!-t^{-\frac{3}{2}} e^{-\!\alpha_*(x-c_*t)} \int_0^{i_\dagger} \omega(i) e^{-\!\alpha_*c_* i} \K_* *\psi_1(t,i,x)\md i\right)\\
	&~~~~~~~~~~~~~~~~~~~~~~~~~~~~~~-1+\exp\!\left(\!-t^{-\frac{3}{2}} e^{-\!\alpha_*(x-c_*t)} \int_0^{i_\dagger} \omega(i) e^{-\!\alpha_*c_* i} \K_* *\mathcal{V}(t,i,x)\md i\right)\Bigg)^+\\
	&=	t^\frac{3}{2} e^{\alpha_*(x-c_*t)}\S_0\exp\!\left(\!-t^{-\frac{3}{2}} e^{-\!\alpha_*(x-c_*t)} \int_0^{i_\dagger} \omega(i) e^{-\!\alpha_*c_* i} \K_* *\mathcal{V}(t,i,x)\md i\right)\Bigg(1-\\
	&~~~~~~~~~~~~~~~~~~~~~~~~~~~~~~ \exp\!\left(\!-t^{-\frac{3}{2}} e^{-\!\alpha_*(x-c_*t)} \int_0^{i_\dagger} \omega(i) e^{-\!\alpha_*c_* i}\K_* *\big(\psi_1-\mathcal{V}\big)(t,i,x)\md i\right)\Bigg)^+\\
	&\le \S_0 \left(\int_0^{i_\dagger} \omega(i) e^{-\!\alpha_*c_* i}\K_* *\big(\psi_1-\mathcal{V}\big)(t,i,x)\md i\right)^+\\
		&\le \S_0 \int_0^{i_\dagger} \omega(i) e^{-\!\alpha_*c_* i}\K_* *\big(\psi_1-\mathcal{V}\big)^+(t,i,x)\md i,~~~~~~~~~~~~t\ge i_\dagger+T_*,~~X^-(t)\le x\le X^+(t).
\end{align*}
We also observe that $\mathcal{Z}(t,0,x)=0$ for $t\in[T_*, i_\dagger+T_*]$ and $x\in[X^-(t),X^+(t)]$, thanks to \eqref{psi_n-2}. Moreover, it follows from \eqref{psi_n-1} that $\mathcal{Z}(t,i,x)=0$ for $t\ge i+T_*$, $i\in[0,i_\dagger)$ and $x\in[X^+(t-i), X^+(t)+R]$, whereas  $\mathcal{Z}(t,i,x)\le \psi_1(t,i,x)\le Ce^{-\alpha_*t^\theta}$ for $t\ge i+T_*$, $i\in[0,i_\dagger)$ and $x\in[X^-(t)-R, X^-(t)]$.  

To reach the conclusion, it suffices to show that
\begin{equation*}
	\label{Z-1}
	\mathcal{Z}(t,i,x)\to 0~~~~\text{as}~t\to+\infty,
\end{equation*}
uniformly for $i\in[0,i_\dagger)$ and  $X^-(t)\le x\le X^+(t-i)$. To do so, we construct 
\begin{equation*}
	\overline{\mathcal{Z}}(t,i,x)=\frac{1+\frac{c_*i}{t^r}}{t^\lambda}\cos\bigg(\frac{x-c_*(t-i)}{t^\sigma}\bigg)
\end{equation*}
for $t\ge i+T_*$, $i\in[0,i_\dagger)$ and $X^-(t)-R\le x\le X^+(t)+R$, in which we  choose 
$\label{para}
0<\lambda<\frac{1}{12}<r<\frac{1}{4}<\sigma<\frac{1}{3}$.
We shall check that $\overline{\mathcal{Z}}$ is an upper barrier of $\mathcal{Z}$ for $t\ge i+i_\dagger+T_*$, $i\in[0,i_\dagger)$ and $X^-(t)\le x\le X^+(t-i)$. For convenience, let us define 
\begin{equation*}
~~~~~~~~~~~~~~~~~~~~~~~	\xi(t,i,x)=\frac{x-c_*(t-i)}{t^\sigma},~~~~~t\ge i+T_*, ~i\in[0,i_\dagger),~x\in[X^-(t)-R,X^+(t)+R].
\end{equation*} 
Noticing that $\theta\in(\frac{4}{25},\frac{1}{4})$ and $0<r+\lambda<\frac{1}{3}<1-\theta$, we infer that 
\begin{align*}
	\big(\partial_t+&\partial_i\big)\overline{\mathcal{Z}}(t,i,x)-\frac{3}{2t}\overline{\mathcal{Z}}(t,i,x)\\
	=&\bigg(\frac{c_*}{t^{r+\lambda}}-\frac{rc_*i}{t^{1+r+\lambda}}-\frac{(\lambda+\frac{3}{2})\big(1+\frac{c_*i}{t^r}\big)}{t^{1+\lambda}}\bigg)\cos\big(\xi(t,i,x)\big)+\frac{\sigma\big(1+\frac{c_*i}{t^r}\big)}{t^{1+\lambda}}\xi(t,i,x)\sin\big(\xi(t,i,x)\big)
	\ge \frac{C}{t^{r+\lambda}}\gg   \frac{C}{t^{1-\theta}}
\end{align*}
for all $t\ge i+T_*$, $i\in(0,i_\dagger)$ and $X^-(t)\le x\le X^+(t-i)$.

On the other hand, we follow the same lines as in the estimation of  $\mathcal{D}(t,x)$ in \eqref{eqn-D} and  derive that   
\begin{align*}
	 \S_0 \!\int_0^{i_\dagger} \!\omega(i) e^{-\!\alpha_*c_* i} \K_* *\overline{\mathcal{Z}}(t,i,x)\md i&= \S_0\! \int_0^{i_\dagger}\! \omega(i) e^{-\!\alpha_*c_* i} \frac{1+\frac{c_*i}{t^r}}{t^\lambda} \!\int_\R\!\K_*(y) \cos\bigg(\frac{x-c_*t-(y-c_*i)}{t^\sigma}\bigg)\md y\md i\\
	 &\le \overline{\mathcal{Z}}(t,0,x)-\frac{C\delta_*\cos\big(\xi(t,0,x)\big)}{4t^{\lambda+2\sigma}}+\frac{C\sin\big(\xi(t,0,x)\big)}{t^{\lambda+3\sigma}}\le  \overline{\mathcal{Z}}(t,0,x)
\end{align*}
for $t\ge i_\dagger+T_*$ and $X^-(t)\le x\le X^+(t)$.

In addition, we notice that $\overline{\mathcal{Z}}(t,i,x)\ge Ct^{-\lambda}\gg Ce^{-\alpha_*t^\theta}\ge \mathcal{Z}(t,i,x)$ for $t\ge i+T_*$, $i\in[0,i_\dagger)$ and  $x\in[X^-(t)-R, X^-(t)]\cup[X^+(t-i),X^+(t)+R]$. For $t\in[T_*,i_\dagger+T_*]$, there holds $\overline{\mathcal{Z}}(t,0,x)>0=\mathcal{Z}(t,0,x)$ for all $x\in[X^-(t), X^+(t)]$. We then conclude from the maximum principle Proposition \ref{prop_cp_linear transport_2} that $\mathcal{Z}(t,i,x)\le \overline{\mathcal{Z}}(t,i,x)$ 
for $t\ge i+i_\dagger+T_*$, $i\in[0,i_\dagger)$ and $X^-(t)\le x\le X^+(t-i)$. Consequently,
\begin{equation*}
	\psi_1-\mathcal{V}(t,i,x)\le \overline{\mathcal{Z}}(t,i,x) \to 0~~~~\text{as}~t\to+\infty,
\end{equation*}
uniformly for $i\in[0,i_\dagger)$ and $x\in[X^-(t), X^+(t-i)]$. That is, $\limsup_{t\to+\infty}\big(\psi_1(t,i,x)- \mathcal{V}(t,i,x)\big)\le 0$ uniformly for $i\in[0,i_\dagger)$ and $x\in[X^-(t), X^+(t-i)]$. This completes the proof. \end{proof}

\bigskip
\noindent{\bf Data availability statement.} No new data were created or analyzed in this study.

\appendix

\section{Appendix}

\begin{proof}[Proof of Proposition \ref{prop_cp_R}]   Let $\overline \varrho$ and $\underline \varrho$ be
	given as in the statement. Set $z(t,i,x):=\underline \varrho(t,i,x)-\overline \varrho(t,i,x)$ for  $(t,i,x)\in\R_+\times[0,i_\dagger)\times\R$. It follows that $z$ satisfies
	\begin{equation}
		\label{z-R}
		\begin{aligned}
			\left\{\begin{split}
				&\partial_t z(t,i,x)  + \partial_i z(t,i,x) \le0,~~~~~~~~~~~~~~~~~~~~t>0, ~i\in(0,i_\dagger), ~ x\in\R, \\
				&	z(t,0,x)\le\S_0 \int_0^\infty \omega(i)  \K  *z(t,i,x)\md i,~~~~~~~~~~~~~~~~~~~~~ t> 0, ~ x\in\R,\\
				&z(0,i,x)\le 0,~~~~~~~~~~~~~~~~~~~~~~~~~~~~~~~~~~~~~~~~~~~~~~~~ i\in[0,i_\dagger), ~ x\in\R.
			\end{split}\right.
		\end{aligned}
	\end{equation}
	Multiplying the transport problem \eqref{z-R} by $\mathds{1}_{\{z>0\}}$, together with the notation  $z^+=\max(z,0)$, yields that
	\begin{equation}
		\label{cp-z^+}
		\begin{aligned}
			\left\{\begin{split}
				&	\partial_t z^+(t,i,x)  + \partial_i z^+(t,i,x) \le 0,~~~~~~~ ~~~~~~~~~~~~~~~~~t> 0, ~i\in(0,i_\dagger), ~x\in\R, \\
				&	z^+(t,0,x)\le \S_0 \int_0^\infty \omega(i)  \K  *z^+(t,i,x)\md i,  ~~~~~~~~~~~~~~~~~~~~~~~~~t> 0,   ~x\in\R,\\
				&	z^+(0,i,x)\equiv  0,~~~~~~~~~~~~~~~~~~~~~~~~~~~~~~~~~~~~~~~~~~~~~~~~~~~~~~  i\in[0,i_\dagger), ~ x\in\R.
			\end{split}\right.
		\end{aligned}
	\end{equation}

	Integrating the  inequality in \eqref{cp-z^+}  over $i\in[0,i_\dagger)$, altogether with the second inequality of \eqref{cp-z^+}, gives that
	\begin{equation*}
		\int_0^{i_\dagger} \partial_t z^+(t,i,x)\md i\le z^+(t,0,x)\le \S_0 \int_0^{i_\dagger} \omega(i)  \K  *z^+(t,i,x)\md i,~~~t> 0,~x\in\R,
	\end{equation*}
	where we has used the fact that the nontrivial part of the integration term in \eqref{cp-z^+} is actually  for $i\in(0,i_\dagger)$.
	We  intergrate the above inequality with respect to time over $(0,t)$ and derive that for $t> 0$ and $x\in\R$, 
	\begin{equation*}
		\int_0^{i_\dagger}  z^+(t,i,x)\md i=	\int_0^{i_\dagger}  z^+(t,i,x) - z^+(0,i,x)\md i\le \S_0\int_0^t \int_0^{i_\dagger} \omega(i)  \K  *z^+(\tau,i,x)\md i\md \tau.
	\end{equation*}
	Hence, for $t> 0$,
	\begin{align*}
		\sup_{x\in\R}	\int_0^{i_\dagger}  z^+(t,i,x)\md i&\le 	\sup_{x\in\R}\S_0\int_0^t \int_0^{i_\dagger} \omega(i)  \K  *z^+(\tau,i,x)\md i\md \tau\\
		& \le C\int_0^t 	\sup_{x\in\R}\int_0^{i_\dagger}  \K  *z^+(\tau,i,x)\md i\md \tau\\
		&\le C \int_0^t \sup_{x\in\R}\int_0^{i_\dagger}  z^+(\tau,i,x)\md i\md \tau.
	\end{align*}
	The Gronwall inequality then implies that  $	\sup_{x\in\R}	\int_0^{i_\dagger}  z^+(t,i,x)\md i\equiv 0$ for $t> 0$, whence $z^+(t,i,x)\equiv 0$ for  $(t,i,x)\in\R_+\times[0,i_\dagger)\times\R$. This implies that $\underline \varrho(t,i,x)\le \overline \varrho(t,i,x)$ for  $(t,i,x)\in\R_+\times[0,i_\dagger)\times\R$.

	Let us prove the strict inequality, for which we shall argue with \eqref{z-R}. Assume that 
	there exists $\tau_0>0$ such that
	\begin{equation*}
		[0,\tau_0]\subset\text{spt}(\tau)\cap\{i\in[0,i_\dagger)|z(0,i,x)<0~\text{for some}~x\in\R\}.
	\end{equation*} 
 We notice first from the method of characteristics that
\begin{equation}
	\label{z-t>i}
	\begin{aligned}
		z(t,i,x)\le 	\begin{cases}
			z(0,i-t,x), &t\le i,~i\in[0,i_\dagger),~x\in\R,\\
			z(t-i,0,x)\le \S_0 \int_0^\infty\! \omega(s)  \K  *z(t-i,s,x)\md s, &t>i,~i\in[0,i_\dagger), ~x\in\R.
		\end{cases}
	\end{aligned} 
\end{equation}
To conclude, it is therefore sufficient to prove that $z(t,0,x)<0$ for $t>0$ and $x\in\R$, thanks to \eqref{z-t>i}.
  
 Assume towards contradiction that there exist $t_0>0$ and $x_0\in\R$ such that $z(t_0,0,x_0)=0$. Then, based on the boundary inequality in \eqref{z-R} as well as the nonpositivity of $z$, it follows that
 \begin{equation*}
 		0=z(t_0,0,x_0)\le\S_0 \int_0^\infty \omega(i)  \K  *z(t_0,i,x_0)\md i\le 0,
 \end{equation*}
  which implies that
  \begin{equation*}
  	\K  *z(t_0,i,x_0)=0,~~~~~~i\in(0,\tau_0).
  \end{equation*}
Therefore, we have $z(t_0,\cdot,\cdot)=0$ in $(0,\tau_0)\times(x_0-R,x_0+R)$. By continuity  in $i$ of the function $z$, we have particularly $z(t_0,0,\cdot)=0$ in $(x_0-R,x_0+R)$. Then, it follows again from the boundary inequality in \eqref{z-R} that $\K_**z(t_0,i,y)=0$ for $i\in(0,\tau_0)$ and $y\in(x_0-R,x_0+R)$. This implies that $z(t_0,\cdot,\cdot)=0$ in $(0,\tau_0)\times(x_0-2R,x_0+2R)$. Due to the continuity, we infer that $z(t_0,\cdot,\cdot)=0$ in $[0,\tau_0]\times(x_0-2R,x_0+2R)$. By repeating the procedures as above and by continuity in $x$, we eventually find that the space set such that $z(t_0,i,\cdot)=0$ with $i\in[0,\tau_0]$ is nontrivial, and  both open and closed. Therefore, $z(t_0,\cdot,\cdot)=0$ in $[0,\tau_0]\times\R$.

Let us now divide into two cases: either $t_0\le\tau_0$, or $t_0>\tau_0$. For the former case, we derive from \eqref{z-t>i} that
\begin{equation*}
	0=z(t_0,i,x)\le z(0,i-t_0,x)\le 0~~~~\text{for}~t_0\le i\le\tau_0 ~\text{and}~x\in\R.
\end{equation*}
Namely, $z(0,i,x)=0$ for $0\le i\le \tau_0-t_0$ and $x\in\R$. This contradicts $z(0,0,\cdot)\not\equiv 0$ in $\R$. Therefore, the case of $t_0\le \tau_0$ is ruled out. Assume now that $t_0>\tau_0$, we have 
\begin{equation*}
	0=z(t_0,i,x)\le z(t_0-i,0,x)\le 0~~~~\text{for}~0\le i\le\tau_0(<t_0) ~\text{and}~x\in\R.
\end{equation*}
That is, $z(t,0,x)=0$ for $t\in[t_0-\tau_0,t_0]$ and $x\in\R$. In particular, we can choose $t_1:=t_0-\tau_0\in(0,t_0)$ such that $z(t_1,0,\cdot)=0$ in $\R$. We argue as above by distinguishing two cases: either $t_1\le \tau_0$, or $t_1>\tau_0$. Again, for the former case, we will find a contradiction through discussing  $z(t_1,i,x)$ for $t_1\le i\le \tau_0$ and $x\in\R$. For the latter, we can further choose $t_2:=t_1-\tau_0=t_0-2\tau_0\in(0,t_1)$ such that $z(t_2,0,\cdot)=0$ in $\R$. After finite steps, we will get $t_k:=t_{k-1}-\tau_0=t_0-k\tau_0\in(0,\tau_0]$  for some $k\in\N$, such that
\begin{equation*}
	0=z(t_k,i,x)\le z(0,i-t_k,x)\le 0~~~~\text{for}~t_k\le i\le\tau_0 ~\text{and}~x\in\R.
\end{equation*}
Thus, $z(0,i,x)=0$ for $0\le i\le \tau_0-t_k$ and $x\in\R$, contradicting $z(0,0,\cdot)\not\equiv 0$ in $\R$. Consequently, we also exclude the case of $t_0>\tau_0$. This completes the proof.
\end{proof}

\end{document}